\documentclass[10pt,a4paper]{article}

\usepackage{amssymb,latexsym,amsmath,enumerate,verbatim,amsfonts,amsthm}
\usepackage{color,bm} 
\usepackage{cite}
\usepackage{mathtools}
\usepackage{lipsum}
\usepackage{dsfont} 
\usepackage[colorinlistoftodos,prependcaption,textsize=footnotesize]{todonotes}
\usepackage{hyperref}
\hypersetup{
       colorlinks,
       linkcolor={red!50!black},
       citecolor={blue!50!black},
       urlcolor={blue!80!black}
}
\usepackage{dsfont} 
\usepackage{lipsum}
\usepackage[shortlabels]{enumitem}

\usepackage{xcolor}
\usepackage{tikz}
\usetikzlibrary{calc}
\usetikzlibrary{decorations.pathmorphing}

\newtheorem{theorem}{Theorem}[section]
\newtheorem{lemma}{Lemma}[section]
\newtheorem{corollary}{Corollary}[section]
\newtheorem{proposition}{Proposition}[section]

\theoremstyle{definition}
\newtheorem{definition}{Definition}[section]
\theoremstyle{definition}
\newtheorem{example}{Example}[section]
\theoremstyle{definition}
\newtheorem{remark}{Remark}[section]

\newcommand{\RP}{\ensuremath{\mathbb{R}_+}}
\newcommand{\RPP}{\ensuremath{\mathbb{R}_{++}}}

\newcommand{\bb}{\ensuremath{\mathbf b}}
\newcommand{\cc}{\ensuremath{\mathbf c}}
\newcommand{\sbf}{\ensuremath{\mathbf s}}
\newcommand{\vv}{\ensuremath{\mathbf v}}
\newcommand{\yy}{\ensuremath{\mathbf y}}
\newcommand{\zz}{\ensuremath{\mathbf z}}

\def\argmin{\mathop{\rm arg\,min}}

\def\stdCone{{\cal K}}
\def\stdFace{{\cal F}}
\newcommand{\stdAffine}{ \mathcal{V}}

\def\bd{\partial}
\def\dist{\mathop{\rm dist}}

\def\bv{\overline{v}}

\newcommand{\oneFRF}{{one-step facial residual function}}
\newcommand{\oneFRFs}{{$\mathds{1}$-FRF}}
\newcommand{\OneFRF}{{One-step facial residual function}}

\allowdisplaybreaks[4]

\newcommand{\ROC}[1]{{\mathcal{Q}_{r}^{#1}}}

\newcommand{\POC}[2]{{\mathcal{K}_{#1}^{#2}}}
\newcommand{\pK}{\POC{p}{n+1}}
\newcommand{\qK}{\POC{q}{n+1}}

\def\ttx{\bar{x}}
\def\ttz{\bar{z}}
\def\ttf{\bar{f}}
\def\ttv{\bar{v}}

\def\Fhatz{{{\cal F}_{z}}}

\newcommand{\norm}[1]{\|#1\|}
\newcommand{\stdMap}{ {\mathcal{A}}}
\newcommand{\spanVec}{\ensuremath{\mathrm{span}\,}}
\newcommand{\stdSpace}{ \ensuremath{\mathcal{L}}}

\newcommand{\ambSpace}{\ensuremath{\mathcal{E}}}
\newcommand{\feas}{\ensuremath{\mathcal{D}}}

\newcommand{\reInt}{\ensuremath{\mathrm{ri}\,}}

\newcommand{\face}{\mathrel{\unlhd}}
\newcommand{\inProd}[2]{\langle #1 , #2 \rangle }
\newcommand{\comp}{\diamondsuit}
\newcommand{\RR}{\ensuremath{\mathbb R}}
\newcommand{\dpp}{d_{\text{PPS}}}

\newcommand{\fraks}{\ensuremath{\mathfrak{s}}}

\newcommand{\frakg}{\ensuremath{\mathfrak{g}}}

\newcommand{\distP}{\ell _{\text{poly}}}

\numberwithin{equation}{section}

\definecolor{darkgray}{rgb}{0.2,0.2,0.2}
\definecolor{wheat}{rgb}{0.5,0.5,0}
\definecolor{aqua}{rgb}{0,0.4,.6}
\definecolor{forest}{rgb}{0,0.4,0}
\definecolor{darkbrown}{rgb}{0.6,0.2,0.2}

\title{Optimal error bounds in the absence of constraint qualifications with applications to the p-cones and beyond}

\begin{document}	
\author{
	Scott B.\ Lindstrom\thanks{
		Centre for Optimisation and Decision Science, Curtin University, Western Australia.
		E-mail: \href{scott.lindstrom@curtin.edu.au}{scott.lindstrom@curtin.edu.au}.}
	\and
	Bruno F. Louren\c{c}o\thanks{Department of Fundamental Statistical Mathematics, Institute of Statistical Mathematics, Japan.
		This author was supported partly by the JSPS Grant-in-Aid for Early-Career Scientists  19K20217, 23K16844 and the Grant-in-Aid for Scientific Research (B)21H03398.
		Email: \href{bruno@ism.ac.jp}{bruno@ism.ac.jp}.}
	\and
	Ting Kei Pong\thanks{
		Department of Applied Mathematics, the Hong Kong Polytechnic University, Hong Kong, People's Republic of China.
		This author was supported partly by Hong Kong Research Grants Council PolyU153000/20p.
		E-mail: \href{tk.pong@polyu.edu.hk}{tk.pong@polyu.edu.hk}.
	}
}

\date{\today}

	\maketitle
\begin{abstract}
We prove tight H\"olderian error bounds for all $p$-cones.
Surprisingly, the exponents differ in several ways from those that have been previously conjectured; moreover, they illuminate $p$-cones as a curious example of a class of objects that possess properties in 3 dimensions that they do not in 4 or more.
Using our error bounds, we analyse least squares problems with $p$-norm regularization, where our results enable us to compute the corresponding KL exponents for previously inaccessible values of $p$.
Another application is a (relatively) simple proof that most $p$-cones are neither self-dual nor homogeneous.
Our error bounds are obtained under the framework of facial residual functions, and we expand it by establishing for general cones an optimality criterion
under which the resulting error bound must be tight. 
\end{abstract}

\section{Introduction}

Consider the following conic feasibility problem:
\begin{align}
\text{find} & \quad {x} \in (\stdSpace + {a}) \cap \stdCone \label{eq:feas}\tag{Feas},
\end{align}
where $\stdCone$ is a closed convex cone contained in a finite dimensional Euclidean space $\ambSpace$, $\stdSpace \subseteq \ambSpace$ is a subspace and $a \in \ambSpace$. Here, we would like to tightly estimate the distance
from $x$ to $(\stdSpace + {a}) \cap \stdCone$ using
the individual distances between $x$ and $\stdSpace+a$ and between $x$ and $\stdCone$.
This is an \emph{error bound question} and is a classical topic in the optimization literature \cite{LT93,HF52,Pang97,LP98,ZS17}.

In this paper, we focus on the case where
$\stdCone = \pK$, the $p$-cone for $p \in (1,\infty)$, which is defined as
\begin{equation}\label{eq:pcone}
\pK := \{x = (x_0,\ttx)\in \RR^{n+1}\mid x_0 \ge \|\ttx\|_p\},
\end{equation}
where $\|\ttx\|_p$ denotes the $p$-norm of $\ttx$.
The cases $p = 1$ or $p = \infty$ are well-understood since they correspond to polyhedral cones and, therefore, Lipschitzian error bounds hold by Hoffman's lemma \cite{HF52}.

The case $p = 2$ corresponds to the second-order cones and their error bounds are also well-understood: Luo and Sturm proved that
if $\stdCone$ is a direct product of second order cones then \eqref{eq:feas} satisfies a H\"olderian error bound with exponent $(1/2)^{\alpha}$ and $\alpha$ depends on the level of regularity of the problem \cite[Theorems~7.4.1 and 7.4.2]{LS00} which, in this case, is upper bounded by the number of cones.
In particular, if $\stdCone = \POC{2}{n+1}$ then the exponent is $1/2$.

The case $p \in (1,\infty)$, $p \neq 2$ is quite peculiar.  Although not as well-known as the 2-cones, it has applications in facility location problems, regularization of least squares problems and others \cite{XY00,KBSZ11,SY15,ZZS15,ZS17}. The $p$-cone also appears in the recent push towards efficient algorithms and solvers for nonsymmetric  cones \cite{Ch09,NY12,SY15,KT19,PY21}.

There are significant differences between the cases $p = 2$ and $p \neq 2$, however. The former is a symmetric cone and, thus, enjoys a number of benefits that come with the Jordan algebraic structure that can be attached to it \cite{FK94,FB08} such as closed form expressions for projections. The other $p$-cones are not symmetric and do not typically have closed form expressions for their projections. See \cite{IL17_2,IL19} for a discussion on the extent to which they fail to be symmetric.

Despite the differences between general $p$-cones and $2$-cones, they still have quite a few similarities, so one
might be tempted to guess that if $\stdCone = \pK$ then \eqref{eq:feas} satisfies a H\"olderian error bound with exponent
$1/p$ as it was conjectured in \cite[Section~5]{L17}. It turns out that
this conjecture is wrong, and the true answer is far more interesting.

In this paper, our main contribution is to show for the
first time that explicit H\"olderian error bounds hold for all the $p$-cones and to determine the optimal exponents.
As we will see in Theorem~\ref{thm:main_err}, for a fixed $p$-cone, there are situations where the exponent is $1/2$ and others where the exponent becomes $1/p$.  It turns out that the correct exponent depends on the number of zeros that a certain vector exposing the feasible region of \eqref{eq:feas} has. Furthermore, there is one special case that only
happens when $p \in (1,2)$. We also compute H\"olderian error bounds for direct products of $p$-cones in Theorem~\ref{thm:pdirect}.
As an application of our results, we compute the KL exponent of the function associated to least squares problems with $p$-norm regularization; see section~\ref{sec:least}. Previously, an explicit exponent was only known when $p \in [1,2] \cup \{\infty\}$; see \cite{ZZS15,ZS17}. We also provide new ``easy'' proofs of some results about self-duality and homogeneity of $p$-cones; see section~\ref{sec:self}.

An important feature of our main $p$-cones error bound result is
that the bound is \emph{optimal} in a strong sense that implies, for example, that the exponents we found cannot be improved. It also precludes the existence of
better error bounds beyond H\"olderian ones; see~Theorem~\ref{thm:main_err} for the details. 

Our results are obtained under the facial residual function (FRF) framework developed in \cite{L17,LiLoPo20} which allows computation of error bounds for \eqref{eq:feas} without assuming constraint qualifications. Another major contribution in this work is that we expand the general framework in \cite{LiLoPo20} to allow verification of the tightness of the error bound. In particular, when the facial residual function satisfies a certain optimality criterion and the problem can be regularized in a single facial reduction step, the obtained error bound must be optimal; see Theorem~\ref{theo:best} and Corollary~\ref{col:besthold}. We believe this expansion, and the new associated criterion will be useful for analysing other cones.
For example, they easily verify the optimality of the FRFs constructed for the nontrivial exposed faces of the exponential cone in \cite{LiLoPo20}, while requiring no additional effort beyond what the authors used to merely show the FRFs were admissible;
see Remark~\ref{rem:expcone}.


This paper is organized as follows. In section~\ref{sec:prel}, we review preliminary materials including some essential tools developed in \cite{LiLoPo20} for computing FRFs. In section~\ref{sec:best}, we build the tightness framework and establish the optimality criterion for certifying tight error bounds. In section~\ref{sec:pcones}, we derive explicit error bounds for \eqref{eq:feas} with $\stdCone = \pK$ and certify their tightness. Finally, we discuss applications of our results in section~\ref{sec:app}. 

\section{Preliminaries}\label{sec:prel}
In this paper, we will follow the notation and definitions used in \cite{LiLoPo20}, where we explained in more details some of the background behind the techniques we used. Here we will be  terser and simply refer to the explanations contained in \cite{LiLoPo20} as needed.
We strongly recommend that a reader unfamiliar with our techniques takes at least a quick look at the main results in \cite{LiLoPo20}.

As a reminder, we assume throughout this paper that $\ambSpace$ is a finite dimensional Euclidean space $\ambSpace$. The inner product of $\ambSpace$ will be denoted by $\inProd{\cdot}{\cdot}$ and the induced norm by $\norm{\cdot}$. With that, for $x \in \ambSpace$ and a closed convex set $C \subseteq \ambSpace$, we denote the distance
between $x$ and $C$ by
$\dist(x,C) \coloneqq \inf_{y \in C} \norm{x-y}$. We denote the projection of $x$ onto $C$ by $P_C(x)$ so that $P_C(x) = \argmin _{y \in C} \norm{x-y}$.
We will also write $C^\perp$ for the orthogonal complement of $C$.
 We use $\RR_+$ and $\RPP$ to denote the set of nonnegative and positive real numbers, respectively. We also write $B(\eta) := \{x\mid \|x\|\le \eta\}$ for any $\eta\ge 0$.

Next, let $\stdCone$ be a closed convex cone contained in $\ambSpace$.
We denote
the boundary, relative interior, linear span, dual cone and dimension of $\stdCone$ by
$\partial \stdCone$, $\reInt \stdCone$,
$\spanVec \stdCone$, $\stdCone^*$ and $\dim \stdCone$, respectively. A cone is said to be \emph{pointed} if
$\stdCone \cap - \stdCone = \{0\}$. Related to that, we note the following well-known fact for further reference
\begin{equation}\label{eq:rint_dual}
z \in \reInt \stdCone^* \Rightarrow \stdCone \cap \{z\}^\perp = \stdCone \cap - \stdCone,
\end{equation}
see, for example, \cite[items~(i) and (iv) of Lemma~2.2]{LMT21} applied to the dual cone.

If $\stdFace \subseteq \stdCone$ is a face of $\stdCone$ we write $\stdFace \face \stdCone$. We say that a face $\stdFace$ is \emph{proper} if $\stdFace \neq \stdCone$ and \emph{nontrivial} if $\stdFace \neq \stdCone \cap - \stdCone$ and $\stdFace \neq \stdCone$. If $\stdFace = \stdCone \cap\{z\}^\perp$ for some $z \in \stdCone^*$, then $\stdFace$ is said to be \emph{exposed}.
A face of dimension one is called an \emph{extreme ray}. By convention, we only consider nonempty faces.

A finite collection of
faces of $\stdCone$ satisfying $\stdFace _{\ell}  \subsetneq \cdots \subsetneq \stdFace_1$ is called a \emph{chain of faces} and its length is defined to be $\ell$. Then, the \emph{distance to polyhedrality of $\stdCone$}, denoted by $\distP(\stdCone)$, is the length {\em minus one} of the longest chain of faces of $\stdCone$ such that only the final face $\stdFace _{\ell}$ is polyhedral. 


Next, we recall some basic definitions and results related to error bounds.
\begin{definition}[Lipschitzian and H\"olderian error bounds]\label{def:eb}
	Let $C_1, C_2 \subseteq \ambSpace$ be closed convex sets with $C_1 \cap C_2 \neq \emptyset$. We say
	that $C_1,C_2$ satisfy a \emph{uniform H\"olderian error bound with exponent $\gamma\in (0,1]$} if for every bounded set $B\subset \ambSpace$ there exists a constant $\kappa_B > 0$ such that
	\[
	\dist(x, C_1 \cap C_2) \leq \kappa_B \max\{\dist(x,C_1), \dist(x,C_2) \}^{\gamma}, \qquad \forall x\in B.
	\]
	If $\gamma = 1$, then the error bound is said to be \emph{Lipschitzian}.
\end{definition}

In what follows we say that \eqref{eq:feas} satisfies the \emph{partial-polyhedral Slater (PPS) condition} (see \cite{LMT18}) if one of the following three conditions holds: $(i)$ $\stdCone$ is polyhedral; $(ii)$ $(\stdSpace+a)\cap (\reInt \stdCone) \neq \emptyset$ (Slater's condition holds); $(iii)$ $\stdCone$ can be written as
$\stdCone = \stdCone^1 \times \stdCone^2$ where
$\stdCone^1$ is polyhedral and $(\stdSpace + a) \cap (\stdCone^1 \times (\reInt\stdCone^2) ) \neq \emptyset$.

From \cite[Corollary~3]{BBL99} it follows that if
\eqref{eq:feas} satisfies the PPS condition, then a Lipschitzian error bound holds; see \cite[Section~2.2]{LiLoPo20} for details. We register this below.

\begin{proposition}[Error bound under the PPS condition]\label{prop:pps_er}
	If \eqref{eq:feas} satisfies the PPS condition, then a Lipschitzian error bound holds.
\end{proposition}

Next, we will quickly review some ideas from \emph{facial reduction} \cite{BW81,Pa13,WM13} and the FRA-poly algorithm of \cite{LMT18}; see also \cite[Section~3]{LiLoPo20}.
The next proposition follows from the correctness of the FRA-Poly algorithm \cite[Proposition~8]{LMT18} and ensures that it is always possible to find a face of $\stdCone$ that contains the feasible region of \eqref{eq:feas} and satisfies the PPS condition.
\begin{proposition}[\!\!{\cite[Proposition~3.2]{LiLoPo20}}]\label{prop:fra_poly}
	Let $\stdCone = \stdCone^1\times \cdots \times \stdCone^s$, where
	each $\stdCone^j$ is a closed convex cone. Suppose \eqref{eq:feas} is feasible.
	Then there is a chain of faces
	\begin{equation}\label{eq:chain}
	\stdFace _{\ell}  \subsetneq \cdots \subsetneq \stdFace_1 = \stdCone
	\end{equation}
	of length $\ell$ and vectors $\{z_1,\ldots, z_{\ell-1}\}$ satisfying the following properties.
	\begin{enumerate}[$(i)$]
		\item \label{prop:fra_poly:1} $\ell -1\leq  \sum _{j=1}^{s} \distP(\stdCone ^j)  \leq \dim{\stdCone}$.
		\item \label{prop:fra_poly:2} For all $i \in \{1,\ldots, \ell -1\}$, we have
		\begin{flalign*}%
		z_i \in \stdFace _i^* \cap \stdSpace^\perp \cap \{a\}^\perp \ \ \ {and}\ \ \
		\stdFace _{i+1} = \stdFace _{i} \cap \{z_i\}^\perp.
		\end{flalign*}		
		\item \label{prop:fra_poly:3} $\stdFace _{\ell} \cap (\stdSpace+a) = \stdCone \cap (\stdSpace + a)$ and $\{\stdFace _{\ell},\stdSpace+a\}$ satisfies the PPS condition.
	\end{enumerate}
\end{proposition}

If \eqref{eq:feas} is feasible, we define the \emph{distance to the PPS condition} $\dpp(\stdCone, \stdSpace+a)$ as the length \emph{minus one} of the smallest chain of faces satisfying Proposition~\ref{prop:fra_poly}. We discuss briefly how to upper bound $\dpp(\stdCone, \stdSpace+a)$.
First, we observe that if $z_i$ belongs to the span of the $\{z_1,\ldots, z_{i-1}\}$ we would have $\stdFace_{i+1} = \stdFace_{i}$. So in order for the containments in \eqref{eq:chain} to be strict, the vectors in $\{z_1,\ldots, z_{\ell-1}\}$ must be linearly independent. From this observation and item~\ref{prop:fra_poly:1} of Proposition~\ref{prop:fra_poly}, we obtain
\begin{equation}\label{eq:dpp}
\textstyle \dpp(\stdCone, \stdSpace+a) \leq \min\left\{\sum _{j=1}^{s} \distP(\stdCone ^j), \dim(\stdSpace^\perp \cap \{a\}^\perp) \right\}.
\end{equation}
Furthermore, from the correctness of the FRA-Poly algorithm in \cite{LMT18}, it follows that there exists at least one chain of faces as in Proposition~\ref{prop:fra_poly}, see \cite[Proposition~8, item~$(i)$]{LMT18}.

Before we proceed, let us briefly recall the intuition behind the strategy in  \cite{LiLoPo20}, which is based on the following points.
\begin{itemize}
	\item We would like to obtain error bounds between $\stdCone$ and $\stdSpace+a$  as in \eqref{eq:feas}. If it were the case that the system in \eqref{eq:feas} satisfied some constraint qualification (e.g., Slater's condition), then a Lipschitzian error bound would hold by Proposition~\ref{prop:pps_er}.
	\item Unfortunately, in general, \eqref{eq:feas} does not satisfy a constraint qualification. However, through {facial reduction}, one may find a chain of faces as in \eqref{eq:chain} starting at $\stdCone$ and ending at a face $\stdFace_{\ell}$ that \emph{does} satisfy a constraint qualification, see item~\ref{prop:fra_poly:3} in Proposition~\ref{prop:fra_poly}.
	\item  Therefore, by Proposition~\ref{prop:pps_er}, a Lipschitzian error bound holds between $\stdSpace+a$ and $\stdFace_{\ell}$. In order to get an error bound between  $\stdSpace+a$ and $\stdCone$ (which is what we actually want), we need to estimate the distance to $\stdFace_{\ell}$ using other available information. This is accomplished through \emph{facial residual functions}, as described in the Section~\ref{subsec:frf}. They help to keep track of the distance to the faces $\stdFace_{i}$ in the chain as we do facial reduction.
	\item Once the facial residual functions are computed, the error bounds can be obtained by composing them in a special manner, as described in Theorem~\ref{theo:err}.
\end{itemize}

\subsection{Facial residual functions and error bounds}\label{subsec:frf}
We recall the definition of {\oneFRF} {(\oneFRFs)} \cite[Definition~3.4]{LiLoPo20}.

\begin{definition}[{\OneFRF} ({\oneFRFs})]\label{def:onefrf}
	Let $\stdCone$ be a closed convex cone and $z \in \stdCone^*$.
	A function $\psi _{\stdCone,z}:\RP \times \RP \to \RP$ is called a \emph{{\oneFRF} (\oneFRFs) for $\stdCone$ and $z$} if it satisfies the following properties:
	\begin{enumerate}[label=({\roman*})]
		\item $\psi_{\stdCone,z}$ is nonnegative, nondecreasing in each argument and $\psi_{\stdCone,z}(0,t) = 0$ for every $t \in \RP$.
		\item The following implication holds for every $x \in \spanVec \stdCone$ and every $\epsilon \geq 0$:
	\[
	\dist(x,\stdCone) \leq \epsilon, \quad \inProd{x}{z} \leq \epsilon \quad \Rightarrow \quad \dist(x,  \stdCone \cap \{z\}^{\perp})  \leq \psi_{\stdCone,z} (\epsilon, \norm{x}).
	\]
\end{enumerate}	

\end{definition}
{\oneFRFs}s are  a less general version of the notion of \emph{facial residual functions} developed in \cite{L17} but they are still enough for error bound purposes.

Next, we state the error bound based on {\oneFRFs}s proved in \cite{LiLoPo20}. In what follows, for functions $f:\RP\times \RP \to \RP$ and $g:\RP\times \RP \to \RP$, we define the \emph{diamond composition} $f\comp g$ to be the function satisfying
\begin{equation}\label{eq:diamond}
(f\comp g)(a,b) = f(a+g(a,b),b), \qquad \forall a,b \in \RP.
\end{equation}
The diamond composition makes it possible to relate the distance functions $\dist(\cdot, \stdCone)$, $\dist(\cdot, \stdSpace+a)$ to the distance functions to the faces $\stdFace_i$ that appear during facial reduction, see, for example, \cite[Lemma~3.7]{LiLoPo20}.
The intuitive idea is that we estimate the distances to the faces $\stdFace_2, \stdFace_3, \ldots, \stdFace_{\ell}$ recursively and, for a given $x$, we use $\dist(x, \stdFace_{i-1})$ to estimate $\dist(x,\stdFace_{i})$, which requires some sort of function composition of the {\oneFRFs}s.
The diamond composition was designed in such a way as to make this strategy succeed and to allow the (upper) estimation of
the distance to the final face $\stdFace_{\ell}$ appearing in Proposition~\ref{prop:fra_poly}.

Also if $\psi$ is a facial residual function, we say that $\hat \psi$ is a \emph{positively rescaled shift of $\psi$} if there are positive constants $M_1,M_2,M_3$ and a nonnegative constant $M_4$ such that
\begin{equation}\label{eq:pos_rescale}
\hat \psi (\epsilon,t) = M_3\psi_{\stdFace,z} (M_1\epsilon,M_2t) + M_4\epsilon.
\end{equation}
If $M_4 = 0$ in \eqref{eq:pos_rescale}, we say that $\hat \psi$ is a \emph{positive rescaling} of $\psi$.
\begin{theorem}[\!\!{\cite[Theorem~3.8]{LiLoPo20}}]\label{theo:err}
	Suppose \eqref{eq:feas} is feasible and let
\[
\stdFace _{\ell}  \subsetneq \cdots \subsetneq \stdFace_1 = \stdCone
\]
be a chain of faces of $\stdCone$  together with $z_i \in \stdFace _i^*\cap\stdSpace^\perp \cap \{a\}^\perp$ such that
$\{\stdFace _{\ell}, \stdSpace+a\}$ satisfies the
PPS condition and $\stdFace_{i+1} = \stdFace _i\cap \{z_i\}^\perp$ for every $i$.
For $i = 1,\ldots, \ell - 1$, let $\psi _{i}$ be a {\oneFRFs} for $\stdFace_{i}$ and $z_i$.\footnote{By convention, when $\ell = 1$, no $\psi_i$ will be defined.}

Then, there is a suitable positively rescaled shift of the $\psi _{i}$ (still denoted as $\psi_i$ by an abuse of notation) such that for any bounded set $B$ there is a  positive constant $\kappa_B$ (depending on $B, \stdSpace, a, \stdFace _{\ell}$) such that
\[
x \in B, \quad \dist(x,\stdCone) \leq \epsilon, \quad \dist(x,\stdSpace + a) \leq \epsilon\quad \Rightarrow \quad \dist\left(x, (\stdSpace + a) \cap \stdCone\right) \leq \kappa _B (\epsilon+\varphi(\epsilon,M)),
\]
where $M = \sup _{x \in B} \norm{x}$,
$\varphi = \psi _{{\ell-1}}\comp \cdots \comp \psi_{{1}}$, if $\ell \geq 2$. If $\ell = 1$, we let $\varphi$ be a function satisfying $\varphi(\epsilon, M) = \epsilon$ for all $\epsilon \ge 0$.
\end{theorem}
Theorem~\ref{theo:err} is a more general version of an analogous result for the so-called \emph{amenable cones}, see \cite[Theorem~23]{L17} and \cite{LRS20}.

To finish this subsection, we prove an auxiliary lemma that will be helpful to analyze H\"olderian error bounds.

%
%
\begin{lemma}\label{lem:hold}
	let $\stdCone$, $\stdSpace$, $a$ and
	$\psi_i$ be as in Theorem~\ref{theo:err}. 
Consider the following additional assumption on the  $\psi_i$:
	\begin{enumerate}[$(i)$]
		\item\label{lem:holdi} there exist $\alpha_i \in (0,1]$ and nonnegative, nondecreasing functions $\rho_i$, $\hat \rho_i$ such that
		$\psi_i(\epsilon,\eta) = \rho_i(\eta)\epsilon + \hat\rho_i(\eta)\epsilon^{\alpha_i}$ for
		every $\epsilon\geq0$ and $\eta\geq0$.
%
	\end{enumerate}
	Then $\stdCone$ and $\stdSpace+a$ satisfy a uniform H\"olderian error bound with exponent $\prod_{i=1}^{\ell-1}\alpha_i$ if $\ell \geq 2$ or exponent $1$ if $\ell = 1$.

\end{lemma}
\begin{proof}
If $\ell = 1$, then Theorem~\ref{theo:err} implies that a Lipschitzian error bound holds, so the result is true. So suppose that $\ell \geq 2$,
let $B\subseteq \ambSpace$ be an arbitrary bounded set and define $d$ as the function
satisfying $d(x) \coloneqq \max\{\dist(x,\stdCone),\dist(x,\stdSpace+a)\}$. Theorem~\ref{theo:err} implies
that there exists $\kappa_B > 0$ such that
\begin{equation}\label{eq:hold}
\dist\left(x, (\stdSpace + a) \cap \stdCone\right) \leq \kappa _B (d(x)+\varphi(d(x),\eta)), \qquad \forall x\in B,
\end{equation}
where $\eta = \sup _{x \in B} \norm{x}$, $\varphi =  (\psi_{\ell-1}\comp(\cdots \comp (\psi_2\comp \psi_1)))$ and the $\psi_i$ might have been positively rescaled with shifts.
Note however, that if $\psi_i$ is positively
rescaled and shifted, then, adjusting $\hat \rho_i$ and $\rho_i$ if necessary, $\psi_i(\epsilon,\eta) = \rho_i(\eta)\epsilon + \hat\rho_i(\eta)\epsilon^{\alpha_i}$ still holds with the same $\alpha_i$.

In what follows we make extensive use of the following principle: if $\beta _1 \leq \beta_2$, then, for a fixed bounded set $B$, we can find a constant $\kappa$ such that $d(x)^{\beta_2} \leq \kappa d(x)^{\beta_1}$ for all $x \in B$. Indeed,
\begin{equation}\label{eq:hold2}
d(x)^{\beta_2} =  d(x)^{\beta_2-\beta_1}d(x)^{\beta_1} \leq \left( \sup _{y \in B} d(y)^{\beta_2-\beta_1}\right)d(x)^{\beta_1}, \qquad \forall x\in B,
\end{equation}
where the sup is finite because the closure of $B$ is a compact set contained in the domain of the continuous function $d(\cdot)$  and $\beta_2 - \beta_1 \geq 0$. 

First, we will prove by induction that there exists
$\kappa _{\ell-1}> 0$ such that
\begin{equation}\label{eq:ind}
\varphi(d(x),\eta) \leq \kappa_{\ell-1} d(x)^{\widetilde\alpha_{\ell-1}},\ \ \ \forall x \in B,
\end{equation}
where $\widetilde\alpha _{\ell-1} := \prod_{i=1}^{\ell-1}\alpha_i$.
If $\ell = 2$, then $\varphi = \psi_1$. From \eqref{eq:hold2} and the assumption on the format of $\psi_1$ we can find a constant $\kappa > 0$ so that
\[
\rho_1(\eta)d(x) + \hat \rho_1(\eta)d(x)^{\alpha_1} \leq (\kappa\rho_1(\eta)+\hat \rho_1(\eta))d(x)^{\alpha_1},\ \ \ \forall x \in B.
\]
This implies that \eqref{eq:ind} holds when $\ell = 2$.
Suppose that $\eqref{eq:ind}$ holds for some $\hat \ell\ge 2$, and we will show that it holds for $\hat \ell+1$ (when a chain of faces of length $\hat \ell+1$ exists). In this case, we have $\varphi =  \psi_{\hat\ell}\comp\varphi_{\hat \ell-1}$, where $\varphi_{\hat \ell-1} = \psi_{\hat \ell-1}\comp(\cdots \comp (\psi_2\comp \psi_1))$. Now, the induction hypothesis (applied to $\varphi_{\hat \ell-1}$), the monotonicity of FRFs, and the observation in \eqref{eq:hold2} together imply that for some constants $\kappa_{\hat\ell-1}$,
$\hat \kappa_{\hat\ell-1}$ we have for all $x\in B$,
\begin{align}
\varphi(d(x),\eta) &= \psi_{\hat\ell}(d(x)+ \varphi_{\hat{\ell}-1}(d(x),\eta),\eta)\leq \psi_{\hat\ell}(d(x)+ \kappa_{\hat \ell-1}d(x)^{\widetilde\alpha_{\hat \ell -1}},\eta)\notag\\
&\leq \psi_{\hat\ell}(\hat \kappa_{\hat\ell-1}d(x)^{\widetilde\alpha_{\hat \ell -1}},\eta)\label{eq:hold3}.
\end{align}
Because of the assumption on the format
of $\psi_{\hat\ell}$, we have that
$\psi_{\hat\ell}(\hat \kappa_{\hat\ell-1}d(x)^{\widetilde\alpha_{\hat \ell -1}},\eta)$ can be written as $a_1 d(x)^{\widetilde\alpha_{\hat \ell -1}} + a_2d(x)^{\alpha_{\hat \ell} \widetilde\alpha_{\hat \ell -1}}$, for some constants $a_1,a_2$ which do not depend on $d(x)$. Since $\alpha_{\hat \ell} \in (0,1]$, we have $\widetilde\alpha_{\hat \ell} = \alpha_{\hat \ell}\widetilde\alpha_{\hat \ell-1} \leq \widetilde\alpha_{\hat \ell -1}$. Therefore,
 the observation in \eqref{eq:hold2} together with \eqref{eq:hold3} imply
the existence of $\kappa_{\hat \ell} > 0$ such that
\begin{equation}\label{eq:ind2}
\varphi(d(x),\eta) \leq \psi_{\hat\ell}(\hat \kappa_{\hat\ell-1}d(x)^{\widetilde\alpha_{\hat \ell -1}},\eta)\leq \kappa_{\hat\ell} d(x)^{\widetilde\alpha_{\hat \ell}},\ \ \ \forall x\in B,
\end{equation}
which concludes the induction proof. We have thus established that \eqref{eq:ind} holds. Plugging \eqref{eq:ind} into \eqref{eq:hold} and applying the observation \eqref{eq:hold2} yet again, we conclude that there exists a constant $\kappa > 0$ such that
\begin{equation*}
\dist\left(x, (\stdSpace + a) \cap \stdCone\right) \leq \kappa d(x)^{\widetilde\alpha_{\ell-1}}, \qquad \forall x\in B.
\end{equation*}
\end{proof}

\subsection{Computing {\oneFRFs}s}
It is clear from Theorem~\ref{theo:err} that the key to obtaining error bounds for \eqref{eq:feas} is the computation of {\oneFRFs}s, so in this final subsection we recall a few tools that will be helpful in this task.

First, we take care of the case when $z \in \reInt \stdCone^*$. In this case, if $\stdCone$ is pointed, then $\stdCone \cap \{z\}^\perp = \{0\}$ by \eqref{eq:rint_dual}.
\begin{lemma}[{\oneFRFs} for the zero face]\label{lem:frfzero}
	Suppose $\stdCone$ is a pointed closed convex cone and
	let $z \in \reInt \stdCone^*$.
	Then, there exists $\kappa > 0$ such that
	$\psi_{\stdCone,z}$ defined by
	$\psi_{\stdCone,z}(\epsilon,t) \coloneqq \kappa \epsilon$ is a {\oneFRFs}  for $\stdCone$ and $z$.
\end{lemma}
\begin{proof}
	Because $z \in \reInt\stdCone^*$ and $\stdCone$ is pointed, there exists $\kappa> 0$ such that
	$y \in \stdCone$ implies that
	$\norm{y} \leq \kappa \inProd{y}{z}$; see for example \cite[Lemma~26]{L17}.
	Next, suppose that $x \in \spanVec\stdCone$ satisfies $\dist(x,\stdCone) \leq \epsilon$ and
	$\inProd{x}{z} \leq \epsilon$. Then, there exists
	$h$ satisfying $\norm{h} \leq \epsilon$ such that
	$x + h \in \stdCone$. From  \eqref{eq:rint_dual} and the pointedness of $\stdCone$, we have
	\[
	\dist(x, \stdCone \cap \{z\}^\perp) = \norm{x} \leq \norm{x+h} + \norm{h} \leq \kappa\inProd{x+h}{z} + \epsilon \leq \kappa \epsilon + \kappa \epsilon\norm{z} + \epsilon.
	\]
	Therefore, $\psi_{\stdCone,z}(\epsilon,t) \coloneqq (\kappa+\kappa \norm{z}+1)\epsilon$ is a {\oneFRFs} for $\stdCone$ and $z$.
\end{proof}

The next several results are based on the idea of establishing local error bounds for $\stdCone$ and $\{z\}^\perp$ and using them to recover the corresponding {\oneFRFs} for $\stdCone$ and $z$. This will be the main approach we will use for computing the {\oneFRFs}s of $\pK$ and its faces. We first provide some of the intuition in Remark~\ref{rem:intuition} and illustrate in Figure~\ref{fig:cone}. See \cite[Section~3.1]{LiLoPo20} for more detailed explanations.

\begin{figure}
	\begin{center}
		\begin{tikzpicture}
			[scale=0.76]
			\node[anchor=south west,inner sep=0] (image) at (0,0) {\includegraphics[width=.4\linewidth]{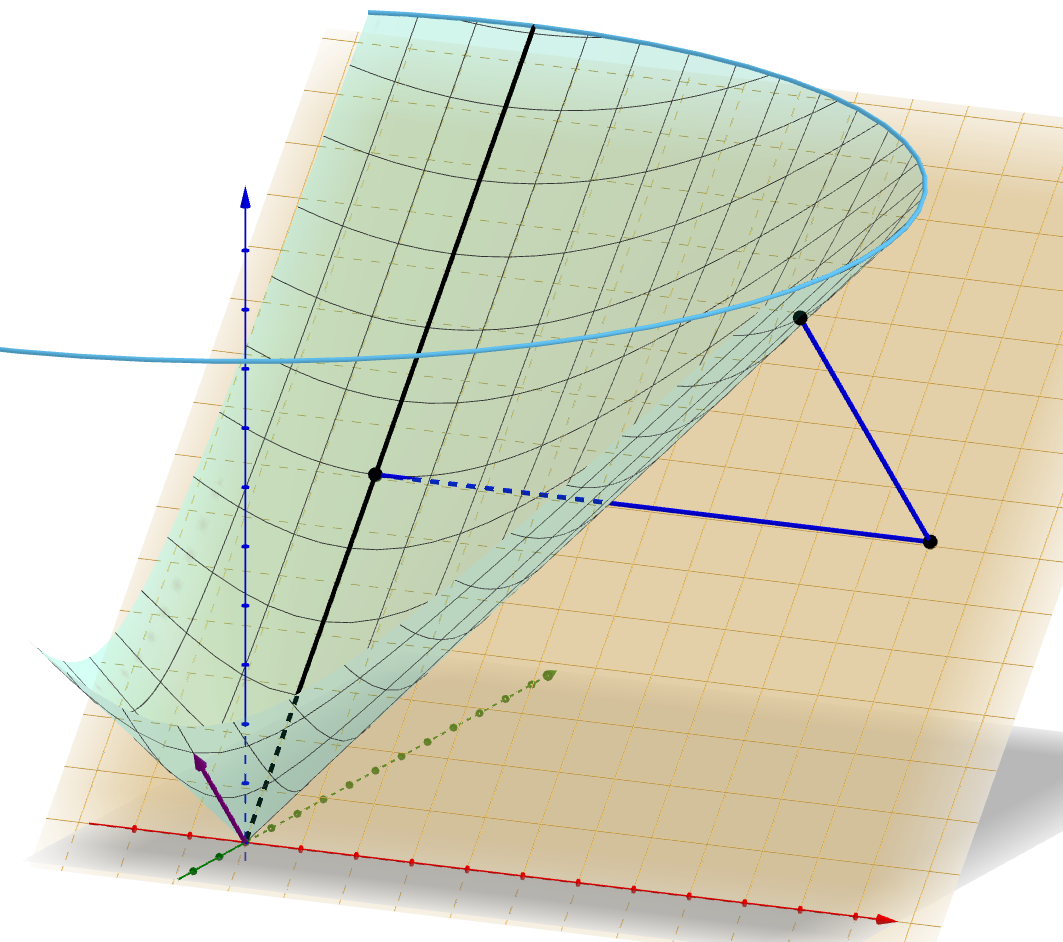}};
			\begin{scope}[x={(image.south east)},y={(image.north west)}]
				
				\definecolor{darkpurple}{rgb}{0.5,0,0.5}
				\node[darkpurple] at (0.175,0.225) {$z$};
				
				\definecolor{darkorange}{rgb}{0.3,0.15,0.15}
				\node[darkorange] at (0.8,0.175) {$z^\perp$};
				
				\node[black] at (0.4,0.75) {$\stdFace$};%
				
				\definecolor{coneteal}{rgb}{0,0.5,0.5}
				\node[coneteal] at (0.8,0.95) {$\stdCone_2^{2+1}$};%
				
				\node[black] at (0.820,0.675) {$v$};
				\shade[shading=ball, ball color=black] (0.755,0.665) circle (0.015);	
				\node[black] at (0.94,0.44) {$w$};
				\shade[shading=ball, ball color=black] (0.875,0.425) circle (0.015);	
				\node[black] at (0.3,0.5) {$u$};
				\shade[shading=ball, ball color=black] (0.355,0.495) circle (0.015);				
				
				%
				%
				%
				%
				%
				
			\end{scope}
		\end{tikzpicture}
		\begin{tikzpicture}[scale=6.0]
			\definecolor{darkorange}{rgb}{1,0.5,00}
			\draw[darkorange,thick] (-0.2,0) -- (1.1,0);
			
			
			\definecolor{coneteal}{rgb}{0,0.5,0.5}
			\draw[scale=1,domain=0.001:0.372,smooth,variable=\u,coneteal,thick] plot ({\u},{1/2.675-sqrt(1-2.675*2.675*\u*\u)/2.675});
			\draw[scale=1,domain=0.001:0.2,smooth,variable=\u,coneteal,thick] plot ({-\u},{1/2.675-sqrt(1-2.675*2.675*\u*\u)/2.675});	
			
			\node[above,darkorange] at (.55,-0.01) {$\{z\}^\perp$};
			\node[above,coneteal] at (.18,0.07) {$\stdCone_2^{2+1}$};
			
			\shade[shading=ball, ball color=purple] (0.95,0.15) circle (0.015) node [above,purple] {$p$};
			\draw[->,purple] (0.95,0.125) -- (0.95,.025);
			\shade[shading=ball, ball color=purple] (0.95,0) circle (0.015) node  [below,purple] {$q =P_{\{{z} \}^\perp}p$};
			\draw[<-,purple] ({exp(-1/0.8)/0.8+.05},{exp(-1/0.8)-0.025}) -- (0.9,.025);
			\shade[shading=ball, ball color=black] ({exp(-1/0.8)/0.8},{exp(-1/0.8)}) circle (0.015) node  [left,black] {$v=P_{\stdCone}q$};
			\draw[->,blue] ({exp(-1/0.8)/0.8},{exp(-1/0.8)-0.025}) -- ({exp(-1/0.8)/0.8},0.025);
			\shade[shading=ball, ball color=black] ({exp(-1/0.8)/0.8},{0}) circle (0.015) node  [below,black] {$w=P_{\{{z} \}^\perp}v$};
			\draw[->,blue] ({exp(-1/0.8)/0.8-0.025},0.025) -- (0.025,0.025);
			\shade[shading=ball, ball color=black] (0,0) circle (0.015) node  [below,black] {$u=P_{\stdFace}w$};
			
		\end{tikzpicture}
	\end{center}
	\caption{The framework of Lemma~\ref{lem:facialresidualsbeta}, Theorem~\ref{thm:1dfacesmain}, and Lemma~\ref{lem:infratio} is illustrated. The right image shows a 2D slice of the left image, where the slice is in the plane given by ${\rm span}\{v,w,u\}$.}\label{fig:cone}
\end{figure}
\begin{remark}[Geometric intuition for finding suitable {\oneFRFs}]\label{rem:intuition}
	We will next introduce the framework we use to find facial residual functions and then use those facial residual functions to recover error bounds. Let us first explain geometrically what the following results allow us to do. See Figure~\ref{fig:cone} for an intuitive illustration of the relative positions of the points that arise.
	\begin{enumerate}
		\item First, Lemma~\ref{lem:facialresidualsbeta} will show that we may replace the problem of computing {\oneFRFs}s with a simpler error bound problem that only considers points $q$ that are constrained to the exposing hyperplane $\{z\}^\perp$.
		\item Second, Theorem~\ref{thm:1dfacesmain} will show that we can replace the problem involving $q$ with an equivalent problem involving $v \in \stdCone$ and $w = P_{\{z\}^\perp}v$ and $u = P_{\stdFace} w$. { At this point the need to have an explicit projection onto $\stdCone$ vanish, because the error bound problem is cast in terms of points $v$ that move along the boundary of $\stdCone$}.
		\item Third, Lemma~\ref{lem:infratio} provides us a path to solve the problem of $v,w,u$. Critically, as we shall see in Section~\ref{sec:pcones}, when applied to $\pK$, the projections involved in solving this problem ($P_{\{z\}^\perp}$ and $P_{\stdFace}$) are piecewise-linear maps that can be described \textbf{explicitly}.
	\end{enumerate}
This reformulation is an improvement over the original problems involving $p$ or $q$, because those problems would seem to require an explicit form of the projection $P_{\stdCone}$, while only an \textbf{implicit} form is available.
\end{remark}
\color{black}

The next lemma shows how to recover from such an
error bound the facial residual function for $\stdCone$ and $z$.

\begin{lemma}[\!\!{\cite[Lemma~3.9]{LiLoPo20}}]\label{lem:facialresidualsbeta}
	Suppose that $\stdCone$ is a closed convex cone and let $z\in \stdCone^*$ be such that $\stdFace = \{z \}^\perp\cap \stdCone$ is a proper face of $\stdCone$.
	Let $\frakg:\RR_+\to\RR_+$ be nondecreasing with $\frakg(0)=0$, and let $\kappa_{z,\fraks}$ be a finite nondecreasing nonnegative function in $\fraks\in \RR_+$ such that
	\begin{equation}\label{assumption:q}
	\dist(q,\stdFace) \leq \kappa_{z,\norm{q}} \frakg(\dist(q,\stdCone))\ \ \mbox{whenever}\ \ q \in \{z\}^\perp.
	\end{equation}
	Define the function $\psi_{\stdCone,z}:\RP\times \RP\to \RP$ by
	\begin{equation*}
	\psi_{\stdCone,z}(s,t) := \max \left\{s,s/\|z\| \right\} + \kappa_{z,t}\frakg \left(s +\max \left\{s,s/\|z\| \right\} \right).
	\end{equation*}
	Then we have
	\begin{equation}\label{haha}
	\dist(p,\stdFace) \leq \psi_{\stdCone,z}(\epsilon,\norm{p}) \mbox{\ \ \ \  whenever\ \ \ \ $\dist(p,\stdCone) \leq \epsilon$\ \ and\ \ $\inProd{p}{z} \leq \epsilon$.}
	\end{equation}
	Moreover, $\psi_{\stdCone,z}$ is a {\oneFRFs} for $\stdCone$ and $z$.
\end{lemma}

Next, we recall a result on how to compute error bounds suitable
to be used in conjunction with Lemma~\ref{lem:facialresidualsbeta}.

\begin{theorem}[\!\!{\cite[Theorem~3.10]{LiLoPo20}}]\label{thm:1dfacesmain}
	Suppose that $\stdCone$ is a closed convex cone and let ${z}\in \stdCone^*$ be such that $\stdFace = \{{z} \}^\perp \cap \stdCone$ is a nontrivial exposed face of $\stdCone$.
	Let $\eta \ge 0$, $\alpha \in (0,1]$ and let $\frakg:\RR_+\to \RR_+$ be nondecreasing with $\frakg(0) = 0$ and $\frakg \geq |\cdot|^\alpha$. Define
	\begin{equation}\label{gammabetaeta}
	\gamma_{{z},\eta} \!:=\! \inf_{{v}} \bigg\{\frac{\frakg(\|{w}-{v}\|)}{\|{w}-{u}\|}\;\bigg|\; {v}\in \bd \stdCone\cap B(\eta)\backslash\stdFace,\ {w} = P_{\{{z} \}^\perp}{v},\  {u} = P_{\stdFace}{w},\ {w}\neq {u}\bigg\}. 
	\end{equation}
	\begin{enumerate}[label=(\roman*)]
	\item\label{eqn:added_for_remark1} If $\gamma_{{z},\eta} \in (0,\infty]$, then it holds that
	\begin{equation}\label{haha2}
		\dist(q,\stdFace) \leq \kappa_{{z},\eta} \frakg(\dist(q,\stdCone))\ \ \mbox{whenever\ $q \in \{{z} \}^\perp \cap B(\eta)$},\color{black}
	\end{equation}
	where $\kappa_{{z},\eta} := \max \left \{2\eta^{1-\alpha}, 2\gamma_{{z},\eta}^{-1}   \right \} < \infty$.
	\item\label{eqn:added_for_remark2} On the other hand, if there exists $\kappa_B \in \left(0,\infty \right)$ so that
	\begin{equation*}
	\dist(q, \stdFace) \leq \kappa_B \frakg (\dist(q,\stdCone)) \;\; \text{whenever} \; q \in \{z\}^\perp \cap B(\eta),
	\end{equation*}
	then $\gamma_{z,\eta} \in \left(0,\infty \right]$.
	\color{black}
	\end{enumerate}
	
\end{theorem}

Note that $\gamma_{z,0} = \infty$.
Given $\stdCone, z$ and $\stdFace$ as in Theorem~\ref{thm:1dfacesmain} and $\eta > 0$, we typically prove
that $\gamma_{{z},\eta}$ in \eqref{gammabetaeta} is nonzero by
contradiction. The next lemma aids in this task.
\begin{lemma}[\!\!{\cite[Lemma~3.12]{LiLoPo20}}]\label{lem:infratio}
	Suppose that $\stdCone$ is a closed convex cone and let ${z}\in \stdCone^*$ be such that $\stdFace = \{{z} \}^\perp \cap \stdCone$ is a nontrivial exposed face of $\stdCone$.
	Let $\eta > 0$, $\alpha \in (0,1]$ and let $\frakg:\RR_+\to \RR_+$ be nondecreasing with $\frakg(0) = 0$ and $\frakg \geq |\cdot|^\alpha$. Let $\gamma_{{z},\eta}$ be defined as in \eqref{gammabetaeta}. If $\gamma_{{z},\eta} = 0$, then there exist
	$\widehat v \in \stdFace$ and a sequence $\{{v}^k\}\subset \bd \stdCone\cap B(\eta) \backslash \stdFace$ such that
	\[
	\underset{k \rightarrow \infty}{\lim}{v}^k = \underset{k \rightarrow \infty}{\lim}{w}^k = \widehat v \ \ \
	{\rm and}\ \ \ \lim_{k\rightarrow\infty} \frac{\frakg(\|{w}^k - {v}^k\|)}{\|{w}^k- {u}^k\|} = 0,
	\]
	where ${w}^k = P_{\{{z}\}^\perp}{v}^k$, ${u}^k = P_{\stdFace}{w}^k$ and ${w}^k\neq {u}^k$.
\end{lemma}

The next result will be helpful in the analysis of one-dimensional faces (i.e., extreme rays).
\begin{lemma}\label{lem:dist}
	Let $\stdCone$ be a pointed closed convex cone and let $z\in \bd\stdCone^*\backslash\{0\}$ be such that $\stdFace := \{z\}^\perp\cap \stdCone$ is a one-dimensional proper face of $\stdCone$. Let $f\in \stdCone\setminus\{0\}$ be such that
	\[
	\stdFace = \{t f\mid t\ge 0\}.
	\]
	Let $\eta > 0$ and $v\in \bd \stdCone\cap B(\eta)\backslash \stdFace$, $w = P_{\{z\}^\perp}v$ and $u = P_{\stdFace}w$ with $u\neq w$.
	Then it holds that $\langle f,z\rangle = 0$ and we have
	\begin{equation*}
	\|v - w\| = \frac{|\langle z,v\rangle|}{\|z\|},\ \ \|u - w\| = \begin{cases}
	\left\|v - \frac{\langle z,v\rangle}{\|z\|^2}z - \frac{\langle f,v\rangle}{\|f\|^2}f\right\| & {\rm if}\ \langle f,v\rangle\ge 0,\\
	\ \left\|v - \frac{\langle z,v\rangle}{\|z\|^2}z\right\| & {\rm otherwise}.
	\end{cases}
	\end{equation*}
	Moreover, when $\langle f,v\rangle \ge 0$ (or, equivalently, $\langle f,w\rangle\ge 0$), we have $u = P_{{\rm span}\stdFace}w$. On the other hand, if $\langle f,v\rangle < 0$, we have $u = 0$.
\end{lemma}
\begin{proof}
	The fact that $\langle f,z\rangle = 0$ follows from $f\in \stdFace\subseteq \{z\}^\perp$. The formula for $\|v - w\|$ holds because the projection of $v$ onto $\{z\}^{\perp}$ is $w = v - \frac{\langle z,v\rangle}{\|z\|^2}z$. Finally, notice that $u$ is obtained as $t^*f$, where
	$
	t^* = \argmin_{t\ge 0}\left\{\|w - tf\|\right\}.
	$
	Then $t^*$ is given by $\frac{\langle w,f\rangle}{\|f\|^2}$ if $\langle w,f\rangle\ge 0$ (in this case, we have $u = P_{{\rm span}\stdFace}w$), and is zero otherwise. The desired formulas now follow immediately by noting that $\langle w,f\rangle = \langle v - \frac{\langle z,v\rangle}{\|z\|^2}z,f\rangle = \langle v,f\rangle$. 
\end{proof}

We conclude this subsection with a result on the direct product of cones. In essence, {\oneFRFs}s of direct products are positively rescaling of sums of the {\oneFRFs}s for each individual block.

\begin{proposition}[{\cite[Proposition~3.13]{LiLoPo20}}]\label{prop:frf_prod}
	Let $\stdCone^i \subseteq \ambSpace^i$ be closed convex cones for every $i \in \{1,\ldots,m\}$ and let $\stdCone = \stdCone^1 \times \cdots \times \stdCone^m$.
	Let
	$\stdFace \face \stdCone$, $z \in \stdFace^*$ and suppose that
	$\stdFace = \stdFace^1\times \cdots \times \stdFace^m$ with $\stdFace^i \face \stdCone^i$ for every $i \in \{1,\ldots,m\}$.
	Write $z = (z_1,\ldots,z_m)$ with $z_i \in (\stdFace^i)^*$.
	
	For every $i$, let $\psi_{\stdFace^i,z_i}$ be a {\oneFRFs} for $\stdFace^i$ and $z_i$.
	Then, there exists a $\kappa > 0$ such that  the function $\psi _{\stdFace,z}$ satisfying
	\[
	\psi _{\stdFace,z}(\epsilon,t) = \sum _{i=1}^m \psi_{\stdFace^i,z_i}(\kappa\epsilon,t)
	\]
	is a {\oneFRFs} for $\stdFace$ and $z$.
	
	
\end{proposition}

We end this section with a remark on the definition of {\oneFRFs}.
\begin{remark}
	As suggested by one of the referees, it is possible to consider facial residual functions
	that are valid over the unit ball and use them to construct {\oneFRFs} as in Definition~\ref{def:onefrf}. For example, suppose that $\bar{\psi}: \RR_+ \to \RR_+$ is a nondecreasing function satisfying $\bar{\psi}_{\stdCone,z}(0) = 0$ and such that the following implication holds for every $x \in \spanVec \stdCone$ with $\norm{x} = 1$ and every $\epsilon \geq 0$:
	\begin{equation}\label{eq:ballonefrf}
	\dist(x,\stdCone) \leq \epsilon, \quad \inProd{x}{z} \leq \epsilon \quad \Rightarrow \quad \dist(x,  \stdCone \cap \{z\}^{\perp})  \leq \bar \psi_{\stdCone,z} (\epsilon).
	\end{equation}
	Since $\dist(\alpha x, \stdCone) = \alpha \dist(x, \stdCone)$ for $\alpha \geq 0$, we obtain the following implication for every $x \in \spanVec \stdCone\setminus\{0\}$ and $\epsilon \ge 0$:
	\[
	\dist(x,\stdCone) \leq \epsilon, \quad \inProd{x}{z} \leq \epsilon \quad \Rightarrow \quad  \dist(x/\norm{x},\stdCone) \leq \epsilon/\norm{x}, \quad \inProd{x/\norm{x}}{z} \leq \epsilon/\norm{x},
	\]
	which, in view of \eqref{eq:ballonefrf}, implies $\dist(x,  \stdCone \cap \{z\}^{\perp})  \leq \norm{x}\bar\psi_{\stdCone,z} (\epsilon/{\norm{x}})$.
	Let $\psi_{\stdCone,z}: \RR_+\times \RR_+ \to \RR_+$  be such that $\psi_{\stdCone,z}(a,b) \coloneqq b\bar\psi_{\stdCone,z}(a/b)$ (if $b > 0$) and
	$\psi_{\stdCone,z}(a,b) \coloneqq 0$ (otherwise).
	In this way, if $\psi_{\stdCone,z}(a,\cdot)$ happens to be monotone\footnote{Note that this is not true for a general function $f:\RR_+\to \RR_+$. The function $f(a) \coloneqq a + a^2$ is monotone nondecreasing for $a \geq 0$, but $bf(1/b)$ is not monotone nondecreasing for $b > 0$. If, however, $f$ is of the form $f(a) = \kappa a + \kappa a^\alpha$ (as in the case associated to H\"olderian error bounds) for some $\alpha \in (0,1]$ and a constant $\kappa > 0$, then $bf(a/b)$ is monotone nondecreasing as a function of $b$ for fixed $a$.}  for every $a > 0$,  then  $\psi_{\stdCone,z}$ would be a {\oneFRFs} as in Definition~\ref{def:onefrf}.
	Conversely, if $\psi_{\stdCone,z}$ is as in Definition~\ref{def:onefrf}, we can specialize it to the unit ball by fixing the second argument to $1$.
	
	Because of these relations, it is plausible that one could rebuild the abstract theory of error bounds (or at least part of it) described in \cite{L17,LiLoPo20} using only {\oneFRFs}s that are restricted to the unit ball.
By doing so, one would hope that some expressions could get simpler.
	However, it is not clear to us what the concrete mathematical benefit in following this path would be.
	Even from a purely notational point of view, keeping track of the norm term (i.e., the second argument in {\oneFRFs}s) has advantages, since for the final error bound in Theorem~\ref{theo:err} we cannot freely rescale elements in the underlying affine space $\stdSpace+a$.
	To be fair, in the end, it could be just a matter of taste. 

	We remark that the core difficulty in this and in our previous works is the \emph{actual} computation of {\oneFRFs}s for a given  $\stdCone$.
	This is where the inevitable particularities and complexities of each cone come into play.
	Whether this is done over the unit ball or over a ball of radius $\eta$ (as in Theorem~\ref{thm:pconeface}), our assessment is that the difficulty is essentially the same.

\end{remark}

\section{Best error bounds}\label{sec:best}
In this section, we will present some results
that will help us identify when certain
error bounds are the ``best'' possible.
The error bounds discussed in this paper come from Theorem~\ref{theo:err}, which uses {\oneFRFs}s.
The {\oneFRFs}s themselves are constructed from Lemma~\ref{lem:facialresidualsbeta}, and most of the properties are inherited from the function
$\frakg$ appearing in \eqref{assumption:q}.
We will show in this section that, under some assumptions, the criterion described in Definition~\ref{def:opt_g} below will be enough to show optimality of the underlying error bound induced by $\frakg$. In order to keep the notation compact, we will use $w_{\epsilon}$ instead of $w({\epsilon})$.

\begin{remark}[Geometric intuition for optimality]
	Let us preface Definition~\ref{def:opt_g} and the next few results with an informal discussion of the geometric intuition that motivates them.
	Let $v:(0,1]\to \partial \stdCone\backslash \stdFace$ be a continuous function such that $w_\epsilon \coloneqq P_{\{z\}^\perp}(v_\epsilon)$ and $u_\epsilon \coloneqq P_{\stdFace}(w_\epsilon)$ satisfy $w_\epsilon \neq u_\epsilon$, following the illustrations of $v$, $w$ and $u$ in Figure~\ref{fig:cone}. As $v_\epsilon$ approaches some $\bar{v} \in \stdFace$ as $\epsilon \downarrow 0$,
	we have  $w_{\epsilon} \to \bar{v} = P_{\{z\}^\perp}(\bar{v})$ and $u_{\epsilon} \to \bar{v}$ as well, since $\stdFace = \stdCone \cap \{z\}^\perp$.
	 Nevertheless, the distances $\|v_\epsilon-w_\epsilon\|$ and $\|w_\epsilon-u_\epsilon\|$ are not necessarily of the same order asymptotically.
	\begin{enumerate}[label=(\roman*)]
		\item\label{remopt:1} By Theorem~\ref{thm:1dfacesmain}, a \textbf{suitable} $\frakg$---``suitable'' in the sense that $\gamma_{{z},\eta}$ is greater than zero---ensures that $\frakg(\|w_\epsilon-v_\epsilon\|)$ does not go to zero with a faster rate than $\|w_\epsilon-u_\epsilon\|$ for any choice of $v$;
		\item\label{remopt:2} A \textbf{suitable} $\frakg$ being \textbf{optimal} in the sense of Definition~\ref{def:opt_g} would satisfy that $\frakg(\dist(w_\epsilon , \stdCone))$ and $\|w_\epsilon-u_\epsilon\|$ are asymptotically of the \textbf{same order} as $\epsilon \downarrow 0$ for \textbf{some choice} of $w$ in the following sense:
		$$
		0< \liminf_{\epsilon \downarrow 0} \frac{\frakg(\dist(w_\epsilon , \stdCone))}{\|w_\epsilon - u_\epsilon \|}	\quad \text{and}\quad \limsup_{\epsilon \downarrow 0} \frac{\frakg(\dist(w_\epsilon , \stdCone))}{\|w_\epsilon - u_\epsilon \|}	< \infty,
		$$
		where the first inequality comes from Theorem~\ref{thm:1dfacesmain} since $\norm{w_{\epsilon} - u_{\epsilon}} = \dist(w_{\epsilon}, \stdFace)$.
		\item\label{remopt:3}
		Under this setting,  we have $w_{\epsilon} \to \bar v \in \stdFace$ as $\epsilon \downarrow 0$, so $\dist(w_{\epsilon}, \stdCone) \to 0$ and $\dist(w_{\epsilon}, \stdFace) \to 0$.
		Naturally, if we were to take this same choice of $w_\epsilon$ and choose any better $\hat{\frakg}(t)$---\textit{better} meaning that $\lim_{t \downarrow 0}\hat{\frakg}(t)/\frakg(t)=0$---then we would have that:
		$$
		\liminf_{\epsilon \downarrow 0}\frac{\hat{\frakg}(\dist(w_\epsilon , \stdCone))}{\dist(w_\epsilon,\stdFace)} = %
		%
		%
		\liminf_{\epsilon \downarrow 0} \underbrace{\frac{\hat{\frakg}(\dist(w_\epsilon , \stdCone))}{\frakg(\dist(w_\epsilon , \stdCone))}}_{\rightarrow 0 \; \text{as}\; \epsilon \downarrow 0}\; \underbrace{\frac{\frakg(\dist(w_\epsilon , \stdCone))}{\dist(w_\epsilon , \stdFace)}}_{\text{bounded}} = 0.
		$$
		%
	Since $w$ is continuous, $w_\epsilon \in \{z\}^\perp \cap B(\eta)$ holds for sufficiently small $\epsilon$ and sufficiently large $\eta$. Therefore, for such $\eta$, no $\kappa_B> 0$ will satisfy
		$$
		\dist(q,\stdFace) \leq \kappa_B \hat{\frakg}(\dist(q,\stdCone)) \;\;\text{whenever} \;\; q \in \{z\}^\perp \cap B(\eta).
		$$
		In view of Theorem~\ref{thm:1dfacesmain}\ref{eqn:added_for_remark1} \& \ref{eqn:added_for_remark2}, the non-existence of such a $\kappa_B$ for $\hat{\frakg}$ is equivalent to  $\gamma_{z,\eta} = 0$ for $\hat{\frakg}$, which means that $\hat{\frakg}$ is \textbf{not suitable}.
		\item\label{remopt:4} Finally, we prove in Theorem~\ref{thm:bestFRF} below, that up to a multiplicative constant and some technicalities related to the parameters, any valid {\oneFRFs} is lower bounded by a {\oneFRFs} built from an optimal $\frakg$.
		In particular, this implies that, when the second-parameter is fixed, no valid {\oneFRFs} can go to zero at a faster rate than a {\oneFRFs} built from an optimal $\frakg$.
	\end{enumerate}
\end{remark}
\color{black}

\color{black}

\begin{definition}[An optimality criterion for $\frakg$]\label{def:opt_g}
Let $\frakg:\RR_+\to\RR_+$ be nondecreasing with $\frakg(0)=0$.
Let $\stdCone$ be a closed convex cone and ${z}\in \stdCone^*$ be such that $\stdFace = \{{z} \}^\perp\cap \stdCone$ is a proper face of $\stdCone$.
If there exist $\bv \in \stdFace \setminus \{0\}$  and a continuous function $w:(0,1] \to  \{z\}^\perp \cap \spanVec\stdCone \setminus \stdFace$ satisfying
\begin{equation}\label{G1}\tag{G1}
\underset{\epsilon \downarrow 0}{\lim}\, w_{\epsilon} = \bv\ \ \ {\rm and}\ \ \ \limsup_{\epsilon \downarrow 0} \frac{\frakg(\dist(w_\epsilon,\stdCone))}{\dist(w_\epsilon,\stdFace)}	=: L_{\frakg}< \infty,
\end{equation}
then we say that $\frakg$ satisfies the
\emph{asymptotic optimality criterion for $\stdCone$ and $z$}.\footnote{With an abuse of terminology, we will simply say ``\eqref{G1} holds" when this happens.}
\end{definition}

We will next show that any {\oneFRFs}s built from a concave $\frakg$ satisfying \eqref{G1} is optimal in the sense that, up to a constant, it must be better than any other possible {\oneFRFs}s for the same sets. First, we need some preliminary lemmas.

\begin{lemma}\label{lem:concave}
	Let $\frakg:\RR_+\to \RR_+$ be concave with $\frakg(0)=0$. Then
	$\frakg((1+\lambda)s) \leq (1+\lambda)\frakg(s)$ for all positive numbers $\lambda$ and $s$.
\end{lemma}
\begin{proof}
	Note that \[s = \frac{1}{1+\lambda}(1+\lambda)s  + \frac{\lambda}{1+\lambda}  0.\]
	Hence, $\frakg(s) \geq \frakg((1+\lambda)s)/(1+\lambda) + (\lambda /(1+\lambda))\frakg(0) = \frakg((1+\lambda)s)/(1+\lambda)$.
\end{proof}

In the following, $w_\epsilon$ may be intuitively thought of as the usual $w$ in Figure~\ref{fig:cone}. In Theorem~\ref{thm:bestFRF}, we will find it useful to ``pin'' the norm of this $w_\epsilon$ term by replacing it with $\tau_\epsilon w_\epsilon$ where $\tau_\epsilon$ is a scalar chosen to force $\|\tau_\epsilon w_\epsilon\|$ to be constant for all $\epsilon$. We will need the following lemma.
\color{black}

\begin{lemma}\label{lem:scaledlimit2}
Let $\frakg:\RR_+\to\RR_+$ be nondecreasing with $\frakg(0)=0$.
Let $\stdCone$ be a closed convex cone and ${z}\in \stdCone^*$ be such that $\stdFace = \{{z} \}^\perp\cap \stdCone$ is a proper face of $\stdCone$.
Suppose that \eqref{G1} holds and let $\tau:(0,1] \to \RPP$ satisfy $\lim_{\epsilon \downarrow 0} \tau_{\epsilon} =\overline{\tau} \in \RPP$. Suppose further that
	the following partial sub-homogeneity holds:\footnote{In view of Lemma~\ref{lem:concave}, \eqref{G2} automatically holds for any $S > 0$ when $\frakg$ is concave.}
	\begin{equation}\label{G2}\tag{SH}
	\exists S>0\;\; \text{such\;that}\;\; s\in [0,S]\;\;\text{and}\;\;\hat \tau \ge 1\;\; \text{together\;imply}\;\; \frakg(\hat{\tau}s) \leq \hat{\tau}\frakg(s).
	\end{equation}
	Then
	\begin{equation}
	\limsup_{\epsilon \downarrow 0} \frac{\frakg(\dist(\tau_{\epsilon} w_\epsilon,\stdCone))}{\dist(\tau_{\epsilon}w_\epsilon,\stdFace)} \leq \max \left\{\frac{L_{\frakg}}{\overline{\tau}},L_{\frakg} \right\}.
	\end{equation}
\end{lemma}
\begin{proof}
	Since $w_\epsilon \to \bv \in \stdFace$ as $\epsilon$ goes to $0$, for all sufficiently small $\epsilon$, we have $\dist(w_\epsilon,\stdCone) \leq S$.
	For those $\epsilon$, if $\tau_{\epsilon} \geq 1$, we have
	\begin{align*}
\frac{\frakg(\dist(\tau_{\epsilon}w_\epsilon,\stdCone))}{\dist(\tau_{\epsilon}w_\epsilon,\stdFace)}\stackrel{\rm (a)}{\leq} \frac{\tau_{\epsilon}\frakg(\dist(w_\epsilon,\stdCone))}{\tau_{\epsilon}\dist(w_\epsilon,\stdFace)} = \frac{\frakg(\dist(w_\epsilon,\stdCone))}{\dist(w_\epsilon,\stdFace)},
	\end{align*}
	where (a) follows by \eqref{G2} and the fact that
	$\dist(aw_\epsilon, \stdCone) = a\dist(w_\epsilon,\stdCone)$ holds for every convex cone $\stdCone$ and nonnegative scalar $a$.
	If $\tau_{\epsilon} < 1$, we have
	\begin{align*}
\frac{\frakg(\dist(\tau_{\epsilon}w_\epsilon,\stdCone))}{\dist(\tau_{\epsilon}w_\epsilon,\stdFace)} \stackrel{\rm (a)}{\leq}\frac{\frakg(\dist(w_\epsilon,\stdCone))}{\tau_{\epsilon}\dist(w_\epsilon,\stdFace)},
	\end{align*}
	where (a) follows from the monotonicity of $\frakg$.
	Overall, we conclude that for sufficiently small $\epsilon$, we have
	\begin{equation*}
\frac{\frakg(\dist(\tau_{\epsilon}w_\epsilon,\stdCone))}{\dist(\tau_{\epsilon}w_\epsilon,\stdFace)}\leq \max\left\{\frac{\frakg(\dist(w_\epsilon,\stdCone))}{\dist(w_\epsilon,\stdFace)}, \frac{\frakg(\dist(w_\epsilon,\stdCone))}{\tau_{\epsilon}\dist(w_\epsilon,\stdFace)} \right\}.
	\end{equation*}
The desired conclusion now follows immediately upon invoking \eqref{G1}.
\end{proof}


\begin{theorem}[Optimality of {\oneFRFs}s satisfying \eqref{G1}]\label{thm:bestFRF}
Let $\stdCone$ be a closed convex cone, $z \in \stdCone^*$ with $\norm{z} = 1$ and let $\stdFace \coloneqq \stdCone \cap \{z\}^\perp$ be a nontrivial exposed face of $\stdCone$.
Let $\frakg$, $\gamma_{{z},\eta}$ and $\kappa_{z,\eta}$ be as in Theorem~\ref{thm:1dfacesmain} such that $\gamma_{{z},\eta} \in (0,\infty]$ for every $\eta > 0$. Let
\begin{equation}\label{stdform}
\psi_{\stdCone,{z}}(s,t) := s + \kappa_{{z},t} \frakg(2s),
\end{equation}
so that $\psi_{\stdCone,{z}}(s,t)$ is a {\oneFRFs} for $\stdCone$ and $z$ (see Lemma~\ref{lem:facialresidualsbeta}).
Suppose further that \eqref{G2} and \eqref{G1} hold.

Consider any $\bar\eta > 0$ and define $M:\RP\to\RP$ as follows:
\begin{equation}\label{def:M}
M(t) := \frac12  \left[1+ \kappa_{{z},\bar{\eta}}\cdot \left(\max \left\{1,\frac{\|{\bv}\|}{t}\right \}L_{\frakg}\right)\right]^{-1},
\end{equation}
where $\bv$ and $L_{\frakg}$ are as in \eqref{G1}.
Let $\psi_{\stdCone,{z}}^\star$ be an arbitrary  {\oneFRFs} for $\stdCone$ and $z$.
	Then, for any $\bar{\eta}_0\in (0,\bar{\eta}]$, there exists $s_0 > 0$ such that
	\begin{equation}\label{goal}
	M(\bar{\eta}_0)\psi_{\stdCone,{z}}(s,b) \leq \psi_{\stdCone,{z}}^{\star}(s,\bar{\eta}_0),\ \ \  \forall (s,b) \in \left[0,s_0\right]\times\left[0,\bar{\eta} \right].
	\end{equation}
\end{theorem}
\begin{proof}
	Let $w_\epsilon$ and $\bv$ be as in \eqref{G1}.
	Let $\tau$  be the function defined as follows:
	\begin{equation}\label{eqn:tau_pin}
	\tau_{\epsilon} \coloneqq
	\begin{cases}
	\frac{\bar{\eta}_0}{\norm{w_\epsilon}}  &\mbox{if}\,\, \epsilon \in (0,1],\\
	\frac{\bar{\eta}_0}{\norm{\bv}} &\mbox{if}\,\, \epsilon=0.
	\end{cases}	
	\end{equation}
	As a reminder, we are using $\tau_\epsilon$ as a shorthand for $\tau(\epsilon)$.
	
Note that $\dist(\tau_{\epsilon}w_\epsilon, \stdFace) \neq 0$ because $w_\epsilon \not \in \stdFace$ by assumption.
  Using this and the definition of $\psi_{\stdCone,z}$ in \eqref{stdform}, we have for all $\epsilon \in (0,1]$ that
	\begin{equation}\label{bubu6}
\begin{aligned}
		\frac{\psi_{\stdCone,{z}} (\dist(\tau_{\epsilon}w_\epsilon,\stdCone),\bar{\eta})}{\dist(\tau_{\epsilon}w_\epsilon, \stdFace)}
 &=\frac{\dist(\tau_{\epsilon}w_\epsilon, \stdCone)}{\dist(\tau_{\epsilon}w_\epsilon, \stdFace)} + \frac{\kappa_{{z},\bar{\eta}} \frakg \left(2\dist(\tau_{\epsilon}w_\epsilon, \stdCone) \right)}{\dist(\tau_{\epsilon}w_\epsilon, \stdFace)}\\
 & \le 1 + \frac{\kappa_{{z},\bar{\eta}} \frakg \left(2\dist(\tau_{\epsilon}w_\epsilon, \stdCone) \right)}{\dist(\tau_{\epsilon}w_\epsilon, \stdFace)},
 \end{aligned}
	\end{equation}
where the inequality holds because $\stdFace\subseteq \stdCone$, which implies that $\dist(\tau_{\epsilon}w_\epsilon,\stdCone) \leq \dist(\tau_{\epsilon} w_\epsilon,\stdFace)$ for all $\epsilon \in (0,1]$.

Next, notice that $\lim_{\epsilon\downarrow 0}\dist(\tau_{\epsilon}w_\epsilon, \stdCone)= 0$. Using this and \eqref{G2}, we see that for all sufficiently small $\epsilon$,
	\begin{equation}\label{bubu8}
		\frac{\kappa_{{z},\bar{\eta}} \frakg \left(2\dist(\tau_{\epsilon}w_\epsilon, \stdCone) \right)}{\dist(\tau_{\epsilon}w_\epsilon, \stdFace)}
		\leq \frac{\kappa_{{z},\bar{\eta}} 2 \frakg \left(\dist(\tau_{\epsilon}w_\epsilon, \stdCone) \right)}{\dist(\tau_{\epsilon}w_\epsilon, \stdFace)}.
	\end{equation}
Combining \eqref{bubu8} with Lemma~\ref{lem:scaledlimit2} and recalling again that \eqref{G1} and \eqref{G2} hold, we deduce further that
	\begin{equation}\label{bubu9}
	\limsup_{\epsilon \downarrow 0}\frac{\kappa_{{z},\bar{\eta}} \frakg \left(2\dist(\tau_{\epsilon}w_\epsilon, \stdCone) \right)}{\dist(\tau_{\epsilon} w_\epsilon, \stdFace)} \leq \kappa_{{z},\bar{\eta}} \cdot2\left(\max \left\{1,\frac{1}{\tau_0}\right \}L_{\frakg}\right) < \infty.
	\end{equation}
	Now, combining \eqref{bubu6} and \eqref{bubu9}, we have that
	\begin{align}
	\limsup_{\epsilon \downarrow 0} \frac{\psi_{\stdCone,{z}} (\dist(\tau_{\epsilon}w_\epsilon,\stdCone),\bar{\eta})}{\dist(\tau_{\epsilon}w_\epsilon, \stdFace)} \leq 1+\kappa_{{z},\bar{\eta}} \cdot 2\left(\max \left\{1,\frac{1}{\tau_0}\right \}L_{\frakg}\right) < \infty.\label{bubu10}
	\end{align}
	Consequently, there must exist $\hat \epsilon \in (0,1]$ such that
	\begin{equation}\label{bubu11}
	\frac{\psi_{\stdCone,{z}} (\dist(\tau_{\epsilon}w_\epsilon,\stdCone),\bar{\eta})}{\dist(\tau_{\epsilon}w_\epsilon, \stdFace)} \leq 2+\kappa_{{z},\bar{\eta}}\cdot 2\left(\max \left\{1,\frac{1}{\tau_{0}}\right \}L_{\frakg}\right) = \frac{1}{M(\bar{\eta}_0)},\ \ \forall \epsilon \in (0,\hat \epsilon].
	\end{equation}
	On the other hand, since $\psi_{\stdCone,{z}}^\star$ is a {\oneFRFs} and $w_\epsilon \in \spanVec\stdCone\cap \{z\}^\perp\setminus \stdFace$, we have from \cite[Remark~3.5]{LiLoPo20} that for all $\epsilon\in (0,1]$,
	\begin{equation}\label{bubu5}
	\quad 1 \leq \frac{\psi_{\stdCone,{z}}^\star(\dist(\tau_{\epsilon}w_\epsilon,\stdCone),\|\tau_{\epsilon}w_\epsilon\|)}{\dist(\tau_{\epsilon}w_\epsilon, \stdFace)}.
	\end{equation}	
	Combining \eqref{bubu11} with \eqref{bubu5}, we deduce that for all $\epsilon \in(0, \hat \epsilon]$,
	\begin{equation}
	M(\bar{\eta}_{0})\psi_{\stdCone,{z}} (\dist(\tau_{\epsilon}w_\epsilon,\stdCone),\bar{\eta}) \leq \psi_{\stdCone,{z}}^\star (\dist(\tau_{\epsilon}w_\epsilon,\stdCone),\|\tau_{\epsilon}w_\epsilon\|).\label{bubu14}
	\end{equation}
	Since $\psi_{\stdCone,{z}}$ is monotone, we further obtain
	that for all $\epsilon \in(0, \hat \epsilon]$ and $b \in \left[0,\bar{\eta}\right]$,
\begin{equation}
\begin{aligned}
	 & M(\bar{\eta}_{0})\psi_{\stdCone,{z}} (\dist(\tau_{\epsilon}w_\epsilon,\stdCone),b) \leq
	M(\bar{\eta}_{0})\psi_{\stdCone,{z}} (\dist(\tau_{\epsilon}w_\epsilon,\stdCone),\bar{\eta}) \\
&\leq \psi_{\stdCone,{z}}^\star (\dist(\tau_{\epsilon}w_\epsilon,\stdCone),\|\tau_{\epsilon}w_\epsilon\|) = \psi_{\stdCone,{z}}^\star (\dist(\tau_{\epsilon}w_\epsilon,\stdCone),\bar{\eta}_0).
	\end{aligned}\label{bubu15}
\end{equation}
 Since $\epsilon\mapsto \dist(\tau_{\epsilon}w_\epsilon,\stdCone)$ is continuous with $\lim_{\epsilon\downarrow 0} \dist(\tau_{\epsilon}w_\epsilon,\stdCone)=0$ but is {\em positive} on $(0,\hat\epsilon]$ (since $w_\epsilon \in \{z\}^\perp \setminus \stdFace$, hence $w_\epsilon \not \in \stdCone$), its image
 contains some interval of the form $(0,s_0]$, where $s_0 \coloneqq \dist(\tau_{\hat \epsilon}w_{\hat\epsilon},\stdCone)$.
	Therefore,
	\eqref{bubu15} shows that
	for every $s \in (0,s_0]$ and for
	every $b \in [0,\bar{\eta}]$ we have
	\begin{align*}
 M(\bar{\eta}_{0})\psi_{\stdCone,{z}} (s,b) \leq  \psi_{\stdCone,{z}}^\star (s,\bar{\eta}_0).
	\end{align*}
	The proof is now complete upon noting that the above relation holds trivially when $s = 0$ because both sides of the inequality become zero when $s=0$, according to the definition of facial residual function.
\end{proof}

We now move on to an application of Theorem~\ref{thm:bestFRF}. The next result says that if \eqref{eq:feas} only requires a single facial reduction step and the {\oneFRFs} is as in Theorem~\ref{thm:bestFRF}, then the obtained error bound must be optimal. This will be discussed in the framework of
consistent error bound functions developed in \cite{LL20}, which we now recall.

\begin{definition}[Consistent error bound functions]\label{def:ceb}
	Let $C_1,\ldots, C_m\subseteq{\cal E}$ be closed convex sets with $C:=\bigcap_{i=1}^mC_i\neq\emptyset$.
	A function $\Phi :\RR_+\times\RR_+ \to \RR_+ $ is
	said to be a \emph{consistent error bound function} for $C_1,\ldots, C_m$ if:
	\begin{enumerate}[$(i)$]
		\item  the following error bound condition is satisfied:
		\begin{equation}\label{def_eb}
		\dist(x,\, C) \le \Phi\left(\max_{1 \le i \le m}\dist(x, C_i), \, \|x\|\right), \ \ \ \forall\ x\in\mathcal{E};
		\end{equation}
		\item 	for  any fixed $b\ge 0$,  the function $\Phi(\cdot,\, b)$ is nondecreasing on $\RR_+$, right-continuous at $0$ and satisfies $\Phi(0,\, b) = 0$;
		\item for any fixed $a\ge 0$, the function $\Phi(a, \, \cdot)$ is nondecreasing on $\RR_+$.
	\end{enumerate}
	We say that \eqref{def_eb} is the \emph{consistent error bound} associated to $\Phi$.
\end{definition}

Consistent error bound functions were defined in such a way as to provide a broad framework for the study of general error bounds and facilitate the study of convergence properties of certain algorithms.
H\"olderian error bounds, the error bounds developed under the theory of amenable cones and the error bounds found in the study of the exponential cone can all be expressed through the framework of consistent error bound functions; see \cite{LL20}.
Furthermore, they are ``universal'' in the sense that whenever finitely many closed convex sets intersect, there is always a least one consistent error bound function describing the error bound associated to their intersection, see \cite[Proposition~3.3]{LL20}.
For applications to convergence analysis, see \cite[Sections~4, 5 and 6]{LL20}.

With that, we can now state the main result of this section.

\begin{theorem}[Best error bounds]\label{theo:best}	
	Suppose \eqref{eq:feas} is feasible and consider the following three assumptions.
	\begin{enumerate}[$(i)$]
		\item\label{3p6i}  There exist ${z} \in \stdCone^* \cap \stdSpace^\perp \cap \{{a}\}^\perp$ and $\stdFace \coloneqq \stdCone \cap \{{z}\}^\perp$  such that $\stdFace$ is a nontrivial exposed face and
		$\{\stdFace, \stdSpace + {a}\}$ satisfies the PPS condition.
		\item\label{3p6ii} The function $\psi$ is a {\oneFRFs} for
		$\stdCone $ and ${z}$ as in Lemma~\ref{lem:facialresidualsbeta}, for some $\frakg$ and $\kappa_{z,\eta}$ as in Theorem~\ref{thm:1dfacesmain} so that the $\gamma_{z,\eta}$ in \eqref{gammabetaeta} satisfies $\gamma_{z,\eta}\in (0,\infty]$ for every $\eta > 0$.
		\item\label{3p6iii} The function $\frakg$ from \ref{3p6ii} also satisfies \eqref{G2} and \eqref{G1} holds.
	\end{enumerate}
	Then, the following hold.
	\begin{enumerate}[$(a)$]
		\item\label{theo:besta} There is a positively rescaled shift of the $\psi$ denoted by $\hat \psi$ such that for any bounded set $B$, there is a  positive constant $\kappa_B$ (depending on $B, \stdSpace, {a}, \stdFace$) such that for every $	{x} \in B$ and $\epsilon \geq 0$ we have the following implication
		\begin{equation}\label{eq:sing_one}
\begin{aligned}
  	\dist({x},\stdCone) \leq \epsilon \ {\rm and}\ \dist({x},\stdSpace + {a}) \leq \epsilon
   \quad \Longrightarrow \quad \dist\left({x}, (\stdSpace + {a}) \cap \stdCone\right) \leq \kappa _B (\epsilon+\hat\psi(\epsilon,\widetilde\eta)),
\end{aligned}
		\end{equation}
		where $\widetilde\eta = \sup _{{x} \in B} \norm{{x}}$.
		\item\label{theo:bestb}  The error bound in \eqref{eq:sing_one} is optimal in the following sense. Suppose we also assume that
        \begin{equation}\label{3p6iv}\tag{$iv$}
        \norm{{z}}=1,\;\; \stdSpace = \{{z}\}^\perp,\;\; \text{and}\;\;  a = 0.
        \end{equation}
		Then for any consistent error bound function $\Phi$ for $\stdCone, \stdSpace$ and any $\hat\eta > 0$, there are constants $\hat \kappa > 0$ and $s_0 > 0$ such that
		\[
		s+\hat\psi(s,\hat{\eta}) \leq \hat \kappa \Phi(\hat \kappa s,\hat{\eta}),\qquad \forall s\in [0,s_0].
		\]
	\end{enumerate}		
\end{theorem}
\begin{proof}
	Item \ref{theo:besta} follows directly from Theorem~\ref{theo:err} with $\ell = 2$.
	
	We move on to item \ref{theo:bestb}.
	Suppose that $\Phi$ is a consistent error bound function for $\hat\stdSpace = \{{z}\}^\perp$ and $\stdCone$.
	First, we will show that $\Phi$ is a {\oneFRFs} for $\stdCone$ and ${z}$.	Let
	\[
	d(x) \coloneqq \max\{\dist({x},\stdCone), \dist({x},\{{z}\}^\perp) \},
	\]
	so that
	\begin{equation}\label{eq:ceb}
	\dist({x}, \stdCone \cap \{{z}\}^\perp ) =\dist({x},\stdFace)\leq \Phi(d(x),\norm{{x}}),\qquad \forall {x}\in \ambSpace.
	\end{equation}
	Suppose that ${x}$ is such that $\dist({x}, \stdCone) \leq \epsilon, \inProd{{x}}{{z}} \leq \epsilon $. Then, there exists $u$ with $\norm{{u}} \leq \epsilon$ such that ${x} + {u} \in \stdCone$, Therefore
	\[
	0 \leq \inProd{{x}+{u} }{{z}}\ \ \ {\rm and\ hence}\ \ \
	-\epsilon \leq \inProd{-u}{z}\le \inProd{({x}+{u})-{u}}{{z}} =\inProd{{x}}{{z}} \leq \epsilon.
	\]
	That is $|\inProd{{x}}{{z}}| =  \dist({x},\{{z}\}^\perp) \leq  \epsilon$ and thus $d(x) \leq \epsilon$.
	Then, \eqref{eq:ceb} implies that
	$\Phi$ is a {\oneFRFs} for $\stdCone$ and ${z}$.
	
	By assumption, the conditions of Theorem~\ref{thm:bestFRF} are satisfied. Hence, for any $\hat{\eta} > 0$, there exists $s_0 > 0$ such that
	\begin{equation}\label{eq:comp}
	M(\hat{\eta})\psi(s,b) \leq \Phi (s,\hat{\eta}),\ \  \forall (s,b) \in \left[0,s_0\right]\times\left[0,\hat{\eta} \right],
	\end{equation}
	where $M$ is as in \eqref{def:M}.
	In addition, by the definition of positive rescaling with a shift, there are positive constants $M_1,M_2,M_3$ and a nonnegative constant $M_4$ such that $\hat \psi(s,t) = M_1 \psi (M_2 s, M_3t) + M_4 s$.  We have:
	\begin{align*}
	{s} +\hat\psi({s} ,\hat{\eta}) & =
	(1+M_4){s}  + M_1 \psi (M_2 {s} , M_3\hat{\eta}) \\
	&= (1+M_4){s}  + M_1(M_2 {s}  + \kappa_{{z},M_3\hat{\eta}}\frakg(2M_2{s} ))\\
	& = {s} (M_1M_2+M_4+1) +  M_1\kappa_{{z},M_3\hat{\eta}}\frakg(2M_2{s} )\\
	& \leq \kappa {s}  + \kappa\frakg(2\kappa{s} ),
	\end{align*}
	where $\kappa \coloneqq \max \{M_1M_2+M_4+1, M_1\kappa_{{z},M_3\hat{\eta}},M_2 \}$ and the last inequality holds because of the monotonicity and nonnegativity of $\frakg$. Continuing, note that since $\hat\eta > 0$, we have $\gamma_{z,\hat\eta}\in (0,\infty]$ and hence $\kappa_{z,\hat\eta} = \max\{2\hat\eta^{1-\alpha},2\gamma_{z,\hat\eta}^{-1}\} > 0$. Thus,
	\begin{align*}
	{s} +\hat\psi({s} ,\hat{\eta}) & \leq \kappa {s}  + \kappa\frakg(2\kappa{s} )= \kappa {s}  + \kappa \frac{\kappa_{{z},\hat{\eta}}}{\kappa_{{z},\hat{\eta}}} \frakg(2\kappa {s} )\\
	& \leq \max\left\{1, \frac{\kappa}{\kappa_{{z},\hat{\eta}}}\right\}(\kappa {s}  + \kappa_{{z},\hat{\eta}}\frakg({2\kappa {s} })) = \max\left\{1, \frac{\kappa}{\kappa_{{z},\hat{\eta}}}\right\} \psi(\kappa {s} , \hat{\eta}).
	\end{align*}
	Therefore, if $s \in [0, s_0/\kappa]$, we have from \eqref{eq:comp} and the above display that
	\[
	{s} +\hat\psi({s} ,\hat{\eta}) \leq  \max\left\{1, \frac{\kappa}{\kappa_{{z},\hat{\eta}}}\right\} M(\hat{\eta})^{-1}\Phi (\kappa{s} ,\hat{\eta} )\leq \hat \kappa \Phi(\hat \kappa s,\hat{\eta})
	\]
	where $\hat \kappa \coloneqq \max\left\{\kappa,\max\left\{1, \frac{\kappa}{\kappa_{{z},\hat{\eta}}}\right\}   M(\hat{\eta})^{-1}\right\} $ and the second inequality follows from the monotonicity of $\Phi$ in the first entry. This completes the proof.
\end{proof}

\begin{corollary}[Best H\"olderian bounds]\label{col:besthold}
Suppose \eqref{eq:feas} is feasible and suppose items \ref{3p6i}, \ref{3p6ii} and \ref{3p6iii} in Theorem~\ref{theo:best} hold for some $\frakg = |\cdot|^\alpha$ with $\alpha \in (0,1)$. Then, the following items hold.
\begin{enumerate}[$(a)$]
	\item\label{col:bestholdi} $\stdCone$ and $\stdSpace+a$ satisfy a uniform H\"olderian error bound with exponent $\alpha$.
	\item\label{col:bestholdii} Suppose the display relation \eqref{3p6iv} in Theorem~\ref{theo:best} also holds. Then for any consistent error bound
	function $\Phi$ for $\stdCone$, $\stdSpace$ and any $\hat\eta > 0$, there are constants $\hat \kappa > 0$  and $s_0 > 0$ such that
	\[
	s^{\alpha} \leq \hat \kappa \Phi(\hat \kappa s,\hat{\eta}),\qquad \forall s\in [0,s_0].
	\]
	In particular, $\stdCone$ and $\stdSpace$ do not satisfy a uniform H\"olderian error bound with exponent
	$\hat \alpha$ where  $\alpha < \hat \alpha \leq 1$.
\end{enumerate}
\end{corollary}
\begin{proof}
First, we prove item $(a)$.
Noting that $z/\norm{z} \in  \stdCone^* \cap \stdSpace^\perp \cap \{a\}^\perp$ and
$\{z/\norm{z}\}^\perp = \{z\}^\perp$, we assume (without loss of generality) that $\norm{z} = 1$.
By assumption, we have $\frakg = |\cdot|^\alpha$, so that the facial residual function $\psi$ from Lemma~\ref{lem:facialresidualsbeta} satisfies
\begin{equation}\label{eq:psi_st}
\psi(s,t) = s + \kappa_{{z},t} (2s)^\alpha,
\end{equation}
where $\kappa_{z,t}$ is as in Theorem~\ref{thm:1dfacesmain} and is nonnegative nondecreasing in $t$. Applying Lemma~\ref{lem:hold} with $\ell = 2$ gives a uniform H\"olderian error bound with exponent $\alpha$.

Next, we move on to item $(b)$.
By item $(b)$ of Theorem~\ref{theo:best}, for any $\hat\eta > 0$,
there are constants $\hat \kappa > 0$ and $s_0 > 0$ such that for every $s \in [0,s_0]$ we have
\begin{equation}\label{eq:best_ceb}
s+\hat\psi(s,\hat{\eta}) \leq \hat \kappa \Phi(\hat \kappa s,\hat{\eta}),
\end{equation}
where $\hat \psi$ is a positively rescaled shift
of $\psi$ in \eqref{eq:psi_st}.
The left-hand-side of \eqref{eq:best_ceb}, as a function of $s$, now has the form $a_1 s + a_2s^{\alpha}$ for some constants $a_1 > 0, a_2 > 0$. Therefore, adjusting $\hat \kappa$ if necessary, we have
\begin{equation}\label{eq:opt_h}
s^{\alpha} \leq \hat \kappa \Phi(\hat \kappa s,\hat{\eta}),\qquad \forall s \in [0,s_0].
\end{equation}
For the sake of obtaining a contradiction, suppose that a uniform H\"olderian error bound holds for $\stdCone$ and $\stdSpace$ with exponent $\hat \alpha$ for some $\hat\alpha\in(\alpha,1]$.
Then, there exists a nonnegative nondecreasing function $\hat \rho$ such that $\Phi$ given
by $\Phi(s,t) = \hat\rho(t)s^{\hat\alpha}$ is a consistent error bound function for $\stdCone$ and $\stdSpace$; see \cite{LL20} or
this footnote\footnote{For any positive integer $r$, there is a constant $\kappa_{r} > 0$ such that Definition~\ref{def:eb} holds for $B$ equal to the ball centered at the origin with radius $r$. Adjusting the constants if necessary, we have $\kappa_{r} \leq \kappa_{r'}$ if $r \leq r'$, so we can let $\hat\rho(b)$ be the $\kappa_{r}$ such that
	$r$ is the smallest integer larger than $b$.}.

By what have been shown so far and in view of \eqref{eq:opt_h}, there are constants $\hat \kappa > 0$  and $s_0 > 0$ such that for every $s \in [0,s_0]$ we have
\begin{equation}\label{eq:bestH}
s^{\alpha} \leq \hat \kappa \hat \rho(\hat{\eta}) (\hat \kappa s)^{\hat \alpha}.
\end{equation}
Dividing both sides by $s^{\hat \alpha}$ and letting $s \downarrow 0$, we get a contradiction: the left-hand-side blows up to infinity, while the right-hand-side is constant.
\end{proof}
\begin{remark}[On the exponential cone]\label{rem:expcone}
The exponential cone in $\RR^3$ is
$$
\stdCone_{\exp}:=\left\{(x,y,z)\mid y>0,z\geq ye^{x/y} \right\}\cup \stdFace_{-\infty},\;\; \stdFace_{-\infty}:= \left\{(x,0,z) \mid  x\leq 0,z\geq 0\right\}.
$$
Its nontrivial exposed faces are the 2-D face $\stdFace_{-\infty}$, infinitely many 1-D faces of form \[\stdFace_{\beta}=\{\left(y-\beta y,y,e^{1-\beta}y \right)\;|\;y \geq 0 \}\] for $\beta \in \RR$, and the exceptional 1-D face \[\stdFace_{\infty}:=\{(x,0,0) \mid x\leq 0\};\] see \cite[Section~4.1]{LiLoPo20}. For $t$ sufficiently near $0$, the $\frakg$ corresponding to their (worst case) {\oneFRFs}s simplify to $\frakg_{-\infty}(t)=-t\ln(t)$, $\frakg_{\beta}(t)=\sqrt{t}$ and $\frakg_{\infty}(t)=-1/\ln(t)$, respectively; see Corollaries 4.4, 4.7 and 4.11 in \cite{LiLoPo20}.

Once admissibility of these $\frakg$ is established, the condition \eqref{G1} may be verified by letting $w_{\epsilon}=(-1,\epsilon,0)$ for $\stdFace_{\infty}$, $w_{\epsilon}=(-\epsilon \ln(\epsilon),0,1)$ for $\stdFace_{-\infty}$, and $w_{\epsilon}=P_{\{z\}^\perp}(1-\beta+\epsilon,1,e^{1-\beta+\epsilon})$ for $\stdFace_{\beta}$ when $\beta \in \RR$: Indeed, the arguments in \cite[Remark~4.14]{LiLoPo20} demonstrate that $L_{\frakg_{\infty}}\le 1$, $L_{\frakg_{-\infty}}\le 1$ and $L_{\frakg_{\beta}} \in \left(0,\infty\right)$ (note that $L_{\frakg_{\beta}}$ is what the authors labeled as $L_{\beta}$ in \cite{LiLoPo20}). Thus,
our framework can be used  to show that the error bounds for the exponential cone are also tight in the sense of item~\ref{theo:bestb} of Theorem~\ref{theo:best}.
\end{remark}

\section{Error bounds for $p$-cones}\label{sec:pcones}
In this section, we will compute the facial residual functions
for the $p$-cones, obtain error bounds and prove their optimality. First, we recall that for
$p \in (1,\infty)$,
$\ttx \in \RR^n$, $n\ge 2$, the $p$-norm of $\ttx$ and the $p$-cone are given by
\begin{equation}\label{eq:pnorm}
\norm{\ttx}_p \coloneqq \sqrt[p]{|\ttx_1|^p + \cdots +  |\ttx_n|^p}, \qquad \pK := \{x = (x_0,\ttx)\in \RR^{n+1}\mid x_0 \ge \|\ttx\|_p\};
\end{equation}
here, given a vector $x\in \RR^{n+1}$, we use $x_0$ to denote its first (0th) entry and $\ttx$ to denote the subvector obtained from $x$ by deleting $x_0$. Here, we fix $\inProd{\cdot}{\cdot}$ the usual Euclidean inner product so that the dual cone of $\pK$ is the $q$-cone, where $\frac1p+\frac1q=1$:
\begin{equation*}
\qK := \{z = (z_0,\ttz)\in \RR^{n+1}\mid z_0 \ge \|\ttz\|_q\}.
\end{equation*}
We recall that $\pK$ is a pointed full-dimensional cone.
In what follows, we will be chiefly concerned with the case
$p \in (1,\infty)$ and $n \geq 2$.
 We will also make extensive use of the following lemma.
\begin{lemma}\label{keylemma}
	Let $p, q\in (1,\infty)$ be such that $\frac1p+\frac1q=1$ and let $\zeta\in \RR^n$ ($n\ge 1$) satisfy $\|\zeta\|_q = 1$. Define \[\overline\zeta \coloneqq -{\rm sgn}(\zeta)\circ|\zeta|^{q-1},\] where $\circ$ is the Hadamard product, and ${\rm sgn}$, absolute value and the $q-1$ power are taken componentwise. Then $\|\overline\zeta\|_p = 1$. Moreover, there exist $C > 0$ and $\epsilon > 0$ so that
	\begin{equation}\label{haha5}
	1 + \langle\zeta,\omega\rangle \ge C\sum_{i\in I}|\omega_i - \overline\zeta_i|^2 + \frac1p\sum_{i\notin I}|\omega_i|^p\ \ \ \mbox{whenever }\ \|\omega - \overline\zeta\|\le \epsilon\ \ {and}\ \ \|\omega\|_p = 1,
	\end{equation}
	where $I = \{i:\; \overline\zeta_i \neq 0\}$. Furthermore, for any $\omega$ satisfying $\|\omega\|_p \le 1$, it holds that $\langle\zeta,\omega\rangle \ge -1$, with the equality holding if and only if $\omega = \overline\zeta$.
\end{lemma}
\begin{proof}
	It is easy to check that $\|\overline\zeta\|_p=1$. Next, for each $i\in I$, by considering the Taylor series at $\overline\zeta_i$ of the function $t\mapsto |t|^p$, we see that
	\begin{equation*}
	|\omega_i|^p\! = \! |\overline\zeta_i|^p + p\,{\rm sgn}(\overline\zeta_i)|\overline\zeta_i|^{p-1}(\omega_i - \overline\zeta_i) + \textstyle{\frac{p(p-1)}2}|\overline\zeta_i|^{p-2}(\omega_i - \overline\zeta_i)^2 + O(|\omega_i - \overline\zeta_i|^3)\ \, \mbox{as}\ \, \omega_i \to \overline\zeta_i.
	\end{equation*}
	In particular, there exist $c_i > 0$ and $\epsilon_i > 0$ so that
	\begin{equation}\label{ineq1}
	|\omega_i|^p \ge |\overline\zeta_i|^p + p\,{\rm sgn}(\overline\zeta_i)|\overline\zeta_i|^{p-1}(\omega_i - \overline\zeta_i) + c_i(\omega_i - \overline\zeta_i)^2\ \ \mbox{whenever}\ |\omega_i - \overline\zeta_i|\le \epsilon_i.
	\end{equation}
	Let $\epsilon:= \min_{i\in I}\epsilon_i$. Then for any $\omega\in \RR^n$ satisfying $\|\omega-\overline\zeta\|\le\epsilon$ and $\|\omega\|_p = 1$, we have
	\begin{align*}
	1 &= \|\omega\|_p^p = \sum_{i\notin I}|\omega_i|^p + \sum_{i\in I}|\omega_i|^p\\
	& \overset{\rm (a)}\ge \sum_{i\notin I}|\omega_i|^p + \sum_{i\in I}\left[|\overline\zeta_i|^p + p\,{\rm sgn}(\overline\zeta_i)|\overline\zeta_i|^{p-1}(\omega_i - \overline\zeta_i) + c_i(\omega_i - \overline\zeta_i)^2\right]\\
	& = \sum_{i\notin I}|\omega_i|^p + \sum_{i\in I}\left[|\overline\zeta_i|^p + p\,{\rm sgn}(\overline\zeta_i)|\overline\zeta_i|^{p-1}\omega_i - p|\overline\zeta_i|^p + c_i(\omega_i - \overline\zeta_i)^2\right]\\
	& \overset{\rm (b)}= \sum_{i\notin I}|\omega_i|^p + \sum_{i\in I}\left[|\overline\zeta_i|^p - p\zeta_i\omega_i - p|\overline\zeta_i|^p + c_i(\omega_i - \overline\zeta_i)^2\right]\\
	& \overset{\rm (c)}= \sum_{i\notin I}|\omega_i|^p + 1 - p\langle\zeta,\omega\rangle - p + \sum_{i\in I}c_i(\omega_i - \overline\zeta_i)^2,
	\end{align*}
	where (a) follows from \eqref{ineq1}, (b) holds since ${\rm sgn}(\overline\zeta_i)|\overline\zeta_i|^{p-1} = -{\rm sgn}(\zeta_i)|\zeta_i|^{(q-1)(p-1)} = -{\rm sgn}(\zeta_i)|\zeta_i|=-\zeta_i$, (c) holds because $\sum_{i\in I}|\overline\zeta_i|^p = \|\overline\zeta\|_p^p = 1$. Rearranging terms in the above display, we see that \eqref{haha5} holds with $C = \frac1p\min_{i\in I}c_i$.
	
	Finally, we see from the H\"{o}lder inequality and the fact $\|\zeta\|_q = 1$ that $\langle\zeta,\omega\rangle \ge -1$ whenever $\|\omega\|_p \le 1$.
	We now discuss the equality case. It is clear that if $\omega = \overline\zeta$, then $\langle\zeta,\omega\rangle = -1$. Conversely, suppose that $\langle\zeta,\omega\rangle = -1$ and $\|\omega\|_p\le 1$. Then we have
	\[
	1 = |\langle\zeta,\omega\rangle|\le \|\zeta\|_q\|\omega\|_p = \|\omega\|_p\le  1.
	\]
	Thus, the H\"{o}lder's inequality holds as an equality and we have $\|\omega\|_p = 1$. This means that there exists $c > 0$ so that $|\omega_i|^p=c|\zeta_i|^q$ for all $i$. Summing both sides of this equality for all $i$ and invoking $\|\omega\|_p =\|\zeta\|_q= 1$, we see immediately that $c = 1$ and hence $|\omega_i|^p=|\zeta_i|^q$ for all $i$. Consequently,
	\begin{equation}\label{wformula}
	\omega_i = {\rm sgn}(\omega_i)|\zeta_i|^\frac{q}p = {\rm sgn}(w_i)|\zeta_i|^{q-1}.
	\end{equation}
	Plugging the above relation into $\langle\zeta,\omega\rangle = -1$ yields
	\begin{equation*}
	-1 = \sum_{i=1}^n \zeta_i\omega_i =  \sum_{i=1}^n {\rm sgn}(\omega_i)\zeta_i\,|\zeta_i|^{q-1}\ge -\sum_{i=1}^n |\zeta_i|^q=-1.
	\end{equation*}
	Thus, we must indeed have ${\rm sgn}(\omega_i) = -{\rm sgn}(\zeta_i)$ whenever $\zeta_i\neq 0$.\footnote{Since $|\zeta_i| = |\omega_i|$ for all $i$, $\zeta_i\neq 0$ is the same as $\omega_i\neq 0$.} Combining this with \eqref{wformula}, we conclude that $\omega = \overline\zeta$. This completes the proof.
\end{proof}

\subsection{Facial structure of $\pK$ for $n\ge 2$ and $p \in (1,\infty) $}\label{sec:pface}
The $p$-cones $\pK$ for $p \in (1,\infty)$ are \emph{strictly convex}, i.e., all faces are either $\{0\}$, $\pK$ or extreme rays (one-dimensional faces).
In this subsection, we characterize all the faces of $\pK$ in terms of the corresponding exposing hyperplanes.

Let $z\in \qK$, so $\pK \cap \{z\}^\perp$ is a face of $\pK$ and, because $\pK$ is facially exposed, all faces arise in this fashion. If
$z \in \reInt \qK$ or $z  = 0$ we have
that $\pK \cap \{z\}^\perp$ is $\{0\}$ (see \eqref{eq:rint_dual}) or $\pK$, respectively.

Next, suppose that $z\in \bd\qK\backslash\{0\}$. Then $z = (z_0,\ttz)$ with $z_0 = \|\ttz\|_q > 0$. Now, $x\in \{z\}^\perp$ if and only if
\[
x_0z_0 + \langle\ttz,\ttx\rangle= 0.
\]
Suppose also that $x\in \pK\backslash\{0\}$. Then $x_0> 0$ and the above display is equivalent to
\[
1 + \left\langle z_0^{-1}\ttz,x_0^{-1}\ttx\right\rangle = 0.
\]
Notice that $\left\|z_0^{-1}\ttz\right\|_q = 1$ and $\left\|x_0^{-1}\ttx\right\|_p \le 1$. An application of Lemma~\ref{keylemma} with $\zeta = z_0^{-1}\ttz$ then shows that
\[
x_0^{-1}\ttx = -{\rm sgn}\left(z_0^{-1}\ttz\right)\circ\left|z_0^{-1}\ttz\right|^{q-1} = -{\rm sgn}\left(\ttz\right)\circ|z_0^{-1}\ttz|^{q-1},
\]
where the last equality holds as $z_0 > 0$.
Thus, it follows from the above displays that
\begin{equation}\label{pconeface}
\Fhatz:= \{z\}^{\perp}\cap \pK = \left\{t f \mid t\ge 0\right\},\ \ \mbox{where } f:= \begin{bmatrix}
1\\
-{\rm sgn}\left(\ttz\right)\circ|z_0^{-1}\ttz|^{q-1}
\end{bmatrix}.
\end{equation}

\subsection{One-step facial residual functions for the faces of  $\pK$ when $n\ge 2$ and $p \in (1,\infty) $}\label{sec:frf}
We recall that our goal is to compute error bounds through Theorem~\ref{theo:err}. In order to do so, we need to compute the {\oneFRFs}s  for the faces of $\pK$.
Let $\stdFace \face \pK$ be a face and $z \in \stdFace^*$. First, we take care of some trivial cases.
\begin{itemize}
	\item If $\stdFace = \{0\}$, then, from Definition~\ref{def:onefrf}, $\psi_{\stdFace,z}(\epsilon,t) \coloneqq \epsilon$ is a {\oneFRFs} for $\stdFace$ and $z$.
	\item If $\stdFace$ is an extreme ray, then
	$\stdFace \cap \{z\}^\perp$ is either $\stdFace$ or $\{0\}$. The latter happens if and only if $z \in \reInt \stdFace^*$ (see \eqref{eq:rint_dual}), so that Lemma~\ref{lem:frfzero} is applicable to $\stdFace$ and $z$.
	Therefore, in both cases there exists $\kappa > 0$ such that $\psi_{\stdFace,z}(\epsilon,t) \coloneqq \kappa \epsilon$ is a {\oneFRFs} for $\stdFace$ and $z$.
	\item If $\stdFace = \pK$ and
	$z \in \reInt \qK$, then we also have $\kappa \epsilon$ as a {\oneFRFs}, by Lemma~\ref{lem:frfzero}.
\end{itemize}




%

This section is focused on the
remaining nontrivial case where $\stdFace = \stdCone$  and
$\stdCone \cap \{z\}^\perp$ is an extreme ray, which happens if and only if $z\in \bd\qK\backslash\{0\}$. With that in mind, we define
\begin{equation}\label{p-alpha}
J_z:= \{i\mid \ttz_i \neq 0\}\ \ \ {\rm and}\ \ \
\alpha_z:= \begin{cases}
\frac12 & {\rm if}\ |J_z| = n,\\
\frac1p & {\rm if}\ |J_z| = 1 \ {\rm and}\ p < 2,\\
\min\left\{\frac12,\frac1p\right\} & {\rm otherwise},
\end{cases}
\end{equation}
where $|J_z|$ is the number of elements of $J_z$.
Then we have the following result.
\begin{theorem}\label{thm:pconeface}
Let $n\ge 2$ and $p, q\in (1,\infty)$ be such that $\frac1p+\frac1q=1$.
	Let $z\in \bd\qK\backslash\{0\}$ and let $\Fhatz := \{z\}^{\perp}\cap \pK$. Let $\eta > 0$, $\alpha_z$ be as in \eqref{p-alpha}, and define
	\begin{equation}\label{gammapcone}
	\gamma_{z,\eta} \!:=\! \inf_v\bigg\{\frac{\|v - w\|^{\alpha_z}}{\|u - w\|}\,\bigg|\, v\in \bd \pK\cap B(\eta)\backslash \Fhatz,\, w \!=\! P_{\{z\}^\perp}v,\, u \!=\! P_{\Fhatz}w,\, u \!\neq\! w\bigg\}.
	\end{equation}
	Then it holds that $\gamma_{z,\eta}\in (0,\infty]$ and that
	\[
	\dist(x,\Fhatz)\le \max\{2\eta^{1-\alpha_z},2\gamma_{z,\eta}^{-1}\}\cdot\dist(x,\pK)^{\alpha_z}\ \ \mbox{whenever $x\in \{z\}^\perp\cap B(\eta)$.}
	\]
\end{theorem}
\begin{proof}
	If $\gamma_{z,\eta} = 0$, in view of Lemma~\ref{lem:infratio}, there exist $\widehat v\in \Fhatz$ and a sequence $\{v^k\}\subset \bd \pK\cap B(\eta)\backslash \Fhatz$ such that
	\begin{equation}\label{forcontradiction2}
	\lim_{k\to \infty}v^k = \lim_{k\to \infty}w^k = \widehat v\ \ {\rm and}\ \ \lim_{k\to\infty}\frac{\|w^k - v^k\|^{\alpha_z}}{\|w^k - u^k\|} = 0,
	\end{equation}
	where $w^k = P_{\{z\}^\perp}v^k$, $u^k = P_{\Fhatz}w^k$ and $u^k\neq w^k$. Since $v^k\notin \Fhatz$ and $v^k\in \bd\pK$, we must have $v_0^k=\|\ttv^k\|_p > 0$. By passing to a further subsequence if necessary, we may then assume that
	\begin{equation}\label{hahahah}
	(v_0^k)^{-1}\ttv^k \to \xi
	\end{equation}
	for some $\xi$ satisfying $\|\xi\|_p=1$. Next, applying Lemma~\ref{keylemma} with $\zeta := z_0^{-1}\ttz$, we have
	\begin{equation*}
	\overline\zeta := -{\rm sgn}\left(z_0^{-1}\ttz\right)\circ|z_0^{-1}\ttz|^{q-1} = -{\rm sgn}\left(\ttz\right)\circ|z_0^{-1}\ttz|^{q-1},\ \ \|\overline\zeta\|_p =1,
	\end{equation*}
	and there exist $C>0$ and $\epsilon > 0$ so that \eqref{haha5} holds. Moreover, one can see that the $J_z$ in \eqref{p-alpha} equals the $I$ in Lemma~\ref{keylemma}.
	We consider two cases.
	\begin{enumerate}[(I)]
		\item\label{pconeproofcase1} $\xi = \overline\zeta$;
		\item\label{pconeproofcase2} $\xi\neq \overline\zeta$.
	\end{enumerate}
	
	\ref{pconeproofcase1}: Suppose that $\xi = \overline\zeta$. Then $\lim\limits_{k\to\infty} \langle(v_0^k)^{-1}\ttv^k,\overline\zeta\rangle = \|\overline\zeta\|^2 >0$. Thus, for all sufficiently large $k$, we have, upon using the definition of $f$ in \eqref{pconeface}, that
	\[
	(v_0^k)^{-1}\langle v^k,f\rangle =
	\left\langle \begin{bmatrix}
	1\\ (v_0^k)^{-1}\ttv^k
	\end{bmatrix},\begin{bmatrix}
	1\\ -{\rm sgn}\left(\ttz\right)\circ|z_0^{-1}\ttz|^{q-1}
	\end{bmatrix} \right\rangle = 1 + \langle(v_0^k)^{-1}\ttv^k,\overline\zeta\rangle\! \ge 1 + \frac{\|\overline\zeta\|^2}2.
	\]
	Consequently, we have $\langle v^k,f\rangle > 0$ for sufficiently large $k$. Thus, if we let
	\[
	Q := I - \frac{zz^T}{\|z\|^2} - \frac{ff^T}{\|f\|^2},
	\]
	then using Lemma~\ref{lem:dist} (with $f$ as in \eqref{pconeface}), we deduce that for all sufficiently large $k$,
	\begin{equation}\label{rel3}
	\begin{aligned}
	\|u^k-w^k\| & = \|Qv^k\| = v_0^k
	\left\|Q\begin{bmatrix}
	1\\ (v_0^k)^{-1}\ttv^k
	\end{bmatrix}\right\| \overset{\rm (a)}= v_0^k
	\left\|Q\begin{bmatrix}
	1\\ (v_0^k)^{-1}\ttv^k
	\end{bmatrix} - Q\begin{bmatrix}
	1\\ \overline\zeta
	\end{bmatrix}\right\|\\
	& \le v_0^k\|(v^k_0)^{-1}\ttv^k - \overline \zeta\|,
	\end{aligned}
	\end{equation}
	where (a) holds because $f = \begin{bmatrix}
	1 & \overline\zeta^T
	\end{bmatrix}^T$ and $Qf = 0$. Moreover, if it happens that $|J_z|=1$, say, $J_z = \{i_0\}$, then $z_0 = |\ttz_{i_0}|\neq 0$ and $f_0 = |\ttf_{i_0}|\neq 0$, and we have $\ttz_i = \ttf_i = 0$ for all $i\neq i_0$ and $\ttz_{i_0} = -\ttf_{i_0}$. Then a direct computation shows that $Q$ is diagonal with $Q_{00} = Q_{i_0i_0} = 0$, and $Q_{ii}=1$ otherwise. Thus, we have the following \emph{refined} estimate on $\|u^k - w^k\|$ for all sufficiently large $k$ when $|J_z|=1$:
\begin{equation}\label{rel3_2}
	\begin{aligned}
	\|u^k-w^k\| & = \|Qv^k\| = v_0^k
	\left\|Q\begin{bmatrix}
	1\\ (v_0^k)^{-1}\ttv^k
	\end{bmatrix}\right\| = v_0^k\sqrt{\sum_{i\notin J_z}|(v^k_0)^{-1}\ttv^k_i|^2}.
\end{aligned}
	\end{equation}

Next, in view of \eqref{haha5} and \eqref{hahahah} and recalling that $\|(v^k_0)^{-1}\ttv^k\|_p = 1$, we have for all sufficiently large $k$ that
\begin{equation}\label{rel1-1}
  1 + \langle\zeta,(v^k_0)^{-1}\ttv^k\rangle \ge C\sum_{i\in I}|(v^k_0)^{-1}\ttv^k_i - \overline\zeta_i|^2 + \frac1p\sum_{i\notin I}|(v^k_0)^{-1}\ttv^k_i|^p.
\end{equation}
If $|J_z|\neq 1$ or $p \ge 2$, then we see from \eqref{hahahah} and \eqref{rel1-1} that for all sufficiently large $k$
	\begin{equation}\label{rel1}
	\begin{aligned}
	&1 + \langle\zeta,(v^k_0)^{-1}\ttv^k\rangle\\
     \overset{\rm (a)}\ge &\min\left\{C,\frac1p\right\}\left(\sum_{i\in I}|(v^k_0)^{-1}\ttv^k_i - \overline\zeta_i|^{1/\alpha_z} + \sum_{i\notin I}|(v^k_0)^{-1}\ttv^k_i|^{1/\alpha_z}\right)\\
	 =  &\min\left\{C,\frac1p\right\}\|(v^k_0)^{-1}\ttv^k - \overline\zeta\|_{1/\alpha_z}^{1/\alpha_z} \ge \, C_1 \|(v^k_0)^{-1}\ttv^k - \overline\zeta\|^{1/\alpha_z},
	\end{aligned}
	\end{equation}
	where (a) follows from the definition of $\alpha_z$ and the observation that $J_z = I$, and the last inequality follows from the equivalence of norms in finite-dimensional Euclidean spaces, and $C_1$ is a constant that depends only on $C$, $p$, $\alpha_z$ and the dimension $n$. Thus, combining \eqref{rel1} with Lemma~\ref{lem:dist} and recalling that $\zeta = z_0^{-1}\ttz$, we see that for all sufficiently large $k$,
	\begin{equation}\label{rel2}
	\begin{aligned}
	\|v^k - w^k\| & = \frac{|z_0v_0^k + \langle \ttz,\ttv^k\rangle|}{\|z\|} = \frac{z_0v_0^k|1 + \langle \zeta,(v_0^k)^{-1}\ttv^k\rangle|}{\|z\|}\\
	& \ge \frac{C_1 z_0v_0^k}{\|z\|}\|(v^k_0)^{-1}\ttv^k - \overline \zeta\|^{1/\alpha_z}\ge \frac{C_1 z_0(v_0^k)^{1-1/\alpha_z}}{\|z\|}\|u^k - w^k\|^{1/\alpha_z},
	\end{aligned}
	\end{equation}
where the last inequality follows from \eqref{rel3}.

 On the other hand, if $|J_z|=1$ and $p < 2$, we have $\alpha_z = 1/p$. We can deduce from \eqref{rel1-1} and Lemma~\ref{lem:dist} that for all sufficiently large $k$,
	\begin{equation}\label{rel2-2}
	\begin{aligned}
	\|v^k - w^k\| & = \frac{z_0v_0^k|1 + \langle \zeta,(v_0^k)^{-1}\ttv^k\rangle|}{\|z\|}\ge \frac{z_0v_0^k}{p\|z\|}\sum_{i\notin I}|(v^k_0)^{-1}\ttv^k_i|^p\\
&\overset{\rm (a)}\ge  \frac{z_0v_0^k}{p\|z\|}\left(\sqrt{\sum_{i\notin I}|(v^k_0)^{-1}\ttv^k_i|^{2}}\right)^p= \frac{z_0(v_0^k)^{1-1/\alpha_z}}{p\|z\|}\|u^k - w^k\|^{1/\alpha_z},
	\end{aligned}
	\end{equation}
where (a) holds because $p < 2$ so that $p$-norm majorizes $2$-norm, and we used \eqref{rel3_2} and the fact that $\alpha_z = 1/p$ for the last equality.

Combining \eqref{rel2} and \eqref{rel2-2}, we see that there exists $C_2 > 0$ such that for all sufficiently large $k$,
	\[
	\begin{aligned}
	\|u^k-w^k\|\le C_2\cdot(v_0^k)^{1-\alpha_z}\|v^k - w^k\|^{\alpha_z}\le C_2 \eta^{1-\alpha_z}\|v^k - w^k\|^{\alpha_z},
	\end{aligned}
	\]
	where the last inequality holds because $v^k\in B(\eta)$. This contradicts \eqref{forcontradiction2} and hence case \ref{pconeproofcase1} cannot happen.
	
	\ref{pconeproofcase2}: In this case, $\xi\neq \overline\zeta$. Since $\|\xi\|_p = 1$, we see from Lemma~\ref{keylemma} that $1 + \langle \zeta,\xi\rangle > 0$. Thus, in view of \eqref{hahahah}, there exists $\ell>0$ such that
	\[
	1 + \langle \zeta,(v_0^k)^{-1}\ttv^k\rangle \ge \ell> 0\  \mbox{ for all large }k.
	\]
	Using this together with Lemma~\ref{lem:dist} and the fact that $\zeta = z_0^{-1}\ttz$, we deduce that for these $k$,
	\begin{equation*}
	\begin{aligned}
	\|v^k - w^k\| & = \frac{|z_0v_0^k + \langle \ttz,\ttv^k\rangle|}{\|z\|} = \frac{z_0v_0^k|1 + \langle \zeta,(v_0^k)^{-1}\ttv^k\rangle|}{\|z\|}\\
	& \ge \frac{\ell z_0v_0^k}{\|z\|} =  \frac{\ell z_0\|\ttv^k\|_p}{\|z\|} \overset{\rm (a)}\ge \frac{\ell z_0\|v^k\|_p}{2\|z\|} \overset{\rm (b)}\ge \frac{C_3\ell z_0\|v^k\|}{2\|z\|}\overset{\rm (c)}\ge \frac{C_3\ell z_0}{2\|z\|}\|u^k - w^k\|,
	\end{aligned}
	\end{equation*}
	where (a) follows from the triangle inequality and the fact that $v_0^k = \|\ttv^k\|_p$, (b) holds for some constant $C_3 > 0$ that only depends on $n$ and $p$, (c) follows from the fact that $\|u^k - w^k\| = {\rm dist}(w^k,\Fhatz)\le \|w^k\|\le \|v^k\|$ (which holds because of the properties of projections). The above display contradicts \eqref{forcontradiction2} and hence case \ref{pconeproofcase2} also cannot happen.
	
	Thus, we have $\gamma_{z,\eta} \in (0,\infty]$; the desired error bound follows from Theorem~\ref{thm:1dfacesmain}.
\end{proof}

Using Theorem~\ref{thm:pconeface} together with Lemma \ref{lem:facialresidualsbeta} and recalling that an upper bound to a {\oneFRFs} is also a {\oneFRFs}, we obtain the following facial residual function for $\pK$.
\begin{corollary}\label{col:pfrf}
Let $n\ge 2$ and $p, q\in (1,\infty)$ be such that $\frac1p+\frac1q=1$.
	Let $z\in \bd\qK\backslash\{0\}$ and let $\Fhatz := \{z\}^{\perp}\cap \pK$. Let $\alpha_z$ be as in \eqref{p-alpha} and let $\gamma_{z,t}$ be as in \eqref{gammapcone}.\footnote{We set $\gamma_{z,0} = \infty$.} Let $\kappa = \max\{1,1/\norm{z}\}$. Then the function $\psi_{\stdCone,z}:\RR_+\times \RR_+\to \RR_+$ given by
	\begin{equation*}
	\psi_{\stdCone,z}(\epsilon,t) := \kappa \epsilon + \max\{2t^{1-\alpha_z},2\gamma_{z,t}^{-1}\}(\kappa + 1)^{\alpha_z}\epsilon^{\alpha_z}
	\end{equation*}
	is a {\oneFRFs} residual function for $\pK$ and $z$.
\end{corollary}

Next, we will prove that the optimality criterion \eqref{G1} is satisfied for $p$-cones when
$\frakg = |\cdot|^{\alpha_z}$, with $\alpha_z$ as in \eqref{p-alpha}. For that, we need two lemmas that assert
the existence of functions $\epsilon\mapsto w_\epsilon$ having certain desirable properties.

\begin{lemma}[A $w_\epsilon$ of order $\frac1p$ when $|J_z|<n$]\label{lem:Jzneqfull}
Let $n\ge 2$ and $p, q\in (1,\infty)$ be such that $\frac1p+\frac1q=1$.
Let $z\in \bd\qK\backslash\{0\}$ be such that
$J_z \neq \{1,2,\ldots,n\}$, where $J_z$ is as in \eqref{p-alpha}. Let $\Fhatz := \{z\}^{\perp}\cap \pK$ and $f\in \Fhatz\setminus \{0\}$ be defined as in \eqref{pconeface}. Then there exists a continuous function
$w:(0,1] \to \{z\}^\perp\backslash \Fhatz $ such that
\[
\underset{\epsilon \downarrow 0}{\lim}\, w_\epsilon = f\ \ \ {\rm and}\ \ \ \limsup_{\epsilon \downarrow 0} \frac{\dist(w_\epsilon,\pK)^{\frac{1}{p}}}{\dist(w_\epsilon,\Fhatz)}	< \infty.
\]	
\end{lemma}
\begin{proof}
Fix any $j\in \{1,2,\ldots,n\}\backslash J_z$ and define
	\[
	\overline\zeta:= -{\rm sgn}(\ttz)\circ |z_0^{-1}\ttz|^{q-1}.
	\]
	Then $\|\overline\zeta\|_p = 1$ and $\overline\zeta_i\neq 0$ if and only if $i\in J_z$.
	Moreover, we have $f = \begin{bmatrix}1&  \overline\zeta^T\end{bmatrix}^T$.
	Define the (bounded) continuous function $w: (0,1]\to \RR^{n+1}$ by
	\[
	(w_\epsilon)_i = \begin{cases}
	f_i & \text{ if } i\neq j\\
	\epsilon & \text{ if } i = j
	\end{cases}.
	\]
	Then $\langle z,w_\epsilon\rangle = 0$ for every $\epsilon \in (0,1]$ and $w_\epsilon \to f\in \Fhatz\backslash\{0\}$.
	Now, observe from $\|\overline\zeta\|_p = 1$ that the image of the function $y:(0,1]\to \RR^{n+1}$ defined by
	\[
	(y_\epsilon)_0 = (1 + \epsilon^p)^\frac1p,\qquad \bar y_\epsilon = \bar w_\epsilon.
	\]
	is entirely contained in $\pK$. Hence, we have
	\begin{equation*}
	{\rm dist}(w_\epsilon,\pK)\le \|w_\epsilon - y_\epsilon\| = (1 + \epsilon^p)^\frac1p - 1  \le \frac1p \epsilon^p,
	\end{equation*}
	where the last inequality follows from the concavity of $t\mapsto t^\frac1p$ and the supgradient inequality.
	In addition, we can also deduce from \eqref{pconeface} and the definition of $\overline\zeta$ that ${\rm dist}(w_\epsilon,\Fhatz) = \epsilon > 0$.
	Thus
	$
	\limsup\limits_{\epsilon \downarrow 0}\frac{{\rm dist}(w_\epsilon,\pK)^\frac1p}{{\rm dist}(w_\epsilon,\Fhatz)} \le \left(\frac1p\right)^\frac1p.
	$
\end{proof}

\begin{lemma}[A $w_\epsilon$ of order $\frac12$ when $|J_z|\ge 2$]\label{lem:Jzgeq2}
Let $n\ge 2$ and $p, q\in (1,\infty)$ be such that $\frac1p+\frac1q=1$.
	Let $z\in \bd\qK\backslash\{0\}$ be such that
	$|J_z|\ge 2$, where $J_z$ is as in \eqref{p-alpha}. Let $\Fhatz := \{z\}^{\perp}\cap \pK$ and $f\in \Fhatz\setminus \{0\}$ be defined as in \eqref{pconeface}. Then there exists a continuous function
	$w:(0,1] \to \{z\}^\perp\backslash \Fhatz$ such that
	\[
	\underset{\epsilon \downarrow 0}{\lim}\, w_\epsilon = f\ \ \ {\rm and}\ \ \ \limsup_{\epsilon \downarrow 0} \frac{\dist(w_\epsilon,\pK)^{\frac{1}{2}}}{\dist(w_\epsilon,\Fhatz)}	< \infty.
	\]	
\end{lemma}
\begin{proof}
Without loss of generality, we assume that $n \in J_z$ so that $\ttz_n\neq 0$; by symmetry, we will assume that $\bar z_n < 0$. Define
\[
\overline\zeta:= -{\rm sgn}(\ttz)\circ |z_0^{-1}\ttz|^{q-1}.
\]
Then $\|\overline\zeta\|_p = 1$, $\overline\zeta_n> 0$ and $f = \begin{bmatrix}1&  \overline\zeta^T\end{bmatrix}^T$. Define the (bounded) continuous function $w: (0,\min\{|\ttz_n|\cdot\overline\zeta_n, z_0\})\to \RR^{n+1}$ by
\[
(w_\epsilon)_i = \begin{cases}
1-\epsilon z_0^{-1} & \text{ if } i = 0\\
\overline\zeta_n + \epsilon \ttz_n^{-1} & \text{ if } i = n \\
\overline\zeta_i & \text{ otherwise}.
\end{cases}
\]
%
Then $\langle z,w_\epsilon\rangle = 0$ and $w_\epsilon \to f\in \Fhatz\backslash\{0\}$.
Now, observe that the image of the function $y:(0,\min\{|\ttz_n|\cdot\overline\zeta_n, z_0\})\to \RR^{n+1}$ defined by
\[
(y_\epsilon)_0 = \left[\sum_{i=1}^{n-1}|\overline\zeta_i|^p + (\overline\zeta_n + \epsilon \ttz_n^{-1})^p\right]^\frac1p,\ \qquad \bar y_\epsilon = \bar w_\epsilon
\]
is contained in $\pK$. Hence, we have
\begin{align}
{\rm dist}(w_\epsilon,\pK)&\le \|w_\epsilon - y_\epsilon\| = \left|\left[\sum_{i=1}^{n-1}|\overline\zeta_i|^p + (\overline\zeta_n + \epsilon \ttz_n^{-1})^p\right]^\frac1p - 1 + \epsilon z_0^{-1}\right|\notag\\
& = \left|\left[\sum_{i=1}^{n-1}|\overline\zeta_i|^p + \overline\zeta_n^p + p\epsilon \overline\zeta_n^{p-1} \ttz_n^{-1} + O(\epsilon^2)\right]^\frac1p - 1 + \epsilon z_0^{-1}\right|\notag\\
& \overset{\rm (a)}= \left|\left[1 + p\epsilon \overline\zeta_n^{p-1} \ttz_n^{-1} + O(\epsilon^2)\right]^\frac1p - 1 + \epsilon z_0^{-1}\right|\notag\\
& = \left|1 + \epsilon \overline\zeta_n^{p-1} \ttz_n^{-1} - 1 + \epsilon z_0^{-1} + O(\epsilon^2)\right| = O(\epsilon^2)\ \ {\rm as}\ \ \epsilon\downarrow 0, \label{distpcone}
\end{align}
where (a) holds because $\|\overline\zeta\|_p = 1$, and the last equality holds because
\[
\overline\zeta_n^{p-1} = |z_0^{-1}\ttz_n|^{(q-1)(p-1)} = |z_0^{-1}\ttz_n| = - z_0^{-1}\ttz_n
\]
as $\overline\zeta_n > 0$ and $\ttz_n < 0$. We next estimate ${\rm dist}(w_\epsilon,\Fhatz)$.
Since $w_\epsilon \to f$, we must have $\langle w_\epsilon,f\rangle > 0$ for all sufficiently small $\epsilon$. Thus, from the definition of $\Fhatz$ and Lemma~\ref{lem:dist}, we see that for these $\epsilon$,
\[
{\rm dist}(w_\epsilon,\Fhatz)^2 = \left\|w_\epsilon - \frac{\langle w_\epsilon, f\rangle}{\|f\|^2}f\right\|^2 = \|w_\epsilon\|^2 - \frac{(\langle w_\epsilon, f\rangle)^2}{\|f\|^2}.
\]
Now, a direct computation shows that
\[
\begin{aligned}
\|w_\epsilon\|^2 & = (1 -\epsilon z_0^{-1})^2 + \sum_{i=1}^{n-1}|\overline\zeta_i|^2 + (\overline\zeta_n + \epsilon \ttz_n^{-1})^2\\
& = 1 + \sum_{i=1}^{n-1}|\overline\zeta_i|^2 + \overline\zeta_n^2 + 2\epsilon(\ttz_n^{-1}\overline\zeta_n - z_0^{-1}) + \epsilon^2 \ttz_n^{-2} + \epsilon^2 z_0^{-2}\\
& = \|f\|^2 + 2\epsilon(\ttz_n^{-1}\overline\zeta_n - z_0^{-1}) + \epsilon^2 \ttz_n^{-2} + \epsilon^2 z_0^{-2},
\end{aligned}
\]
where the last equality follows from the definition of $f$ in \eqref{pconeface}. In addition, we have
\[
\begin{aligned}
& (\langle w_\epsilon, f\rangle)^2 = \left( 1 -\epsilon z_0^{-1} + \sum_{i=1}^{n-1}|\overline\zeta_i|^2 + \overline\zeta_n^2 + \epsilon \ttz_n^{-1}\overline\zeta_n\right)^2\\
& = \left[ \|f\|^2  + \epsilon (\ttz_n^{-1}\overline\zeta_n - z_0^{-1})\right]^2= \|f\|^4 + 2\epsilon\|f\|^2 (\ttz_n^{-1}\overline\zeta_n - z_0^{-1}) + \epsilon^2(\ttz_n^{-1}\overline\zeta_n - z_0^{-1})^2.
\end{aligned}
\]
Thus, it holds that for all sufficiently small $\epsilon$,
\begin{align}\label{lowerlowerbound}
{\rm dist}(w_\epsilon,\Fhatz)^2 & = \epsilon^2\left(\ttz_n^{-2} + z_0^{-2} - \frac{(\ttz_n^{-1}\overline\zeta_n - z_0^{-1})^2}{\|f\|^2}\right)\nonumber\\
& \ge \epsilon^2\left(\ttz_n^{-2} + z_0^{-2} - \frac{(-\ttz_n^{-1}\overline\zeta_n + z_0^{-1})^2}{1 + \overline\zeta_n^2}\right).
\end{align}
Since $|J_z|\ge 2$ and $\|\overline\zeta\|_p=1$, we must have $\overline\zeta_n < 1$. Then the Cauchy-Schwarz inequality gives
\[
[\ttz_n^{-2} + z_0^{-2}]\cdot [\overline\zeta_n^2 + 1]> (-\ttz_n^{-1}\overline\zeta_n + z_0^{-1})^2,
\]
and the inequality is {\em strict} because
\[
\frac{\overline\zeta_n}{-\ttz_n^{-1}}\cdot z^{-1}_0 = \frac{|z_0^{-1}\ttz_n|^{q-1}}{|\ttz_n^{-1}z_0|} = |z_0^{-1}\ttz_n|^q = |\overline\zeta_n|^{\frac{q}{q-1}} < 1.
\]
Consequently, we see from \eqref{lowerlowerbound} that there exists $c > 0$ such that ${\rm dist}(w_\epsilon,\Fhatz) \geq c\epsilon$ for all sufficiently small $\epsilon$. Combining this with \eqref{distpcone}, we have
$
\limsup\limits_{\epsilon\downarrow 0}\frac{{\rm dist}(w_\epsilon,\pK)^\frac12}{{\rm dist}(w_\epsilon,\Fhatz)} < \infty.
$
To conclude, we observe that we can perform a change of variable $\hat \epsilon = a \epsilon$ for some $a > 0$ so that $w_{\hat \epsilon}$ is defined for $\hat \epsilon \in (0,1]$ with $w_{\hat \epsilon}\notin \Fhatz$ for all $\hat \epsilon \in (0,1]$.
\end{proof}	

We now have all the tools to prove the following theorem.

\begin{theorem}[The optimality criterion \eqref{G1} is satisfied for $p$-cones]\label{theo:pcone_opt}
Suppose that $n\ge 2$ and $p, q\in (1,\infty)$ are such that $\frac1p+\frac1q=1$.
Let $z\in \bd\qK\backslash\{0\}$ and let $\Fhatz := \{z\}^{\perp}\cap \pK$. Let $\alpha_z$ be as in \eqref{p-alpha}. Then the function $\frakg = |\cdot|^{\alpha_z}$ satisfies the asymptotic optimality criterion \eqref{G1} for $\pK$ and $z$.
\end{theorem}
\begin{proof}
Let $J_z$ be as in \eqref{p-alpha}. We split the proof in a few cases:
\begin{enumerate}[label=(\Roman*)]
	\item\label{JzFullset} $J_z= \{1,2,\ldots,n\}$.
	In this case $\alpha_z =1/2$.
	Since $n\geq 2$, we have $|J_z|\geq 2$ so we can invoke Lemma~\ref{lem:Jzgeq2}, which gives the required function $w$ satisfying \eqref{G1} with $\overline v = f$ in \eqref{pconeface}.
	\item\label{Jzeq1} $|J_z|=1$ and $p < 2$. In this case,
	$\alpha_z = 1/p$. Since $n \geq 2$, we have $J_z \neq \{1,\ldots,n\}$, so that we can invoke Lemma~\ref{lem:Jzneqfull}, which gives the required function $w$ satisfying \eqref{G1} with $\overline v = f$ in \eqref{pconeface}.
	\item\label{Jzeq1b} $|J_z|=1$ and $p \geq 2$. In this case,
	$\alpha_z = 1/p$. Similarly to the previous item, it follows from Lemma~\ref{lem:Jzneqfull}.
	\item\label{Jzother} $2\leq |J_z|< n$. In this case,
	$\alpha_z = \min\{1/2,1/p\}$. Similarly to the previous items, it follows from either Lemma~\ref{lem:Jzneqfull} or Lemma~\ref{lem:Jzgeq2}.	
\end{enumerate}
\end{proof}

Theorem~\ref{theo:pcone_opt} shows an interesting property of $p$-cones: namely, in $3$ dimensions, $n=2$, so that the cases $|J_z| = n$ and $|J_z|\neq 1$ exactly coincide, thereby eliminating the $\min\{\frac12,\frac1p\}$ case in the definition of $\alpha_z$ when $p < 2$. Thus, $p$-cones of dimension $4$ and higher exhibit a greater complexity in their optimal {\oneFRFs}s than those in $\RR^3$.
This difference will be discussed in the next subsection.

\subsubsection{How $p$-cones of dimension 3 differ from those of dimension 4}

The following Examples~\ref{ex:3d4d} and \ref{ex:3d4d_3halves} will, taken together, illustrate how $p$-cones of dimension 3 differ from those of dimension 4 geometrically. Examples~\ref{ex:3d4d} also shows another curiosity: that faces defined very similarly can have different error bounds, depending on the dimension.
As the main goal of this subsection is to provide geometric intuition and we will not prove new results, the style here will be more informal. Although we hope this will be offset by a gain in intuition.

\begin{example}[$p=3$ cones in $3$ and $4$ dimensions]\label{ex:3d4d}
	We consider the example in Figure~\ref{fig:cone_slices}. In order to visualize the 4-dimensional cone $\stdCone_{3}^{3+1}$, we will depict 3-dimensional slices of it (in columns 2 and 3). Those slices are obtained by intersecting the cone with hyperplanes and plotting the intersection of the cone with those hyperplanes. To aid the reader when interpreting this method of visualization, we also depict (in column 1) the 3-dimensional cone $\stdCone_{3}^{2+1}$ using the same approach: we intersect the cone with hyperplanes and plot the intersection of the cone with those hyperplanes. In addition, we uniquely color\footnote{with apologies to color-blind readers, who may instead make use of the axis labels.} each of the hyperplanes, so that the reader can also see how the various slices mutually intersect one another.
	
	For the 3-dimensional cone $\stdCone_{3}^{2+1}$, we consider the face given by $\stdFace' = \stdCone_{3}^{2+1} \cap \{x_1\geq 0\} \cap \color{aqua}\{x\;|\;x_1=x_2\} \color{black}$, and for the 4-dimensional cone $\stdCone_{3}^{3+1}$, we consider the face defined very similarly by $\stdFace = \stdCone_{3}^{3+1} \cap \{x_1\geq 0\} \cap \color{aqua}\{x\;|\;x_1=x_2\} \color{black} \cap \color{magenta} \{x\;|\;x_3=0\} \color{black}$.
	
	When $\stdFace$ is viewed from the perspective of the 3-dimensional slice $\stdCone_{3}^{3+1} \cap \color{magenta} \{x\;|\;x_3=0 \} \color{black}$, the situation seems quite analogous to the face of the 3-dimensional cone $\stdCone_{3}^{2+1}$ given by $\stdFace' = \stdCone_{3}^{2+1} \cap \{x\;|\;x_1\geq0 \} \cap \color{aqua}\{x\;|\;x_1=x_2 \} \color{black}$. Indeed, the geometry depicted in these two images is identical, because the natural embedding $\iota$ of $\stdCone_{3}^{2+1}$ into $\mathbb{R}^4$ yields $\iota: \stdCone_{3}^{2+1} \to \stdCone_{3}^{3+1} \cap \color{magenta}\{x\;|\;x_3=0 \}\color{black}$ and $\iota: \stdFace' \to \stdFace$. However, the appearances are deceptive. For $\stdFace' = \{z'\}^\perp \cap \stdCone_{3}^{2+1}$, it holds that the best possible exponent is $\alpha_z' = \frac12$.
	However, in contradistinction, for $\stdFace = \color{darkbrown}\{z\}^\perp \color{black}\cap \stdCone_{3}^{3+1}$, the best possible exponent is $\alpha_z = \frac13$. The need for this smaller exponent is more apparent when we view $\stdFace$ from the perspective of the 3-dimensional slice $\stdCone_{3}^{3+1} \cap \color{aqua} \{x\;|\;x_1=x_2 \} \color{black}$. Viewed from this perspective, the cone appears to have less curvature---or ``more flatness''---at the site of the face, which would indicate that $\stdCone_{3}^{3+1}$ intersects with the exposing hyperplane $\color{darkbrown}\{z\}^\perp \color{black}$ ``more sharply'' than it would appear to when viewed from the perspective of the slice $\stdCone_{3}^{3+1} \cap \color{magenta} \{x\;|\;x_3=0 \}$.
	
	What we have observed here is that the slice in which we view the face can change what exponent appears to be suitable, because it changes how curved the cone $\stdCone_{3}^{3+1}$ appears to be locally at the face $\stdFace$. We can better understand the reason for this if we consider the slice $\stdCone_{3}^{3+1} \cap \color{forest}\{x\;|\;x_0=1\}$. \color{black}
	
	If we further intersect $\stdCone_{3}^{3+1} \cap \color{forest}\{x\;|\;x_0=1\}$ with the hyperplane $\color{magenta}\{x\;|\; x_3=0 \}$, then $\{z\}^\perp$ appears to intersect $\stdCone_{3}^{3+1}$ with the more rounded local curvature of the top image in column 3. However, if we instead intersect $\stdCone_{3}^{3+1} \cap \color{forest}\{x\;|\;x_0=1\}$ with the hyperplane $\color{aqua}\{x\;|\; x_1=x_2 \}$, then $\{z\}^\perp$ appears to intersect $\stdCone_{3}^{3+1}$ with the flatter local curvature of the middle image in column 3.
	
	Let us describe what we have just observed in a slightly different way. If we take a continuous function $w:(0,1] \to  \color{magenta}\{x\;|\;x_3=0\}\color{black}\cap \color{darkbrown}\{z\}^\perp \color{black}\cap{\color{forest} \{x\;|\;x_0=1 \}}
\backslash \stdFace$ with $\lim_{\epsilon\downarrow 0}w_\epsilon = f$, then we will \textit{not} obtain a counterexample to the claim that
	$$
	0 < \liminf_{\epsilon \downarrow 0} \frac{\sqrt{\dist(w_\epsilon,\stdCone)}}{\dist(w_\epsilon,\stdFace)}.
	$$
	On the other hand, if we take a properly constructed continuous function $w:(0,1] \to \color{aqua}\{x\;|\;x_1=x_2\}\color{black}\cap \color{darkbrown}\{z\}^\perp \color{black}\cap{\color{forest} \{x\;|\;x_0=1 \}}\backslash \stdFace$ with $\lim_{\epsilon\downarrow 0}w_\epsilon = f$, then we \textit{will} obtain such a counterexample.

	Using the slice $\stdCone_{3}^{3+1} \cap \color{forest}\{x\;|\;x_0=1\}$, we can see that the (exposed) 1-dimensional faces that require exponent $\alpha_z = \frac13$ are exactly those that are contained in any of the slices $\stdCone_{3}^{3+1} \cap \color{wheat} \{x\;|\;x_1=0 \}$ or $\stdCone_{3}^{3+1} \cap \color{aqua} \{x\;|\;x_2=0 \}$ or $\stdCone_{3}^{3+1} \cap \color{magenta} \{x\;|\;x_3=0 \}$. For all other (exposed) 1-dimensional faces, we will have that $\alpha_z = \frac12$ will be a suitable exponent. This is consistent with what we know from \eqref{p-alpha}, because such faces will have exposing vectors $z$ that satisfy $|J_z| = n$.
\end{example}

By comparing and contrasting the case $p=3$ from Example~\ref{ex:3d4d}  with the case $p=3/2$ from the following Example~\ref{ex:3d4d_3halves}, we will see how the complexity of the error bounds changes as dimension increases from $n<3$ to $n \geq 3$.

\begin{example}[$p=3/2$ cones in $3$ and $4$ dimensions]\label{ex:3d4d_3halves}
	We next consider the case of $\stdCone_{3/2}^{3+1}$. This example is illustrated in Figure~\ref{fig:cone_slices_3halves}. Most faces have best exponent $\alpha_z = \frac12$. On the set $\stdCone_{3/2}^{3+1} \cap \color{forest}\{x \;|\; x_0=1 \}$ the $6$ ``corners''---i.e. the $6$ points on the visible surface with exactly $1$ nonzero coordinate (other than the $0$th coordinate)---belong to the $6$ exceptional faces that satisfy $|J_z| = 1$. For these exceptional faces, $\alpha_z = \frac23$. This is in contradistinction with the example $p=3$ that we considered in Example~\ref{ex:3d4d}. In that example, most faces of the 4D cone $\stdCone_{3}^{3+1}$ still have $\alpha_z = \frac12$. However, the exceptional faces, those that have $\alpha_z = \frac13$, are the faces that intersect $\stdCone_{3}^{3+1} \cap \color{forest}\{x\;|\; x_0=1\}$ in the set $\color{wheat}\{x\;|\;x_1=0 \} \color{black} \cup \color{cyan}\{x\;|\;x_2=0 \} \color{black} \cup \color{magenta}\{x\;|\;x_3=0 \}$. Notice that there are \textit{infinitely many} of them. For $\stdCone_{3}^{3+1}$, all of the additional exceptional faces (those not shared in common with $\stdCone_{3/2}^{3+1}$) have exponents that fall under the ``otherwise'' case in \eqref{p-alpha}.
	{For $n < 3$, it can never hold that $1 < |J_z| < n$, and so an analogous difference in the number of exceptional faces does not occur in the lower dimensional $n = 2$ cases.} For both cones $\stdCone_{3}^{2+1}$ and $\stdCone_{3/2}^{2+1}$, the exceptional faces---those without $\alpha_z = \frac12$---are exactly the $4$ faces where the cones intersect with $\color{wheat}\{x\;|\;x_1=0\}$ or $\color{cyan}\{x\;|\;x_2=0\}$.
	
	Now we turn our attention to the face $\stdFace$ similarly defined as $\stdCone_{3/2}^{3+1}\cap \color{aqua}\{x\;|\;x_1=x_2\}$. Another difference between the $p=3$ case and $p=3/2$ case is as follows. For the case $p=3/2$, the need for the smaller exponent (now $\alpha_z=\frac12$) is more apparent when we view $\stdFace$ from the perspective of the 3-dimensional slice $\stdCone_{3/2}^{3+1} \cap \color{magenta} \{x\;|\;x_3=0 \}$. Viewed from this perspective, the cone appears to have less curvature---or ``more flatness''---at the site of the face, which would indicate that $\stdCone_{3/2}^{3+1}$ intersects with the exposing hyperplane $\color{darkbrown}\{z\}^\perp \color{black}$ ``more sharply'' than it would appear to when viewed from the perspective of the slice $\stdCone_{3/2}^{3+1} \cap \color{aqua} \{x\;|\;x_1=x_2 \} \color{black}$. In the case $p=3$, the role of the slices $\stdCone_{3}^{3+1} \cap \color{magenta} \{x\;|\;x_3=0 \}$ and $\stdCone_{3}^{3+1} \cap \color{aqua} \{x\;|\;x_1=x_2 \} \color{black}$ was exactly the reverse of this, in the sense that the \textit{former} deceptively suggested more curvature at the face and the \textit{latter} revealed the true situation.
	This role reversal for $p = 3/2$ is similarly apparent in the slices $ \color{magenta}\{x\;|\;x_3=0\}\color{black}\cap \color{darkbrown}\{z\}^\perp \color{black}\cap\color{forest} \{x\;|\;x_0=1 \}$ and $\color{aqua}\{x\;|\;x_1=x_2\}\color{black}\cap \color{darkbrown}\{z\}^\perp \color{black}\cap\color{forest} \{x\;|\;x_0=1 \}$, where the need for the smaller exponent (now $1/2$ instead of $2/3$) is now apparent in the former instead of the latter.
\end{example}

     \color{black}

\begin{figure}[p]
	\begin{center}
		\begin{tikzpicture}
			[scale=0.76]
			\node[anchor=south west,inner sep=0] (image) at (0,0) {\includegraphics[height=\textheight]{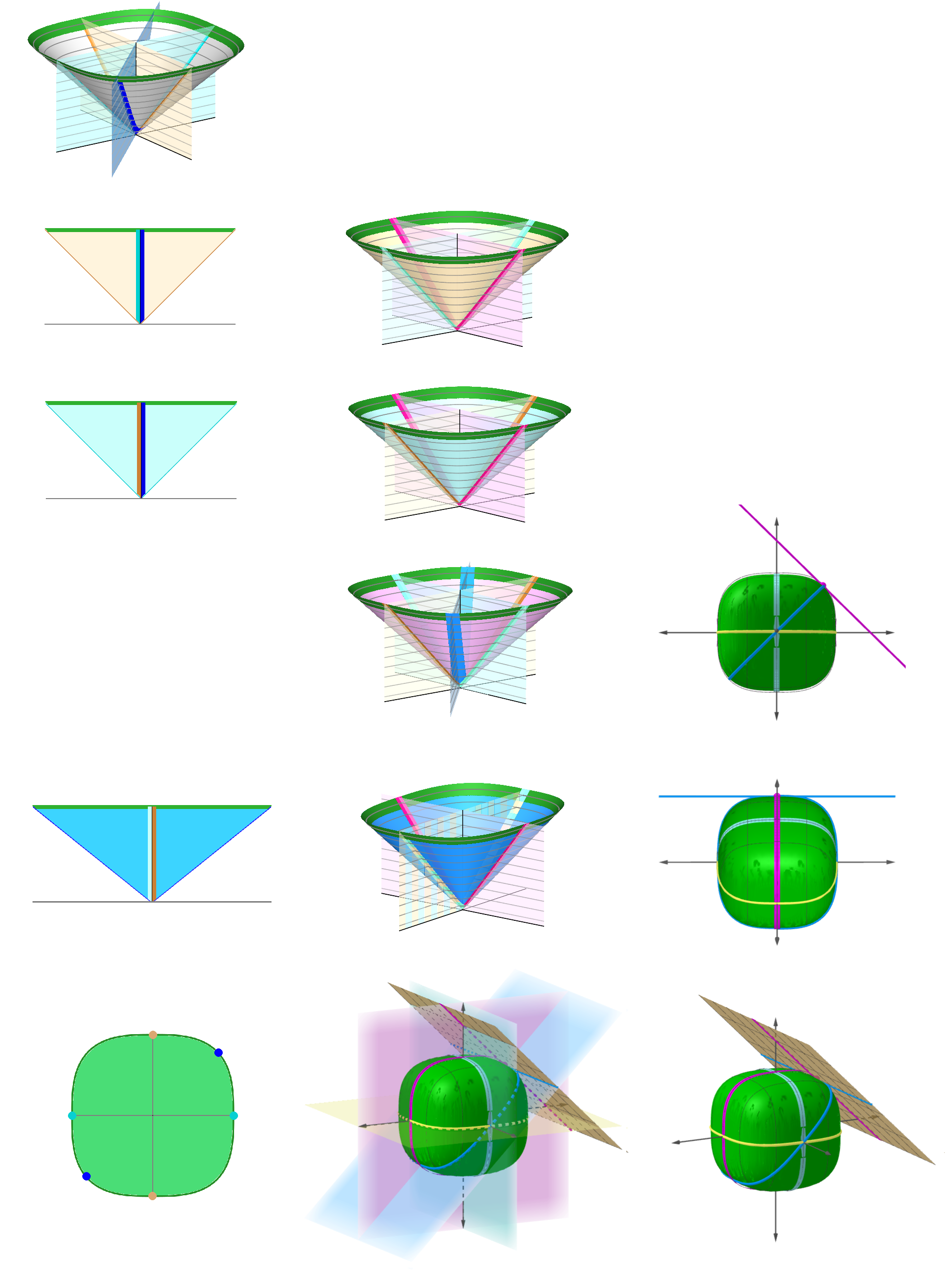}};
			\begin{scope}[x={(image.south east)},y={(image.north west)}]

				\definecolor{darkgray}{rgb}{0.2,0.2,0.2}
				\definecolor{wheat}{rgb}{0.5,0.5,0}
				\definecolor{aqua}{rgb}{0,0.4,.6}
				\definecolor{forest}{rgb}{0,0.4,0}
				\definecolor{darkbrown}{rgb}{0.6,0.2,0.2}
				
				
				\node[darkgray] at (.15,0.865) {$\stdCone_3^{2+1}$};
				\draw[->,blue,thick] (.235,0.92) -- (0.14,0.92);
				\node[right,blue] at (0.235,0.92) {$\stdFace'$};
				
					\node[gray] at (.045,.88) {\scriptsize $x_1$};
					\node[gray] at (.215,.877) {\scriptsize $x_2$};
					
				\node[wheat, below] at (.15,.7275) {$\color{darkgray}\stdCone_3^{2+1} \color{black} \cap \color{wheat} \{x\;|\;x_1=0 \}$};
				
					\node[gray] at (.045,.755) {\scriptsize $x_2$};
				
				\node[cyan, below] at (.15,.5925) {$\color{darkgray}\stdCone_3^{2+1} \color{black}\cap \color{cyan} \{x\;|\;x_2=0 \}$};
				
					\node[gray] at (0.045,0.62) {\scriptsize $x_1$};
				
				\node[black] at (.15,0.50) {$\left(\begin{tabular}{p{3cm}}\text{3-dimensional object}\\\text{does not have}\\ \text{coordinate $x_3$}\end{tabular}\quad \right)$};
				\node[darkgray, below] at (.15,.44) {$\color{darkgray}\stdCone_3^{2+1} \color{black}\cap \color{magenta}\{x\;|\;x_3=0 \}$};
				
				\node[aqua, below] at (.15,.2725) {$\color{darkgray}\stdCone_3^{2+1} \color{black}\cap \color{aqua}\{x\;|\;x_1=x_2 \}$};
				
					\node[gray] at (.05,.305) {\scriptsize line $x_1=x_2$};
				
				\node[forest] at (.15,.01) {$\color{darkgray} \stdCone_3^{2+1} \color{black}\cap \color{forest} \{x\;|\;x_0=1 \}$};
				
					\node[gray] at (.05,0.127) {\scriptsize $x_1$};
					\node[gray] at (.169,0.050) {\scriptsize $x_2$};
				
				
				\node[darkgray] at (.48,0.93) {$\left(\begin{tabular}{p{3cm}}\text{4-dimensional object}\\\text{cannot be plotted}\end{tabular}\quad \right)$};	
				\node[darkgray] at (.48,0.865) {$\stdCone_3^{3+1}$};
				
				\node[wheat, below] at (.48,.7275) {$\color{darkgray}\stdCone_3^{3+1} \color{black}\cap \color{wheat}\{x\;|\;x_1=0 \}$};
				
					\node[gray] at (.39,.735) {\scriptsize $x_3$};
					\node[gray] at (.5605,.735) {\scriptsize $x_2$};
				
				\node[cyan, below] at (.48,.5925) {$\color{darkgray}\stdCone_3^{3+1} \color{black}\cap \color{cyan}\{x\;|\;x_2=0 \}$};
				
					\node[gray] at (.39,.596) {\scriptsize $x_3$};
					\node[gray] at (.5625,.596) {\scriptsize $x_1$};
				
				\node[magenta, below] at (.48,.44) {$\color{darkgray}\stdCone_3^{3+1} \color{black}\cap \color{magenta}\{x\;|\;x_3=0 \}$};
				
				\draw[->,aqua,thick] (0.585,0.49) -- (0.489,0.49);
				\node[right,aqua] at (0.585,0.49) {$\stdFace$};
					
					\node[gray] at (.39,.455) {\scriptsize $x_2$};
					\node[gray] at (.563,.455) {\scriptsize $x_1$};
				
				\node[aqua, below] at (.48,.2725) {$\color{darkgray}\stdCone_3^{3+1} \color{black}\cap \color{aqua}\{x\;|\;x_1=x_2 \}$};
				
					\node[gray] at (.405,.28) {\scriptsize $x_3$};
					\node[gray] at (.605,.28) {\scriptsize $x_1=x_2$};
				
				\draw[->,magenta,thick] (.59,.33) -- (.53,.33);
				\node[right,magenta] at (.59,.33) {$\stdFace$};
				
				\node[forest] at (.48,.01) {$\color{darkgray}\stdCone_3^{3+1} \color{black}\cap \color{forest} \{x\;|\;x_0=1 \}$};
					
					\node[gray] at (.36,.115) {\scriptsize $x_2$};
					\node[gray] at (.485,.23) {\scriptsize $x_1$};
					\node[gray] at (.565,0.105) {\scriptsize $x_3$};
				
				
				\draw (0.845,0.80) node [draw] {\scriptsize $\begin{array}{p{6.1cm}}
					For the 3-dimensional cone	$\stdCone_{3}^{2+1}$:\\[0.1cm]
					\quad \; $z'=\left[ \begin{array}{c}1\\ -\frac{1}{2c} \\-\frac{1}{2c}\end{array} \right]\;\; \text{with}\;\; c= 2^{-\frac13}$\\[0.4cm]
					$\{z'\}^\perp = {\rm span} \left\{ \left[ \begin{array}{c}1\\2c\\0 \end{array} \right],\left[ \begin{array}{c}1\\c\\c\end{array} \right]\right\}$\\[0.4cm]
					$\quad \; \stdFace' = {\rm span}\left \{ \left[ \begin{array}{c}1\\c\\c \end{array} \right] \right \}$\\[0.4cm]
					For the 4-dimensional cone $\stdCone_{3}^{3+1}$:\\[0.1cm]
					\quad \; $z=\left[ \begin{array}{c}1\\ -\frac{1}{2c} \\-\frac{1}{2c} \\ 0 \end{array} \right]\;\; \text{with}\;\; c= 2^{-\frac13}$\\[0.6cm]
					$\{z\}^\perp = {\rm span} \left\{ \left[ \begin{array}{c}1\\2c\\0\\0 \end{array} \right],\left[ \begin{array}{c}1\\0\\2c\\0 \end{array} \right],\left[ \begin{array}{c}1\\c\\c\\1 \end{array} \right]    \right\}$\\[0.6cm]
					$\quad \; \stdFace = {\rm span}\left \{ \left[ \begin{array}{c}1\\c\\c\\0 \end{array} \right] =: f \right \}$
				\end{array}$};
			
				\node[magenta, below] at (0.875,.43) {$\{x\;|\;x_3=0\}\color{black}\cap \color{darkbrown}\{z\}^\perp \color{black}\cap\color{forest} \{x\;|\;x_0=1 \}$};
				\definecolor{deeppink}{rgb}{0.8,0.0,0.6}
				\draw[deeppink,->,thick] (0.92,.44) to [out=60,in=-90] (0.96,0.4575) to [out=90,in=-45] (0.955,0.475);
				
				\draw[->,red] (0.90,0.55) -- (0.875,0.547);
				\node[right,red] at (0.90,0.55) {\small $f$};
				
					\node[gray] at (.68,.51) {\scriptsize $x_2$};
					\node[gray] at (.485+0.34,.60) {\scriptsize $x_1$};

				\node[aqua, below] at (0.875,.26) {$\{x\;|\;x_1=x_2\}\color{black}\cap \color{darkbrown}\{z\}^\perp \color{black}\cap\color{forest} \{x\;|\;x_0=1 \}$};
				\draw[thick,aqua,->] (0.92,.27) to [out=60,in=-90] (0.99,.34) to [out=90,in=0] (0.96,.378);
				
				\draw[->,red] (0.73,0.37) to [out=65,in=180] (0.77,0.39) to [out=0,in=135] (0.805,0.385);
				\node[below,red] at (0.73,0.37) {\small $f$};
				
					\node[gray] at (.68,.33) {\scriptsize $x_3$};
					\node[gray] at (.485+0.35,.39) {\scriptsize $x_1$};
					\node[gray] at (.485+0.35,.265) {\scriptsize $x_2$};
				
				\node[darkbrown] at (.9,0.01) {$\{z\}^\perp \color{black}\cap\color{forest} \{x\;|\; x_0=1\}$};
				
				\draw[thick,darkbrown,->] (0.88,0.03) to [out=60,in=-120] (0.94,.06) to [out=60,in=270] (0.95,0.09);
				
					\node[gray] at (.36+0.33,.11) {\scriptsize $x_2$};
					\node[gray] at (.485+0.33,.215) {\scriptsize $x_1$};
					\node[gray] at (.565+0.33,0.095) {\scriptsize $x_3$};

			\end{scope}
		\end{tikzpicture}
	\end{center}
\caption{Example~\ref{ex:3d4d} is illustrated.}\label{fig:cone_slices}
\end{figure}

\begin{figure}[p]
	\begin{center}
		\begin{tikzpicture}
			[scale=0.76]
			\node[anchor=south west,inner sep=0] (image) at (0,0) {\includegraphics[height=\textheight]{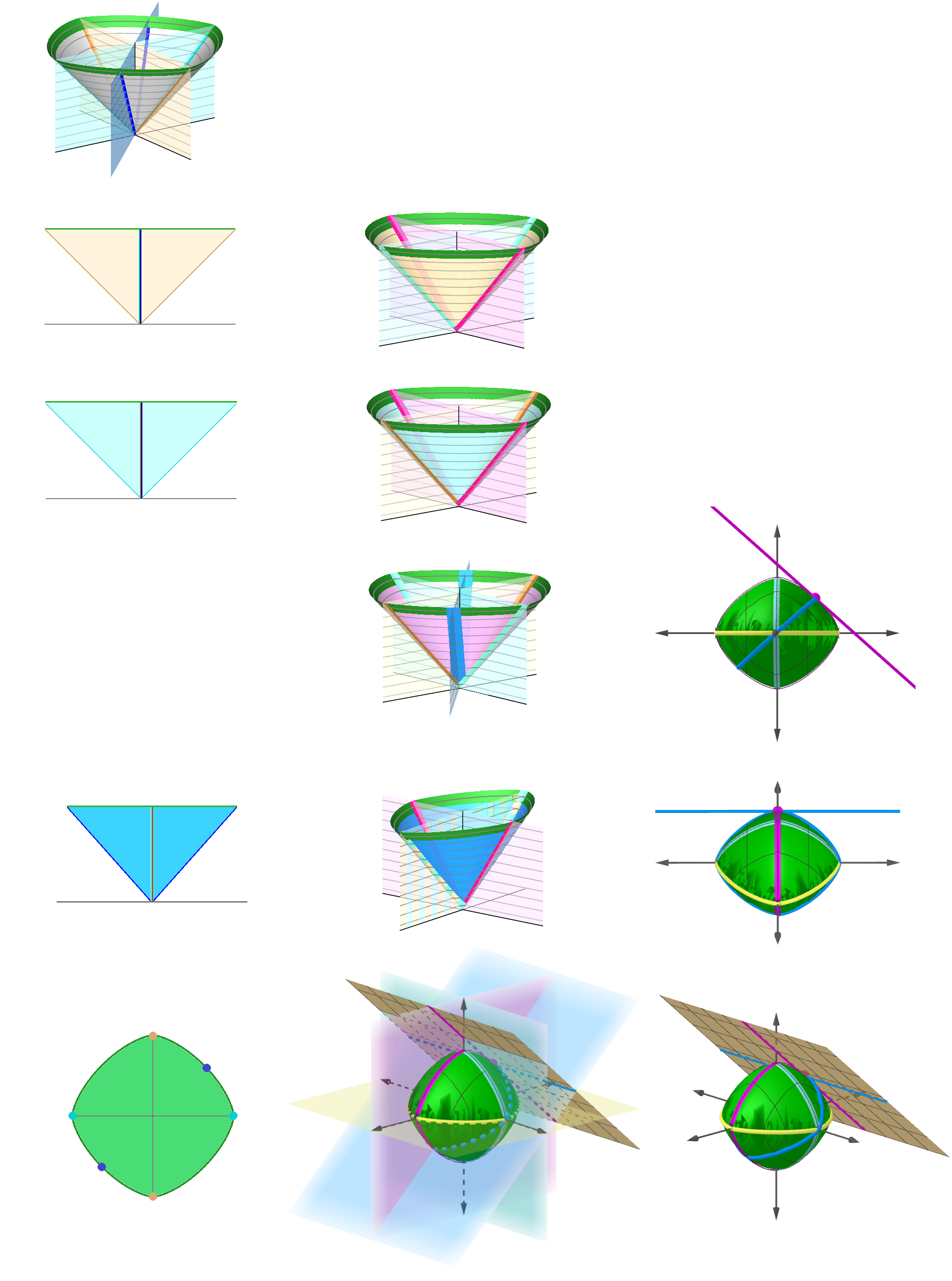}};
			\begin{scope}[x={(image.south east)},y={(image.north west)}]

				\definecolor{darkgray}{rgb}{0.2,0.2,0.2}
				\definecolor{wheat}{rgb}{0.5,0.5,0}
				\definecolor{aqua}{rgb}{0,0.4,.6}
				\definecolor{forest}{rgb}{0,0.4,0}
				\definecolor{darkbrown}{rgb}{0.6,0.2,0.2}
				
				
				\node[darkgray] at (.15,0.865) {$\stdCone_{3/2}^{2+1}$};
				\draw[->,blue,thick] (.235,0.92) -- (0.14,0.92);
				\node[right,blue] at (0.235,0.92) {$\stdFace'$};
				
				\node[gray] at (.045,.88) {\scriptsize $x_1$};
				\node[gray] at (.215,.877) {\scriptsize $x_2$};
				
				\node[wheat, below] at (.15,.7275) {$\color{darkgray}\stdCone_{3/2}^{2+1} \color{black} \cap \color{wheat} \{x\;|\;x_1=0 \}$};
				
				\node[gray] at (.045,.755) {\scriptsize $x_2$};
				
				\node[cyan, below] at (.15,.5925) {$\color{darkgray}\stdCone_{3/2}^{2+1} \color{black}\cap \color{cyan} \{x\;|\;x_2=0 \}$};
				
				\node[gray] at (0.045,0.62) {\scriptsize $x_1$};
				
				\node[black] at (.15,0.50) {$\left(\begin{tabular}{p{3cm}}\text{3-dimensional object}\\\text{does not have}\\ \text{coordinate $x_3$}\end{tabular}\quad \right)$};
				\node[darkgray, below] at (.15,.44) {$\color{darkgray}\stdCone_{3/2}^{2+1} \color{black}\cap \color{magenta}\{x\;|\;x_3=0 \}$};
				
				\node[aqua, below] at (.15,.2725) {$\color{darkgray}\stdCone_{3/2}^{2+1} \color{black}\cap \color{aqua}\{x\;|\;x_1=x_2 \}$};
				
				\node[gray] at (.05,.305) {\scriptsize line $x_1=x_2$};
				
				\node[forest] at (.15,.01) {$\color{darkgray} \stdCone_{3/2}^{2+1} \color{black}\cap \color{forest} \{x\;|\;x_0=1 \}$};
				
				\node[gray] at (.05,0.127) {\scriptsize $x_1$};
				\node[gray] at (.169,0.050) {\scriptsize $x_2$};
				
				
				\node[darkgray] at (.48,0.93) {$\left(\begin{tabular}{p{3cm}}\text{4-dimensional object}\\\text{cannot be plotted}\end{tabular}\quad \right)$};	
				\node[darkgray] at (.48,0.865) {$\stdCone_{3/2}^{3+1}$};
				
				\node[wheat, below] at (.48,.7275) {$\color{darkgray}\stdCone_{3/2}^{3+1} \color{black}\cap \color{wheat}\{x\;|\;x_1=0 \}$};
				
				\node[gray] at (.39,.735) {\scriptsize $x_3$};
				\node[gray] at (.5605,.735) {\scriptsize $x_2$};
				
				\node[cyan, below] at (.48,.5925) {$\color{darkgray}\stdCone_{3/2}^{3+1} \color{black}\cap \color{cyan}\{x\;|\;x_2=0 \}$};
				
				\node[gray] at (.39,.596) {\scriptsize $x_3$};
				\node[gray] at (.5625,.596) {\scriptsize $x_1$};
				
				\node[magenta, below] at (.48,.44) {$\color{darkgray}\stdCone_{3/2}^{3+1} \color{black}\cap \color{magenta}\{x\;|\;x_3=0 \}$};
				
				\draw[->,aqua,thick] (0.585,0.49) -- (0.489,0.49);
				\node[right,aqua] at (0.585,0.49) {$\stdFace$};
				
				\node[gray] at (.39,.455) {\scriptsize $x_2$};
				\node[gray] at (.563,.455) {\scriptsize $x_1$};
				
				\node[aqua, below] at (.48,.2725) {$\color{darkgray}\stdCone_{3/2}^{3+1} \color{black}\cap \color{aqua}\{x\;|\;x_1=x_2 \}$};
				
				\node[gray] at (.405,.28) {\scriptsize $x_3$};
				\node[gray] at (.605,.28) {\scriptsize $x_1=x_2$};
				
				\draw[->,magenta,thick] (.59,.33) -- (.52,.33);
				\node[right,magenta] at (.59,.33) {$\stdFace$};
				
				\node[forest] at (.48,.01) {$\color{darkgray}\stdCone_{3/2}^{3+1} \color{black}\cap \color{forest} \{x\;|\;x_0=1 \}$};
				
				\node[gray] at (.37,.115) {\scriptsize $x_2$};
				\node[gray] at (.485,.23) {\scriptsize $x_1$};
				\node[gray] at (.595,0.115) {\scriptsize $x_3$};
				
				
				\draw (0.845,0.80) node [draw] {\scriptsize $\begin{array}{p{6.1cm}}
						For the 3-dimensional cone	$\stdCone_{3/2}^{2+1}$:\\[0.1cm]
						\quad \; $z'=\left[ \begin{array}{c}1\\ -\frac{1}{2c} \\-\frac{1}{2c}\end{array} \right]\;\; \text{with}\;\; c= 2^{-\frac23}$\\[0.4cm]
						$\{z'\}^\perp = {\rm span} \left\{ \left[ \begin{array}{c}1\\2c\\0 \end{array} \right],\left[ \begin{array}{c}1\\c\\c\end{array} \right]\right\}$\\[0.4cm]
						$\quad \; \stdFace' = {\rm span}\left \{ \left[ \begin{array}{c}1\\c\\c \end{array} \right] \right \}$\\[0.4cm]
						For the 4-dimensional cone $\stdCone_{3/2}^{3+1}$:\\[0.1cm]
						\quad \; $z=\left[ \begin{array}{c}1\\ -\frac{1}{2c} \\-\frac{1}{2c} \\ 0 \end{array} \right]\;\; \text{with}\;\; c= 2^{-\frac23}$\\[0.6cm]
						$\{z\}^\perp = {\rm span} \left\{ \left[ \begin{array}{c}1\\2c\\0\\0 \end{array} \right],\left[ \begin{array}{c}1\\0\\2c\\0 \end{array} \right],\left[ \begin{array}{c}1\\c\\c\\1 \end{array} \right]    \right\}$\\[0.6cm]
						$\quad \; \stdFace = {\rm span}\left \{ \left[ \begin{array}{c}1\\c\\c\\0 \end{array} \right] =: f \right \}$
					\end{array}$};
				
				\node[magenta, below] at (0.875,.43) {$\{x\;|\;x_3=0\}\color{black}\cap \color{darkbrown}\{z\}^\perp \color{black}\cap\color{forest} \{x\;|\;x_0=1 \}$};
				\definecolor{deeppink}{rgb}{0.8,0.0,0.6}
				\draw[deeppink,->,thick] (0.92,.43) to [out=40,in=-100] (0.97,0.450) to [out=100,in=-45] (0.965,0.459);
				
				\draw[->,red] (0.90,0.55) -- (0.868,0.540);
				\node[right,red] at (0.90,0.55) {\small $f$};
				
				\node[gray] at (.68,.51) {\scriptsize $x_2$};
				\node[gray] at (.485+0.34,.60) {\scriptsize $x_1$};

				\node[aqua, below] at (0.875,.26) {$\{x\;|\;x_1=x_2\}\color{black}\cap \color{darkbrown}\{z\}^\perp \color{black}\cap\color{forest} \{x\;|\;x_0=1 \}$};
				\draw[thick,aqua,->] (0.92,.27) to [out=60,in=-90] (0.99,.34) to [out=90,in=0] (0.96,.368);
				
				\draw[->,red] (0.73,0.36) to [out=65,in=180] (0.77,0.38) to [out=0,in=135] (0.805,0.375);
				\node[below,red] at (0.73,0.36) {\small $f$};
				
				\node[gray] at (.68,.33) {\scriptsize $x_3$};
				\node[gray] at (.485+0.35,.39) {\scriptsize $x_1$};
				\node[gray] at (.485+0.35,.2675) {\scriptsize $x_2$};
				
				\node[darkbrown] at (.9,0.01) {$\{z\}^\perp \color{black}\cap\color{forest} \{x\;|\; x_0=1\}$};
				\draw[thick,darkbrown,->] (0.88,0.03) to [out=60,in=-120] (0.94,.06) to [out=60,in=270] (0.95,0.09);
				
				\node[gray] at (.36+0.345,.105) {\scriptsize $x_2$};
				\node[gray] at (.485+0.33,.218) {\scriptsize $x_1$};
				\node[gray] at (.565+0.34,0.098) {\scriptsize $x_3$};

			\end{scope}
		\end{tikzpicture}
	\end{center}
	\caption{Example~\ref{ex:3d4d_3halves} is illustrated.}\label{fig:cone_slices_3halves}
\end{figure}

%

\subsection{Error bounds}\label{sec:p_er}
In this subsection, we gather all the results we have proved so far
and  prove a tight error bound for problems involving a single
$p$-cone.
\begin{theorem}[Error bounds for the $p$-cone and their optimality]\label{thm:main_err}
	Let $n \geq 2$ and $p \in (1,\infty)$.
	Let $\stdSpace \subseteq \RR^{n+1}$ be a subspace and ${a} \in \RR^{n+1}$ such that $(\stdSpace + {a}) \cap \pK \neq \emptyset$.
	Then the following items hold.
	\begin{enumerate}[$(i)$]
		\item\label{thm:maini} If $(\stdSpace + {a}) \cap \pK = \{0\}$ or $(\stdSpace + {a}) \cap (\reInt\pK) \neq \emptyset$, then
		$\pK$ and $\stdSpace+{a}$ satisfy a Lipschitzian error bound.
		\item\label{thm:mainii} Otherwise, $\pK$ and $\stdSpace+{a}$
		satisfy a uniform H\"olderian error bound with exponent $\alpha_z \geq \min\{\frac{1}{2},\frac{1}{p}\}$,
		where $\alpha_z$ is as in \eqref{p-alpha}
		and $z \in \bd\qK \cap \stdSpace^\perp \cap \{a\}^\perp $ with $z \neq 0$ and $1/q + 1/p = 1$.
	\end{enumerate}
Furthermore, the error bound in \ref{thm:mainii} is optimal in the following sense. If $\stdSpace = \{z\}^\perp$, $a = 0$,
then for any  consistent error bound function $\Phi$ for $\pK$ and $\{z\}^\perp$ and any $\hat \eta > 0$, there are constants $\hat \kappa > 0$  and $s_0 > 0$ such that
\[
s^{\alpha_z} \leq \hat \kappa \Phi(\hat \kappa s,\hat{\eta}),\qquad \forall s\in [0,s_0].
\]
In particular, $\pK$ and $\{z\}^\perp$ do not satisfy a uniform H\"olderian error bound with exponent
$\hat \alpha$ satisfying  $\alpha_z < \hat \alpha \leq 1$.

\end{theorem}
\begin{proof}
If $(\stdSpace + {a}) \cap (\reInt\pK) \neq \emptyset$ then Proposition~\ref{prop:pps_er} implies that a Lipschitzian error bound holds.
If we have $\pK \cap (\stdSpace+a) = \{0\}$, a Lipschitzian error bound holds  by \cite[Proposition~27]{L17}. This concludes the proof of item~\ref{thm:maini}.

Next, we move on to item \ref{thm:mainii}.
In this case we have 	$(\stdSpace + {a}) \cap (\reInt\pK) = \emptyset$ and $\pK \cap (\stdSpace+a) \neq \{0\}$.
The $p$-cone for $p \in (1,\infty)$ only has faces of dimension $0$, $1$ or $n+1$.
As such, its distance to polyhedrality $\distP(\pK) = 1$ (see Section~\ref{sec:prel}).

By \eqref{eq:dpp}, we have $\dpp(\pK,\stdSpace+{a})\leq1$.
Since $(\stdSpace + {a}) \cap (\reInt\pK) = \emptyset$, we have $\dpp(\pK,\stdSpace+{a})=1$.
Therefore,
there exists a chain of faces
$\stdFace_2 \subseteq \pK$ satisfying items~\ref{prop:fra_poly:2} and \ref{prop:fra_poly:3} of  Proposition~\ref{prop:fra_poly} together with
$z \in \qK \cap \stdSpace^\perp \cap \{a\}^\perp $ with $1/q +1/p = 1$,
such that $
\stdFace_2 = \pK \cap \{z\}^\perp$.

Since $(\stdSpace + {a}) \cap (\reInt\pK) = \emptyset$, we have $z \neq 0$. Since $\pK \cap (\stdSpace+a) \neq \{0\}$,
we have $\stdFace_2 \neq \{0\}$ (recall that $\stdFace_2$ contains $\pK\cap(\stdSpace+a)$) so that $z \not \in \reInt \qK$ by \eqref{eq:rint_dual}. 
We conclude
that $z \in \bd \qK \setminus \{0\}$ and
\[
\stdFace_2 = \Fhatz,
\]
where $\Fhatz$ is as in \eqref{pconeface}.
Let $\frakg = |\cdot|^{\alpha_z}$ and
let $\psi$ be the {\oneFRFs} given by Corollary~\ref{col:pfrf}.
By Theorem~\ref{thm:pconeface}, we have $\gamma_{z,\eta} \in (0,\infty]$ for every $\eta > 0$, so that $\frakg$ satisfies  Theorem~\ref{thm:1dfacesmain}.
Furthermore, by Theorem~\ref{theo:pcone_opt},
$\frakg$ satisfies the asymptotic optimality criterion \eqref{G1} for $\pK$ and $z$ besides satisfying \eqref{G2} (since $\frakg$ is concave, see Lemma~\ref{lem:concave}).

Now, all assumptions of Theorem~\ref{theo:best} are satisfied. In particular, Corollary~\ref{col:besthold}~\ref{col:bestholdi} tells us that $\stdCone$ and $\stdSpace+a$ satisfy a uniform H\"olderian error bound with exponent $\alpha_z$ as in \eqref{p-alpha}. By definition, $\alpha_z \geq \min\{\frac12,\frac1p\}$ which concludes the proof of item~\ref{thm:mainii}.

The optimality statement follows directly from
item~\ref{col:bestholdii} of Corollary~\ref{col:besthold} by noting that $z$ can be scaled so that $\|z\|=1$.
\end{proof}

The optimal error bound in Theorem~\ref{thm:main_err} inherits its properties, in essence, from those of the optimal {\oneFRFs}s from Theorem~\ref{theo:pcone_opt}. This highlights the importance of the framework we built in Section~\ref{sec:best}. Worthy of additional note is that the optimal error bounds also differ dramatically from the hypothesized form in \cite[Section~5]{L17}.

We conclude this subsection with a result on the direct product of  nonpolyhedral $p$-cones.
\begin{theorem}[Direct products of $p$-cones]\label{thm:pdirect}
Let $\stdCone = \POC{p_1}{n_1+1} \times \cdots\times
\POC{p_s}{n_s+1}$, where $n_i \geq 2$ and
$p_i \in (1,\infty)$ for $i=1,\ldots,s$.

Let $\stdSpace$ be a subspace and ${a}$ be such that $(\stdSpace + {a}) \cap \stdCone \neq \emptyset$.
Then the following hold.	
\begin{enumerate}[$(i)$]
	\item\label{thm:pdirecti} $	\dpp(\stdCone, \stdSpace+a) \leq \min\left\{s,\dim(\stdSpace^\perp \cap \{a\}^\perp) \right\}$.
	\item\label{thm:pdirectii} Let $\alpha = \min\{\frac12,\frac1{p_1},\ldots,\frac1{p_s} \}$ and $d = \dpp(\stdCone, \stdSpace+a)$. Then, $\stdCone$ and $\stdSpace+a$ satisfy a
	uniform H\"olderian error bound with
	exponent $\alpha^d$.
\end{enumerate}

\end{theorem}
\begin{proof}
Each $\POC{p_i}{n_i+1}$ has only three types of faces: $\{0\}$, $\POC{p_i}{n_i+1}$ and extreme rays as in \eqref{pconeface}. Therefore,
the distance to polyhedrality satisfies $\distP(\POC{p_i}{n_i+1}) = 1$
for every $i$. With that, \ref{thm:pdirecti} follows from \eqref{eq:dpp}.	

We move on to item~\ref{thm:pdirectii}. We invoke
Theorem~\ref{theo:err} and let $\stdFace _{\ell}  \subsetneq \cdots \subsetneq \stdFace_1 = \stdCone$ be a chain of faces of $\stdCone$ with
$\ell = \dpp(\stdCone, \stdSpace+a)+1$. At least one such chain exists; see Proposition~\ref{prop:fra_poly} and the subsequent discussion.

First, we need to
determine the {\oneFRFs}s for the faces of $\stdCone$.
Let $\stdFace \face \stdCone$, then
\[
\stdFace = \stdFace^1 \times \cdots\times
\stdFace^s
\]
and each $\stdFace^i$ is a face of $\POC{p_i}{n_i+1}$.
Also, let $z \in \stdFace^*$, so that $z = (z_1,\ldots,z_s)$ where
$z_i \in (\stdFace^i)^*$ for every $i$.
By Proposition~\ref{prop:frf_prod}, {\oneFRFs}s for $\stdFace$ and $z$ can be obtained by positively rescaling the sum $\psi_{\stdFace^1,z_1} + \cdots+ \psi_{\stdFace^s,z_s}$, where each $\psi_{\stdFace^i,z_i}$ is a {\oneFRFs} for $\stdFace^i$ and $z_i$.
By the discussion in Section~\ref{sec:frf}, these
{\oneFRFs}s are of the form $\rho_i(t)\epsilon + \hat \rho_i(t)\epsilon^{\alpha_i}$ where $\rho_i$, $\hat \rho_i$ are nonnegative nondecreasing functions and $\alpha_i$ is either $1$ (see the beginning of Section~\ref{sec:frf}) or $\alpha_{z_i}$ as in Corollary~\ref{col:pfrf}. Since $\alpha \leq \alpha _i$, adjusting
$\rho_i$ and $\hat \rho _i$ if necessary, $\rho_i(t)\epsilon + \hat \rho_i(t)\epsilon^{\alpha}$ is also a {\oneFRF} for $\stdFace^i$ and $z_i$.\footnote{Specifically, if $\rho_i(t)\epsilon + \hat \rho_i(t)\epsilon^{\alpha_i}$ is an FRF  of $\stdFace^i,z_i$ with respect to $\POC{p_i}{n_i+1}$, since $\alpha \leq \alpha _i\le 1$, we have
\[
\rho_i(t)\epsilon + \hat \rho_i(t)\epsilon^{\alpha_i}\le \begin{cases}
  \rho_i(t)\epsilon + \hat \rho_i(t)\epsilon^{\alpha} & {\rm if}\ \epsilon \in [0,1],\\
  \rho_i(t)\epsilon + \hat \rho_i(t)\epsilon & {\rm if}\ \epsilon > 1.
\end{cases}
\]
Then $(\rho_i(t) + \hat \rho_i(t))\epsilon + \hat \rho_i(t)\epsilon^{\alpha}$ is also an FRF.}
Finally, summing $s$ functions of
the form $\rho_i(t)\epsilon + \hat \rho_i(t)\epsilon^{\alpha}$ and positively rescaled shifts  still lead to a function of the same
form $\rho_i(t)\epsilon + \hat \rho_i(t)\epsilon^{\alpha}$.

We conclude that all the {\oneFRFs}s for $\stdCone$ and its faces can be taken to be of the
form $\rho(t)\epsilon + \hat \rho(t)\epsilon^{\alpha}$ for some nonnegative nondecreasing functions $\rho$ and $\hat\rho$.
The result then follows from Lemma~\ref{lem:hold}.\!\!
\end{proof}

\section{Applications}\label{sec:app}
Using our framework for certifying optimal {\oneFRFs}s, we have built optimal error bounds for the $p$-cones. In this final section, we showcase  two applications of these results.
\subsection{Least squares with $p$-norm regularization}\label{sec:least}
In this subsection we consider the following least squares problem with (sum of) $p$-norm regularization:
\begin{equation}\label{eq:l2p}
\textstyle \theta= \min_{x \in \RR^n}\quad g(x) := \frac{1}{2}\norm{Ax-b}^2 + \sum _{i=1}^s \lambda_i\norm{x_i}_p,
\end{equation}
where $A$ is an $m\times n$ matrix, $b \in \RR^m$, $\lambda_i > 0$ for each $i$, and $x$ is partitioned in $s$ blocks so that $x = (x_1,\ldots, x_s)$ with $x_i \in \RR^{n_i}$ for some $n_i \geq 2$, $i = 1,\ldots,s$.
When $p = 2$, problem \eqref{eq:l2p} corresponds to the group LASSO model in statistics for inducing group sparsity \cite{YL06}. The same model can also be used in compressed sensing when the original signal is known to belong to a {\em union} of subspaces; see \cite{EM09}.

Instances of \eqref{eq:l2p} are usually presented in large scale and are solved via various first-order methods such as the proximal gradient algorithm. Here, we are interested in local convergence properties of these methods. Nowadays, it is known that local convergence properties of first-order methods are closely related to the Kurdyka-{\L}ojasiewicz (KL) property (see \cite[Definition~3.1]{ABRS10}) and the associated exponents (see \cite[Definition~2.3]{LP18}) of the underlying optimization models; see, for example, \cite{BDL07,ABRS10,LP18}. Specifically, if $g$ in \eqref{eq:l2p} is a KL function with exponent $\frac12$, then the sequence generated by the proximal gradient algorithm converges locally linearly to a global minimizer; on the other hand, a KL exponent greater than $\frac12$ can only guarantee a sublinear convergence rate.

For the convenience of the readers, we recall the definitions of KL functions and exponents below. We start by introducing some necessary notations. We say that an extended real valued function $h:\RR^n \to [-\infty,\infty]$ is proper if ${\rm dom}\,h:= \{x \mid h(x) < \infty\}\neq \emptyset$ and $h(x) > -\infty$ for all $x\in \RR^n$, and such a function is said to be closed if it is lower semi-continuous. For a proper convex function $h$, its set of subdifferential at $x\in \RR^n$ is defined as $\partial h(x) := \{u \mid h(y) - h(x)\ge \langle u,y-x\rangle\ \ \forall y\in \RR^n\}$, and we define ${\rm dom}\,\partial h := \{x\mid \partial h(x)\neq \emptyset\}$. We are now ready to present the definitions of KL functions and exponents as follows.

\begin{definition}[\!\!{\cite[Definition~2.1]{YLP21}}]
  We say that a proper closed convex function $h:\RR^n\to [-\infty,\infty]$ satisfies the KL property at an $\bar x\in {\rm dom}\,\partial h$ if there exist $a\in (0,\infty]$, $\epsilon > 0$ and a continuous concave function $\psi:[0,a)\to [0,\infty)$ such that
  \begin{enumerate}[{(i)}]
    \item $\psi(0)= 0$ and $\psi$ is continuously differentiable on $(0,a)$ with positive derivatives;
    \item It holds that
    \begin{equation}\label{defpsi}
  \psi'(h(x) - h(\bar x))\dist(0,\partial h(x)) \ge 1
  \end{equation}
  whenever $h(\bar x) < h(x) < h(\bar x) + a$ and $\|x - \bar x\|\le \epsilon$.
  \end{enumerate}
  If $h$ satisfies the KL property at $\bar x\in {\rm dom}\,\partial h$ and the $\psi(t)$ in \eqref{defpsi} can be chosen as $c t^{1-\alpha}$ for some $c > 0$ and $\alpha \in [0,1)$, then $h$ is said to satisfy the KL property at $\bar x$ with exponent $\alpha$.

  A proper closed convex function $h$ is said to be a KL function if it satisfies the KL property at every $x\in {\rm dom}\,\partial h$, and is said to be a KL function with exponent $\alpha\in [0,1)$ if it satisfies the KL property with exponent $\alpha$ at every $x\in {\rm dom}\,\partial h$.
\end{definition}

When $p \in [1,2]$ or $p = \infty$ in \eqref{eq:l2p}, it has been shown in \cite{ZZS15,ZS17} that a certain first-order error bound condition holds for the $g$ in \eqref{eq:l2p}. 
This first-order error bound condition for $g$ turns out to be equivalent to the fact that $g$ is a KL function with exponent $\frac12$; see \cite[Corollary~3.6]{DL18} and \cite[Theorem~5]{BNPS17}. Consequently, we know that the $g$ in \eqref{eq:l2p} is a KL function with exponent $\frac12$ when $p \in [1,2]\cup\{\infty\}$.

On the other hand, in view of \cite[Example~4]{ZZS15}, it is known that the $g$ in \eqref{eq:l2p} is in general {\em not} a KL function with exponent $\frac12$ when $p\in (2,\infty)$. Indeed, it is not even clear whether $g$ is a KL function, not to mention whether it has a KL exponent.
Here, leveraging our error bound results in
Section~\ref{sec:p_er} on direct products of $p$-cones, we will compute explicitly a KL exponent for $g$ in \eqref{eq:l2p} when $p \in (2,\infty)$. The KL exponent can then be used to estimate the convergence rate of the sequence generated when, for example, proximal gradient algorithm is applied to \eqref{eq:l2p} with these $p$ values.

Our analysis starts by observing that \eqref{eq:l2p}, for any $p \in (1,\infty)$, can be equivalently reformulated as
a conic linear program. For that, it will be convenient to consider the \emph{rotated second order cone} which is defined as
\[
\ROC{m+2} = \{ (t,u,x) \in \RR\times \RR \times \RR^m \mid tu \geq \norm{x}^2,\quad t \geq 0, \quad u\geq 0 \}.
\]
Let $T :\RR^{m+2} \to \RR^{m+2}$ be the bijective linear map such that $T(t,u,x) = (t+u,t-u,2x)$.
Then, $T \ROC{m+2} = \POC{2}{m+2}$, i.e.,
$\ROC{m+2}$ and $\POC{2}{m+2}$ are linearly isomorphic cones.
By \cite[Proposition~17]{L17}, linearly isomorphic cones have the same facial residual functions up to positive rescaling. This implies that Theorem~\ref{thm:pdirect} is still valid if a
cone $\POC{p_i}{n_i+1}$ with $p_i = 2$ is replaced by $\ROC{n_i+1}$.

With that, we can write \eqref{eq:l2p} as
\begin{equation}\label{eq:ref}
\begin{array}{rl}
\min\limits_{t,u,w,y,x} & 0.5t + \sum _{i=1}^{s} \lambda_i y_i\\
{\rm s.t.}&  Ax -w = b, \ \ \  u = 1, \\
&  (t,u,w) \in \ROC{m+2}, \ \ \  (y_i,x_i) \in \POC{p}{n_i+1}, \qquad i = 1, \ldots, s.
\end{array}
\end{equation}
The optimal values of \eqref{eq:l2p} and \eqref{eq:ref} are the same (i.e., both are $\theta$) and an optimal solution to the former can be readily used to construct an optimal solution to the latter and vice-versa. For notational convenience, in what follows we write
\begin{equation}\label{eq:vv}
\vv = (t,u,w,(y_1,x_1),\ldots, (y_s,x_s)).
\end{equation}
Then, the optimal set of \eqref{eq:ref} can be written as the intersection of the affine space
\begin{equation}\label{defV}
\textstyle \stdAffine = \{\vv\ |\  0.5t + \sum _{i=1}^{s} \lambda_i y_i = \theta, u = 1, Ax - w = b \}
\end{equation}
with the
cone
\begin{equation}\label{defK}
\stdCone = \ROC{m+2}\times \POC{p}{n_1+1} \times \cdots \times \POC{p}{n_s+1}.
\end{equation}
The feasible region of \eqref{eq:ref} will be denoted by $\feas$, so that
\[
\feas = \{\vv \mid  u = 1, Ax - w = b, \vv \in \stdCone \}.
\]
We can then apply Theorem~\ref{thm:pdirect}, since, as remarked previously, $\ROC{m+2}$ and $\POC{2}{m+2}$ are linearly isomorphic.
Therefore, there exists
$\alpha \in (0,1]$ such that for every bounded set $B$ there exists $\kappa_B > 0$ such that
\begin{equation}\label{eq:aux_err}
\dist(\vv,\stdCone \cap \stdAffine) \leq \kappa_B \max(\dist(\vv,\stdCone),\dist(\vv,\stdAffine) )^\alpha,\qquad \forall \vv\in B,
\end{equation}
and we will discuss the value of $\alpha$ later.
Because $\stdAffine$ is an affine set,
it follows from Hoffman's lemma \cite{HF52} that
there exists a constant $\kappa_{\stdAffine} > 0$ such that if  $\vv  $ (as in \eqref{eq:vv}) satisfies $u = 1$ and $Ax - w = b$, we have
\[
\textstyle \dist(\vv, \stdAffine) \leq\kappa_{\stdAffine} \left| 0.5t + \sum _{i=1}^{s} \lambda_i y_i - \theta \right|.
\]
Plugging this in \eqref{eq:aux_err}, we obtain
\begin{equation}\label{eq:aux_err2}
\textstyle \dist(\vv,\stdCone \cap \stdAffine) \leq \kappa\left| 0.5t + \sum _{i=1}^{s} \lambda_i y_i - \theta \right| ^\alpha,\qquad \forall \vv\in B\cap \feas,
\end{equation}
for some constant $\kappa > 0$.
Next, denoting by $\delta_{\feas}$ the indicator function of $\feas$, we define
\[
\textstyle G(\vv) := 0.5t + \sum _{i=1}^{s} \lambda_i y_i - \theta + \delta_{\feas}(\vv),
\]
so that $G$ is a proper convex lower semicontinuous function satisfying $\inf_{\vv} G(\vv) = 0$.
Then \eqref{eq:aux_err2} implies the following error bound condition: if $\vv^* \in \argmin G$ and $B$ is any ball centered at $\vv^*$, then there exists $\kappa > 0$ such that
\begin{equation}\label{eq:aux_err3}
\dist(\vv,\argmin G) \leq \kappa G(\vv) ^\alpha,\qquad \forall \vv\in B\cap \feas.
\end{equation}
By \cite[Theorem~5]{BNPS17}, this means that $G$ satisfies the KL property at $\vv^*$  with exponent $1-\alpha$.
Now, recalling the definition of $\vv$ in \eqref{eq:vv} and writing $x = (x_1,\ldots, x_s)$ and $\zz = (t,u,w,y_1,\ldots, y_s)$, we can see that
\[
g(x) - \theta = \inf_{\zz} G (\vv),
\]
where $g$ is defined in \eqref{eq:l2p}.
Let $Z(x) \coloneqq \argmin_{\zz} G (\vv)$.
Then, $Z(x)$ is nonempty and if $\zz \in Z(x)$, it must be the case that
$\zz = (t,u,w,y_1,\ldots, y_s)$ satisfies
$t = \norm{Ax-b}^2$, $u = 1$,
$t \geq \norm{w}^2$ and $y_i = \norm{x_i}_p$.
This shows that $Z(x)$ is compact.

We have thus fulfilled all the conditions necessary
to invoke \cite[Corollary~3.3]{YLP21} which says that the KL exponent of $G$ gets transferred to $g$. Therefore, if $x^*$ is an optimal solution to
\eqref{eq:l2p}, then $g$ satisfies the KL property with exponent $1-\alpha$ at $x^*$.

The final piece we need is a discussion on the value of $\alpha$. By Theorem~\ref{thm:pdirect},
$\alpha$ can be chosen as $\min\{0.5,1/p\}^{d}$, where
$d = \dpp(\stdCone, \stdAffine)$ and $d \leq s + 1$. So let us take a look at a situation under which we have $d \leq 1$.
By selecting $\bb$, $\cc$ and $\stdMap$ appropriately, we can write \eqref{eq:l2p} and its dual as
\begin{equation}\label{eq:conic}
\underset{\vv}{\min}\{  \inProd{\cc}{\vv} \mid \stdMap \vv = \bb, \vv \in \stdCone \}, \qquad \underset{\yy}{\max}\{  \inProd{\bb}{\yy} \mid \cc - \stdMap^T \yy \in \stdCone^*\}.
\end{equation}
Because of the format of \eqref{eq:ref}, Slater's condition is satisfied, since one can take
$y_i$ and $t$ large enough so that $\vv$ is feasible and $\vv \in \reInt \stdCone$. Therefore,
the corresponding dual problem has an optimal solution $\yy^*$ attaining the same optimal value $\theta$.  Let $\sbf^* = \cc - \stdMap^T \yy^*$,
so that $\sbf^* \in \stdCone^*\cap\stdAffine^{\perp}$ holds because
$\yy^*$ is dual optimal. Then $\inProd{\vv^*}{\sbf^*} = \inProd{\cc}{\vv^*} - \inProd{\stdMap\vv^*}{\yy^*} = 0$ whenever $\vv^*$ is a primal optimal solution.
Thus, $\stdFace \coloneqq \stdCone \cap \{\sbf^*\}^\perp$ defines a face of $\stdCone$ containing the optimal set of \eqref{eq:ref}. In particular,
if the following strict complementarity-like condition holds for some optimal solution $\vv^*$,
\begin{equation}\label{eq:sc}
\vv^* \in \reInt (\stdCone \cap \{\sbf^*\}^\perp),
\end{equation}
then $\stdAffine \cap \stdFace = \stdAffine \cap \stdCone$ and $\stdAffine \cap (\reInt \stdFace) \neq \emptyset$ holds, so that $\dpp(\stdCone, \stdAffine) \leq 1$.
We have thus proved the following result.

\begin{theorem}
Let $x^*$ be an optimal solution to \eqref{eq:l2p}. Then $g$ satisfies the KL property
at $x^*$ with exponent $1- \alpha$, where
$\alpha = \min\{0.5,1/p\}^{d}$ and $d = \dpp(\stdCone, \stdAffine)$, with $\stdAffine$ and $\stdCone$ given in \eqref{defV} and \eqref{defK} respectively.
Furthermore, $d \leq s + 1$, and if \eqref{eq:ref} satisfies strict complementarity (see \eqref{eq:sc}), then $d \leq 1$.
 \end{theorem}


\subsection{Self-duality and homogeneity of $p$-cones}\label{sec:self}

The second-order cone $\POC{2}{n+1}$ is quite special since it is \emph{symmetric}, that is self-dual and homogeneous. Self-duality means that $(\POC{2}{n+1})^* = \POC{2}{n+1}$ and homogeneity means that for every $x,y \in \reInt \POC{2}{n+1}$, there exists a linear map $A$ such that $Ax = y$ and $A\POC{2}{n+1} = \POC{2}{n+1}$.
Symmetric cones have many nice properties coming from the theory of Jordan algebras \cite{FK94,FB08}. 

A basic question then is the following: are all $p$-cones symmetric? At first glance, the answer might seem \emph{obviously no}; however, this is a subtle question, and the path to its solution is rife with tempting pitfalls. For example, a common source of confusion is as follows: in order to disprove that a cone is symmetric, \emph{it is not enough} to show that $\stdCone^* \neq \stdCone$. The reason is that the self-duality requirement, in the Jordan algebra context, can be met by \emph{arbitrary} inner products, and $\stdCone^*$ changes if $\inProd{\cdot}{\cdot}$ varies.
An interesting discussion on symmetrizing a cone by changing the inner product can be seen in \cite{Or20}.
In fact, the existence of an inner product making a cone $\stdCone$ self-dual is equivalent to the existence of a positive definite symmetric matrix $Q$ such that $Q\stdCone = \stdCone^*$, where $\stdCone^*$ is the dual cone obtained under the usual Euclidean inner product.

In what follows, we let $p \in (1,\infty)$, $p \neq 2$, $n \geq 2$ and $q$ be such that $1/p + 1/q = 1$. 
In order to prove that a $p$-cone
is not a symmetric cone via the self-duality route, what is actually required is to show
that $Q\POC{p}{n+1} = \POC{q}{n+1}$ never holds for any positive definite symmetric matrix $Q$.
It might be fair to say that this is harder than merely
showing that $\POC{p}{n+1} \neq \POC{q}{n+1}$.
One might then try to focus on the homogeneity requirement instead, but this is also a nontrivial task. In fact, the homogeneity of general $p$-cones was one of the open questions mentioned by Gowda and Trott in \cite{GT14}.

These issues were later settled in \cite{IL17_2,IL19} using techniques such as T-algebras \cite{V63} and tools borrowed from differential geometry.
In this final subsection, we show ``easy'' proofs for the questions above based on our error bound results.
The only preliminary fact we need is that if $A$ is a matrix, then $A\POC{1}{n+1} =  \POC{1}{n+1}$ if and only if
\begin{equation}\label{eq:perm}
A = \alpha \begin{bmatrix}
1 & 0\\
0 & D
\end{bmatrix},
\end{equation}
for some $\alpha > 0$ and generalized permutation matrix $D$ (i.e., $\pm 1$ are allowed in the entries of $D$); see \cite[Theorem~7]{GT14}. Here,
$\POC{1}{n+1} := \{(x_0,\ttx) \in \RR^{n+1}\mid x_0 \geq \|\bar x\|_1\}$.


\begin{theorem}[\!\!{\cite[Theorem~11, Corollaries~13 and 14]{IL19}}]
Suppose that $p \in (1,\infty)$ and $n \geq 2$. Then, the following items hold.
\begin{enumerate}[$(i)$]
	\item\label{5p2i} If $\hat p > p$, then there is no matrix $A$ such that $A\POC{p}{n+1} = \POC{\hat p}{n+1}$.
	\item\label{5p2ii} If $p \neq 2$ and $A\POC{p}{n+1} = \POC{p}{n+1}$ holds for some matrix $A$, then $A\POC{1}{n+1} = \POC{1}{n+1}$; in particular, the matrix $A$ must be as in \eqref{eq:perm}.
	\item\label{5p2iii} If $p \neq 2$, then $\POC{p}{n+1}$ is neither self-dual nor homogeneous.
\end{enumerate}
\end{theorem}
\begin{proof}
Before we proceed we make some general observations. Suppose that $\hat p \geq p$.
and there is a matrix $A$ such that $A\POC{p}{n+1} = \POC{\hat p}{n+1}$.
Let $z \in (\POC{p}{n+1})^*$ be such that
$\stdFace_1 \coloneqq \POC{p}{n+1}\cap \{z\}^\perp$ is an arbitrary extreme ray.
Since $A$ is invertible, $A\stdFace_1$ must be an an extreme ray of $\POC{\hat p}{n+1}$ satisfying
\[
A\stdFace_1 = \POC{\hat p}{n+1} \cap \{\hat z\}^\perp,
\]
where $\hat z \coloneqq A^{-T}(z)$.
Therefore, if  $\psi$ is a {\oneFRFs} for $\POC{p}{n+1}$ and $z$, then $A$ must map $\stdFace_1$ onto a face $A\stdFace_1 \face \POC{\hat p}{n+1}$ such that a positive rescaling of $\psi$ is a {\oneFRFs} for $\POC{\hat p}{n+1}$ and $\hat z = A^{-T}(z)$; see \cite[Proposition~17]{L17}.
Conversely, since $A^{-1}\POC{\hat p}{n+1} = \POC{p}{n+1}$, if $\psi$ is a {\oneFRF} for $\POC{\hat p}{n+1}$ and $A^{-T}(z)$, then a positive rescaling of $\psi$ is a {\oneFRFs} for $\POC{p}{n+1}$ and $z$.

The {\oneFRFs}s for $\pK$ and $z$ we constructed  are built from functions that satisfy the optimality criterion \eqref{G1}; see Theorem~\ref{theo:pcone_opt}. This means the exponents appearing in Corollary~\ref{col:pfrf} are the largest possible, which follows from Theorem~\ref{thm:bestFRF} and an argument similar to the proof of item~\ref{col:bestholdii} of Corollary~\ref{col:besthold}.
So, in what follows, we will refer to those exponents appearing in a {\oneFRFs} for $\pK$ and
$z$ constructed from Corollary~\ref{col:pfrf} as the \emph{best exponent of $\stdFace_1$}.

Since positive rescaling does not alter the exponents, the argument we just outlined implies that $A$ must map an extreme ray of $\pK$ with best exponent $\alpha$ to an extreme ray of $\POC{\hat p}{n+1}$ having the same best exponent.

With that in mind, first, we prove item \ref{5p2i}. As discussed above, the only possibility of having $A\POC{p}{n+1} = \POC{\hat p}{n+1}$ is if $\POC{p}{n+1}$ and $\POC{\hat p}{n+1}$ have extreme rays with the same best exponents.

By assumption, we have $\hat p > p$.
If $p = 2$, then all the extreme rays of $\POC{2}{n+1}$ have  best exponent $1/2$.
Since this is not true for $\POC{\hat p}{n+1}$, this case cannot happen.
The case $p \in (1,2)$ is also impossible because the largest best exponent appearing in an extreme ray of  $\POC{\hat p}{n+1}$ is $\max\{1/2,1/\hat p\}$ and $1/p > \max\{1/2,1/\hat p\}$.

Finally, suppose that $p > 2$. Then, there is an extreme ray  $\stdFace_1 \face \POC{p}{n+1}$ with best exponent $1/p$. However, the best exponents for the extreme rays of $\POC{\hat p}{n+1}$ are
$1/\hat p$ and $1/2$, so this case cannot happen.
This concludes the proof of item \ref{5p2i}.

Next, we move on to item \ref{5p2ii}.
Suppose first that $p \in (1,2)$.
Let $\bar{e}_i$ denote the $i$-th unit vector in
$\RR^{n}$. Then, the  two half-lines generated by the vectors $(1,\bar{e}_i)$ and $(1,-\bar{e}_i)$ are extreme rays \color{black} of $\POC{p}{n+1}$ with  best exponent $1/p$, by Corollary~\ref{col:pfrf}. Observing \eqref{p-alpha}, we see that those are the only extreme rays of $\POC{p}{n+1}$ having best exponent $1/p$, and there are $2n$ of them. Since $A\POC{p}{n+1} = \POC{p}{n+1}$ and $A$ must map a face to another face having an identical best exponent, we conclude that $A$ permutes this set of $2n$ extreme rays. However, they are also all the extreme rays of the $1$-cone $\POC{1}{n+1}$, so $A\POC{1}{n+1} = \POC{1}{n+1}$.
In view of the discussion around \eqref{eq:perm}, there exists positive $\alpha > 0$ and a generalized permutation matrix $D$ such that
\[
A = \alpha \begin{bmatrix}
1 & 0\\
0 & D
\end{bmatrix}.
\]
This concludes the case $p \in (1,2)$.
Next, suppose that $p \in (2,\infty)$, then taking duals we have $A^{-T}\POC{q}{n+1} = \POC{q}{n+1}$, where
$1/q + 1/p = 1$, so that $q \in (1,2)$.
Applying what we have shown in the first part for $q \in (1,2)$, we know that  $A^{-T}$ must be a matrix  as in \eqref{eq:perm}.
Then, $A$ is another matrix having the same format and
$A \POC{1}{n+1} = \POC{1}{n+1}$. This completes the proof of item \ref{5p2ii}.

Let $p \neq 2$. As discussed previously, self-duality implies
the existence of a positive definite symmetric matrix
$Q$ such that $Q\POC{p}{n+1} = \POC{q}{n+1}$, which is impossible by item \ref{5p2i}.
Next, we disprove homogeneity. Notice
that $\reInt \POC{1}{n+1}$ is properly contained in  $\reInt \POC{p}{n+1}$, but they do not coincide since
$\POC{1}{n+1} \neq \POC{p}{n+1}$.
So let $x \coloneqq (1,0,\ldots,0)$ and $y$ be a point in
$(\reInt \POC{p}{n+1})\setminus \POC{1}{n+1}$.
If $A\POC{p}{n+1} = \POC{p}{n+1}$ then $A\POC{1}{n+1} = \POC{1}{n+1}$ by item \ref{5p2ii}. Since $x \in \reInt \POC{1}{n+1}$, we have $Ax \neq y$ for all $A$ satisfying $A\POC{p}{n+1} = \POC{p}{n+1}$. This shows that $\POC{p}{n+1}$ is not homogeneous.
\end{proof}


\section*{Acknowledgments}
We thank the referees and the editors for their  comments, which helped to improve the paper.

\bibliographystyle{abbrvurl}
\bibliography{bib_plain}

\end{document}